\documentclass[12pt]{amsart}
\usepackage{amscd,amssymb,graphics,color,a4wide,hyperref,verbatim}

\usepackage{multicol, fullpage}
\usepackage[usenames,dvipsnames]{xcolor} 
\usepackage{tikz, tikz-3dplot, pgfplots}
\usepackage{tkz-graph}
\usepackage{tikz-cd}
\usetikzlibrary[positioning,patterns] 
\usetikzlibrary{matrix,arrows,decorations.pathmorphing}

\usepackage{graphicx}
\usepackage{psfrag}
\usepackage{marvosym}
\usepackage{amsmath}
\usepackage{amsfonts}
\usepackage{amssymb}
\usepackage{amsthm}
\usepackage{mathrsfs}
\usepackage{enumitem}

\input xy
\xyoption{all}

\usepackage{ytableau}

\usepackage{slashed}

\usepackage{comment}

\usepackage{calligra}
\usepackage{mathrsfs}
\DeclareMathAlphabet{\mathcalligra}{T1}{calligra}{m}{n}
\DeclareFontShape{T1}{calligra}{m}{n}{<->s*[1.5]callig15}{}

\newtheorem{theorem}{Theorem}[section]
\newtheorem{lemma}[theorem]{Lemma}
\newtheorem{proposition}[theorem]{Proposition}
\newtheorem{corollary}[theorem]{Corollary}

\theoremstyle{definition}
\newtheorem{definition}[theorem]{Definition}

\newtheorem{example}[theorem]{Example}

\newtheorem{remark}[theorem]{Remark}

\numberwithin{equation}{section}

\newtheorem{thm}{Theorem}[section] 
\theoremstyle{plain} 
\newcommand{\thistheoremname}{}
\newtheorem{genericthm}[thm]{\thistheoremname}
  
\newtheorem*{genericthm*}{\thistheoremname}
\newenvironment{namedthm*}[1]
  {\renewcommand{\thistheoremname}{#1}%
   \begin{genericthm*}}
  {\end{genericthm*}}

\newcommand{\CC} {\mathbb{C}}

\newcommand{\HH} {\mathbb{H}}

\newcommand{\LL} {\mathbb{L}}

\newcommand{\PP} {\mathbb{P}}

\newcommand{\RR} {\mathbb{R}}

\newcommand{\ZZ} {\mathbb{Z}}

\newcommand {\shA} {\mathcal{A}}
\newcommand {\shB} {\mathcal{B}}
\newcommand {\shC} {\mathcal{C}}
\newcommand {\shD} {\mathcal{D}}
\newcommand {\shE} {\mathcal{E}}

\newcommand {\shH} {\mathcal{H}}

\newcommand {\shQ} {\mathcal{Q}}
\newcommand {\shR} {\mathcal{R}}

\newcommand {\shT} {\mathcal{T}}
\newcommand {\shP} {\mathcal{P}}
\newcommand {\shV} {\mathcal{V}}

\newcommand {\sA} {\mathscr{A}}

\newcommand {\sC} {\mathscr{C}}
\newcommand {\sD} {\mathscr{D}}
\newcommand {\sE} {\mathscr{E}}
\newcommand {\sF} {\mathscr{F}}

\newcommand {\sH} {\mathscr{H}}

\newcommand {\sL} {\mathscr{L}}

\newcommand {\sO} {\mathscr{O}}

\newcommand {\sR} {\mathscr{R}}

\newcommand {\sU} {\mathscr{U}}

\newcommand {\sX} {\mathscr{X}}
\newcommand {\sY} {\mathscr{Y}}

\newcommand {\foa}  {\mathfrak{a}}

\newcommand {\foc}  {\mathfrak{c}}

\newcommand {\Coh} {\operatorname{Coh}}
\newcommand {\coh} {\operatorname{coh}}

\newcommand {\cone} {\operatorname{cone}}

\newcommand {\End} {\operatorname{End}}
\newcommand {\Ext} {\operatorname{Ext}}
\newcommand{\sExt}{\mathscr{E} \kern -3pt xt}

\newcommand {\Gr} {\operatorname{Gr}}

\newcommand {\Hom} {\operatorname{Hom}}
\newcommand {\sHom}{\mathscr{H}\kern-5pt\mathcalligra{om}}

\newcommand {\id} {\operatorname{id}}
\newcommand {\Id} {\operatorname{Id}}

\newcommand {\kk} {\Bbbk}

\newcommand {\Ker} {\operatorname{Ker}}

\newcommand {\Pf} {\operatorname{Pf}}

\newcommand {\Proj} {\operatorname{Proj}}

\newcommand {\pt} {\operatorname{pt}}

\newcommand {\rank} {\operatorname{rank}}

\newcommand {\Spec} {\operatorname{Spec}}

\newcommand {\Sym} {\operatorname{Sym}}

\newcommand {\Tor} {\operatorname{Tor}}
\newcommand {\Tot} {\operatorname{Tot}}

\renewcommand {\HH} {\operatorname{HH}}

\title[]{Categorical Pl\"ucker Formula and Homological Projective Duality}
\author[Q.Y.\ JIANG, N.C.\ Leung, Y. \Xie]{Qingyuan Jiang, Naichung Conan Leung, Ying Xie}

\address{The Institute of Mathematical Sciences and Department of Mathematics,
The Chinese University of Hong Kong, Shatin, N.T., Hong Kong}
\email{qyjiang@math.cuhk.edu.hk}

\address{The Institute of Mathematical Sciences and Department of Mathematics,
The Chinese University of Hong Kong, Shatin, N.T., Hong Kong}\email{leung@math.cuhk.edu.hk}

\address{The Institute of Mathematical Sciences and Department of Mathematics,
The Chinese University of Hong Kong, Shatin, N.T., Hong Kong}\email{yxie@math.cuhk.edu.hk}

\begin{document}

\begin{abstract} \noindent Homological Projective duality (HP-duality) theory, introduced by Kuznetsov \cite{Kuz07HPD}, is one of the most powerful frameworks in the homological study of algebraic geometry. The main result (HP-duality theorem) of the theory gives complete descriptions of bounded derived categories of coherent sheaves of (dual) linear sections of HP-dual varieties.
We show the theorem also holds for more general intersections beyond linear sections. More explicitly, for a given HP-dual pair $(X,Y)$, then analogue of HP-duality theorem holds for their intersections with another HP-dual pair $(S,T)$, provided that they intersect properly. 
We also prove a relative version of our main result. Taking $(S,T)$ to be dual linear subspaces (resp. subbundles), our method provides a more direct proof of the original (relative) HP-duality theorem.
\vspace{-2mm}
\end{abstract}

\maketitle

\section{Introduction}

Homological Projective duality (HP-duality for short), as a homological generalization of classical projective duality, is one of the most powerful theories in the homological study of algebraic geometry. Since its introduction by Kuznetsov \cite{Kuz07HPD}, HP-duality theory has become the most important method to produce interesting semiorthogonal decompositions of derived categories of coherent sheaves on projective varieties, as well as relate derived categories of different varieties, cf. \cite{Kuz04, Kuz06HPDGr, Kuz08quadric, Kuz10cubic4, IK10, ABB14, Kuz14SODinAG, RT15HPD, HT15, CZ15, BBF16, HTa16}.

The Homological Projective duality is very closely related to equivalence and dualities in physical theories of Gauge Linear Sigma Models (GLSMs), cf. \cite{HHP08, HT07, DSh08, CDH+10, Sha10, Hor13, HK13}. Mathematically, this is related to the idea of relating the derived category of subvariety $X = Z(s) \subset P$ of a regular section $s \in \Gamma(P, \shV)$ of a vector bundle $\shV$ to the total space $\Tot \shV$ coupled with the section $s$ (the mathematical formulation of 'Landau-Ginzburg (LG) models'), cf. \cite{Isi12, Shi12, ADS15}, and when $\shV$ is given by Geometric Invariant Theory (GIT) quotients, the study of how derived categories behave under variation of GIT quotients (VGIT), cf. \cite{Kaw02, Vdb04, Seg11, BFK12, HW12, DS12}.
This thread of ideas has been a very powerful approach of constructing HP-duals as well as finding applications of HP-duality, cf. \cite{BDF+, BDF+14, H-L15, ADS15, ST14, Ren15, ReEd16}.


The main result of the theory, HP-duality theorem (Kuznetsov \cite[Thm. 6.3]{Kuz07HPD}) states if a projective variety $X$ with a morphism $X \to \PP V$, where $V$ is a vector space of dimension $N$, has a Homological Projective dual (HP-dual for short) $Y \to \PP V^*$ with respect to a semiorthogonal decomposition of the form (called Lefschetz decomposition, cf. \S \ref{sec:lef})
	$$D(X) = \langle \shA_0, \shA_1(1), \ldots, \shA_{i-1}(i-1) \rangle, \quad \shA_0 \supset \shA_1 \supset \ldots \supset \shA_{i-1},$$
where $(k)$ means $\otimes \sO_X(k)$. Then the bounded derived categories of coherent sheaves on its HP-dual $Y$ has a decomposition obtained by 'complementary boxes' of components of $D(X)$:
	$$D(Y) = \langle  \shB^1(2-N), \shB^2(3-N), \ldots, \shB^{N-1} \rangle, \quad \shB^1 \subset \shB^2 \subset \ldots \subset \shB^{N-1},$$
where $\shB^k$ is defined in by (\ref{def:B^k}), see the first diagram of Figure \ref{Figure:lf_2}. Moreover, not only we have a semiorthogonal decomposition for any complete
linear section $X_{L^\perp}: =X \times_{\PP V } L^\perp$ 
	$$D(X_{L^\perp})  = \langle \sC_{X_{L^\perp}}, \shA_l(l),  \ldots, \shA_{i-1}(i-1) \rangle,$$
where $\shA_k(k)$ terms are 'ambient parts', i.e. restrictions from $X$ to $X_{L^\perp}\subset X$, and $L^\perp \subset \PP V$ is a projective linear subspace, and a similar decomposition of $D(Y_L)$ :
	$$D(Y_L)  = \langle \shB^1(2-l), \ldots, \shB^{l-2}(-1), \shB^{l-1}, \sC_{Y_{L}} \rangle,$$
where $L \subset \PP V$ is the dual projective linear subspace of $L$, and $Y_L: =Y \times_{\PP V^*} L$ is dual complete linear intersection of $Y$,
but also a derived equivalence $\sC_{X_{L^\perp}} \simeq \sC_{Y_{L}}$ of 'primitive parts' of $D(X_{L^\perp})$ and $D(Y_L)$. See Thm. \ref{thm:HPD} for the precise statement.
 
As dual linear subspaces are first examples of HP-dual pairs, it is natural to ask:
 
\medskip\noindent\textbf{Question.} How much can we generalize these results by replacing $(L,L^\perp)$ by some other classes of pairs of HP-dual spaces $(S,T)$ beyond linear sections?

\medskip The answer turns out to be: we can replace $(L,L^\perp)$ by \emph{any other pair} of HP-dual spaces $(S,T)$, provided they intersect properly. Therefore the framework of HP-duality theory is even more powerful as we already witnessed, and the HP-duality theorem can take a more general and stronger form as follows: assume we have another HP-dual pair $(S \to \PP V^*,T \to \PP V)$, 
with the following decompositions:
	\begin{align*}
	D(S) & = \langle\shC_0, \shC_1(1), \ldots, \shC_{l-1}(l-1) \rangle,\quad \shC_0 \supset \shC_1 \ldots \supset \shC_{l-1}.\\
	D(T) & = \langle \shD^1(2-N), \shD^2(3-N), \ldots, \shD^{N-1} \rangle,\quad \shD^1 \subset \shD^2 \subset \ldots \shD^{N-1},
	\end{align*}
where $\shD^k$'s are the 'complementary boxes' of $\shC_k$'s defined by (\ref{def:D^k}) (also see Fig. \ref{Figure:lf_2}). Then
\begin{namedthm*}{Main theorem}[HP-duality theorem for general intersections; Thm. \ref{thm:HPDgen}]\label{thm:intro} If the two above HP-dual pairs $(X,Y)$ and $(S,T)$ are admissible (Def. \ref{def:admissible}), then we have semiorthogonal decompositions for the fiber products $X_T := X \times_{\PP V} T$ and $Y_S : = Y \times_{\PP V^*} S$ into 'ambient' and 'primitive' parts:
	\begin{align*} D(X_T) & = \Big\langle \sE_{X_T}, \,\big\langle (\shA_k \boxtimes \shD^k) \otimes \sO_{X_T}(k) \big\rangle_{k=1, \ldots, i-1}\Big\rangle, \\
		D(Y_S) & = \Big\langle  \big\langle  (\shB^k \boxtimes \shC_k) \otimes \sO_{Y_S}(1-l+k) \big\rangle_{k=1, \ldots, l-1} \, ,\sE_{Y_S} \Big\rangle,
	\end{align*}
and a derived equivalence of the 'primitive' parts 
	$$\sE_{X_T} \simeq \sE_{Y_S}.$$
\end{namedthm*}

Our strategy is based playing the 'game on chessboard' introduced by Kuznetsov and Thomas \footnote{Cf. \cite{RT15HPD}. We were told by Richard Thomas that the 'chess game' was initially invented by Kuznetsov. The computing patterns for proving fully-faithfulness part already appeared in Kuznetsov's \cite{Kuz07HPD}.}. The most subtle part of this strategy is the proof of generation (cf. \S \ref{sec:generation}). Our proof does not rely on the statement of the original HP-duality theorem \cite[Thm. 6.3]{Kuz07HPD}.
In fact, our method in the case $(S,T)=(\PP V^*,\emptyset)$ can be viewed as answering the question of Kuznetsov (posed after Prop. 5.10 in \cite{Kuz07HPD}) of 'finding a direct proof' of the fact that $\shB^1(2-N), \ldots, \shB^{N-1}$ generate $\sC \simeq D(Y )$,  see Rmk. \ref{rmk:dual-decomposition}, and in the case $(S,T)= (L,L^\perp)$, also gives a 'direct proof' of the original HP-duality theorem, cf. Rmk. \ref{rmk:KuznetsovHPD}.

The admissible condition (Def. \ref{def:admissible}) in the theorem is a a technical way of saying the two pairs $X$ and $T$, resp. $Y$ and $S$ intersect 'dimensionally transversely' (Lem. \ref{lem:admissible pairs} for criteria), and we expect these conditions which guarantee certain 'flatness' can be removed if we consider derived intersections $X\times^R_{\PP V} T$ and $Y \times^R_{\PP V^*} S$ and the corresponding $dg$-categories.

The 'ambient' parts are given by restrictions of $\shA_k \boxtimes \shD^k \subset D(X \times T)$ (resp. $\shB^k \boxtimes \shC_k \subset D(Y \times S)$) from product space to $X \times_{\PP V} T \subset X \times T$ (resp. $ Y \times_{\PP V^*} S \subset Y \times S$). The terms appearing in the 'ambient' parts are illustrated in Figure \ref{Figure:lf_2}.
The derived equivalence of the 'primitive' parts $\sE_{X_T} \simeq \sE_{Y_S}$ are given explicitly by Fourier-Mukai functors in Rmk. \ref{rmk:FM}. The statements of the theorem is at least twofold:

\begin{enumerate}[leftmargin=*, wide, align=left]
\item \textit{Decomposition part} states that for the above HP-dual pairs, the derived categories of fibre products $X_T$ and $Y_S$ are decomposed into 'ambient' part and 'primitive' part, with 'ambient' parts coming from ambient spaces $X \times T$ and $Y\times S$ respectively. This shares the spirit of \emph{Lefschetz decomposition} for linear sections in topology.
Decomposition result itself can be \emph{improved}: the decomposition for $X_T$ holds for any $f: X \to \PP V$ with a Lefschetz decomposition, as long as the pair $(S,T)$ is geometric (e.g. linear sections or quadric sections) and $X_T$ satisfies a quite weak dimensional condition, cf. Cor. \ref{cor:HPDdec}; Starting with $f: X \to \PP V$ and $q: S \to \PP V^*$ with Lefschetz decompositions, then there are always decompositions on categorical levels (i.e. for the sections of the respective HP-dual categories), cf. Prop. \ref{prop:HPDcat}. 
\item \textit{Comparison part} states the 'primitive' parts, which do not come from ambient spaces, of two decompositions are derived equivalent. This shares the spirit of \emph{topological Pl\"ucker formula}, which compares the Euler characteristics of intersections of classical projective dual varieties. The relation will be explored more later.
\end{enumerate}

Based on the theorem, we also reprove HP-duality is a reflexive relation ({\cite[Thm. 7.3]{Kuz07HPD}}), and show the Fourier-Mukai kernel of dual direction can be given by the same kernel, cf. \S \ref{sec:duality}, Thm. \ref{thm:duality}. As our proof is categorical, it is straightforward to generalize the results to noncommutative settings, see \S \ref{sec:ncHPD}. In \S \ref{sec:examples}, we apply the result to some concrete examples including intersections of quadrics and quadric sections of Pfaffian-Grassmannian dualities. In \S \ref{sec:relative} we will prove the relative version of our main result, Thm. \ref{thm:HPDgen:rel}.

\medskip\noindent \textbf{Strategy of proof.} The strategy of the proof of main theorem is as follows:
\begin{enumerate}[leftmargin=*, wide, align=left]
	\setlength\itemsep{0.5 em}
	\item \textbf{Base change.}  In order to compare the categories $D(Y_S)$ and $D(X_T)$, notice from definition of HP-dual, $D(Y)$ (resp. $D(T)$) naturally embeds into $D(\shH_X)$ (resp. $D(\shH_S)$). If we base-change $D(\shH_X)$ along $S \to \PP V^*$ and base-change $D(\shH_S)$ along $X \to \PP V$, we put $D(Y_S)$ and $D(X_T)$ into a common ambient category $D(\shH_{X,S}) = D(\shH_{S,X})$. Cf. \S\ref{sec:base-change}. 
	\item \textbf{Game on the 'chessboard'.} After the base-change step, the proof of the theorem is reduced to playing the 'chess game' on the 'chessboard' (see Figure \ref{Fig_chessboard}, \ref{Figure:l<=i} and \ref{Figure:l>=i}) introduced by Kuznetsov and Thomas. The process is further separated into:
	\begin{enumerate}[leftmargin=*, wide, align=left]

		\item[$(i)$] \emph{Fully-faithfulness.} We show the desired 'ambient' parts embed into $D(X_T)$ and $D(Y_S)$ respectively. Also the 'primitive' part of $D(X_T)$ embeds into the 'primitive' part of $D(Y_S)$. Since the projection functors are mutation functors passing through certain regions in the 'chessboard', the key is to compute the region of the desired categories after mutations. This is done in \S \ref{sec:fullyfaithful}. 
		\item[$(ii)$]  \emph{Generation.} In order to show the derived equivalence of primitive parts, it remains to show the image generates the desired category. Our strategy is to show the right orthogonal of the image inside $D(Y_S)$ is zero. By adjunction of the mutation functor, this reduces to show certain vanishing conditions for an element $b \in D(Y_S)$ forces $b=0$. This is done following a specific 'Zig-Zag' scheme as described in \S \ref{sec:generation}.
	\end{enumerate}
\end{enumerate}

The underlying reason why this works is, although the derived categories of $X_T$ and $Y_S$ are quite non-linear (in many examples contains Calabi-Yau or fractional Calabi-Yau categories components), their \textit{'difference'} inside $D(\shH)$ is \textit{linear}. All the functors involved in the statement of the theorem (inclusion functors, projection functors as mutations, Serre functors, twist functors by $\sO_X(1)$ and $\sO_S(1)$, etc) can be reflected precisely on the 'chessboards' in Figure \ref{Figure:l<=i} and \ref{Figure:l>=i}. Then everything reduces to playing the 'chess game' on the 'chessboards'.

\medskip \noindent \textbf{Pl\"ucker formula.} As Kuznetsov's HP-duality theorem can be regarded as homological counterpart of classical Lefschetz theory of cohomology of linear sections of projective varieties, our investigation is motivated by the \emph{topological Pl\"ucker formula} for classical dual varieties. Recall if $X$ and $T$ are subvarieties of $\PP V$ which intersect transversely, and same holds for their classical projective dual $X^\vee, T^\vee \subset \PP V^*$. Then
		$$(-1)^*\left(\chi(X\cap T) - \frac{\chi^M(X) \cdot\chi^M(S)}{N} \right)=  \chi(X^\vee \cap T^\vee) - \frac{\chi^M(X^\vee) \cdot\chi^M(T^\vee)}{N},$$
where $* = \dim X + \dim T + \dim X^\vee + \dim T^\vee$, $\chi(-)$ is the topological Euler characteristic, $\chi^M(-)$ is the weighted Euler characteristic $\chi(-, Eu[-])$ with respect to constructible function $Eu([-])$ for singular space, and $Eu$ is MacPherson's Euler obstruction. Cf. \cite{L01, CL15}. 

Taking Hochschild homology of Thm. \ref{thm:HPDgen} and then taking Euler characteristics, we get the corresponding version for HP-duals:
 
\begin{proposition}[Pl\"ucker formula for HP-duals, Cor \ref{cor:plucker}] Assume $X_T$ and $Y_S$ are smooth, then (recall $N = \dim V$)
	$$\chi(X_T)  - \frac{\chi(X) \cdot \chi(T)}{N} = \chi(Y_S) - \frac{\chi(Y)\cdot \chi(S)}{N},$$
where $\chi(-)$ is the ordinary topological Euler characteristic for topological spaces.
\end{proposition}

Hence our theorem Thm. \ref{thm:HPDgen} can be viewed as a 'double categorification' of the topological Pl\"ucker formula. It states the comparison of Euler characteristics for intersections of HP-dual varieties, as analogy of topological ones, actually is a shadow of what happens on categorical levels: (i) the decompositions of derived categories of both intersections into ambient parts and primitive parts, and (ii) the equivalence of the primitive parts.

The topological Pl\"ucker formula comes from the second author's study of categorifications of Lagrangian intersections inside hyperk\"ahler manifolds. It will be very interesting to further explore the relations between our Thm. \ref{thm:HPDgen}, as a categorical Pl\"ucker formula for HP-duals, and the topological ones from Lagrangian intersections.

\medskip\noindent\textbf{Related works.}
While preparing this paper, we learned that Alexander Kuznetsov and Alex Perry also claimed very similar results using different methods, the categorical joins.

\medskip
\noindent \textbf{Acknowledgement.} The authors would like to thank specially Richard Thomas for many helpful communications and discussions. The first and second author would like to thank Andrei C\v{a}ld\v{a}raru for useful discussions and helps during their visit at University of Wisconsin-Madison. The first and third author would like to thank Alex Perry for discussions during their stay at Harvard in May of 2016. The authors like to thank the hospitality of CMSA of Harvard during their two visits, and the members there, especially Hansol Hong for discussions. The authors also appreciate very much the hospitality of YMSC of Tsinghua, and ITP of Chinese Academy of Sciences in Beijing during their visit in the summer of 2016. 
The authors would like to thank their fellow colleagues at IMS for many helpful discussions, especially Mathew Young. The research described in this article was substantially supported by grants from the Research Grants Council of the Hong Kong Special Administrative Region, China (Project No. CUHK14303516 and 14302215).

\medskip
\noindent \textbf{Conventions.} We fix an algebraically closed field $\kk$ of characteristic zero. All varieties considered in this paper are assumed to be embeddable $\kk$-varieties, i.e. $\kk$-varieties admitting finite surjections onto smooth $\kk$-varieties. Morphisms between varieties, fibered products, etc are considered inside the category $\operatorname{Var}_\kk$ of $\kk$-varieties. Products of varieties are fibered products over $\pt$, where $\pt : = \Spec \kk$ denotes the final object of $\operatorname{Var}_\kk$. All categories are assumed to be $\kk$-linear. We fix $V$ to be  a $N$-dimensional $\kk$-vector space, where $N$ is a fixed positive integer, and let $V^*$ be the dual vector space. Denote by $\PP(V):=\Proj \Sym^\bullet V^*$ the projective space of $V$, i.e. the moduli space parametrizing one dimensional subspaces of $V$.

We use $D(X)$ to denote the \textbf{bounded} derived category of coherent sheaves on an algebraic variety $X$, i.e. $D(X) : = D^b (\coh(X))$. For a morphism $f: X \to Y$ between algebraic varieties, we denote by $f_*: D(X) \to D(Y)$ and $f^*: D(Y) \to D(X)$ the derived pushforward and derived pullback funtors (when they are well-defined). We use $\otimes$ for the derived tensor product. For a $\kk$-linear category $\shC$, we denote by $\Hom_{\shC}$ the $\kk$-linear hom spaces inside the category, and by $R\Hom_{\shC}$ the derived hom functor. When $X$ is a variety, we denote by $\Hom_X$ the $\kk$-linear hom spaces inside $D(X)$, by $\sHom_X$ the local sheaf hom functor, and by $R\Hom_X$, resp. $R\sHom_X$ the corresponding derived functors. We may omit the subscript if it is clear from the context. For two objects $A,B \in \shC$, we denote by $0 \in \Hom(A,B)$ the zero element of the $\kk$-vector space $\Hom(A,B)$. We also use $0 \in \shC$ to denote the zero object of the $\kk$-linear category $\shC$.

\section{Preliminaries}

\subsection{Semiorthogonal decompositions} We first recall basic facts about semiorthogonal decompositions of triangulated categories. Standard references are \cite{B}, \cite{BK}. See also \cite{Kuz09HHSOD}.

\begin{definition} \label{def:sod} A \textbf{semiorthogonal decomposition} of the triangulated category $\shT$:
	\begin{equation} \label{SOD:T}
	\shT = \langle \shA_1, \shA_2, \ldots, \shA_{n} \rangle,
	\end{equation}
	 is a sequence of full triangulated subcategories $\shA_1, \ldots, \shA_n$ such that
	\begin{enumerate}
		\item $\Hom_\shT (a_k ,a_l) = 0$ for all $a_k \in \shA_k$ and $a_l \in \shA_l$ \footnote{We will also use $\Hom_{\shT}(\shA_k, \shA_l)=0$ to denote this condition. Notice since $\shA_k$, $\shA_l$ are triangulated subcategories, this condition is equivalent to vanishing of the derived $\Hom$, $R\Hom_{\shT}(\shA_k, \shA_l)$=0.}, if $k > l$.

		\item For any object $T \in \shT$ there is a sequence of morphism:
		\begin{equation*}
		\begin{tikzcd} [back line/.style={dashed}, row sep=1.5 em, column sep=1.5 em]
	0=T_n \ar{rr} 	& 	& T_{n-1} \ar{rr} \ar{ld}		&		& T_{n-2} \ar{r} \ar{ld}	&\cdots \ar{r} 	& T_{1} \ar{rr} 	&	& T_0=T \ar{ld}, \\
							& a_{n} \ar{lu}{[1]} &	& a_{n-1}  \ar{lu}{[1]} & & & & a_1 \ar{lu}{[1]}
		\end{tikzcd}
		\end{equation*}
			such that each cone $a_k = \cone (T_{k} \to T_{k-1}) \in \shA_k$, $k=1,\ldots, n$.
	\end{enumerate}
	The subcategories $\shA_k$ are called \textbf{components} of $\shT$ with respect to (\ref{SOD:T}). A sequence $\shA_1, \ldots, \shA_n$ satisfying the first condition will be called \textbf{semiorthogonal}. We denote by $\langle \shA_1, \ldots, \shA_n \rangle$ the smallest triangulated subcategory containing all these $\shA_k$'s.
\end{definition}

\begin{remark} The first condition implies the objects $T_k \in \shT$ and $a_k \in \shA_k$ are uniquely determined by and functorial on $T$, cf. \cite[Lem. 2.4]{Kuz09HHSOD}. The functors $\shT \to \shA_k$, $T \mapsto a_k$ are called \textbf{projection functors}, and $a_k$ is called the \textbf{component} of $T$ in $\shA_k$ with respect to the decomposition (\ref{SOD:T}). 
\end{remark}

\begin{example} For a projective space $\PP^{l-1}$ with $l \ge 1$ we have {\em Beilinson's decomposition} \cite{Bei}
	\begin{equation}\label{sod:beilinson}
	D(\PP^{l-1}) = \langle \sO_{\PP^{l-1}}, \sO_{\PP^{l-1}}(1), \ldots, \sO_{\PP^{l-1}}(l-1)\rangle,
	\end{equation}
by which we mean $\langle \sO_{\PP^{l-1}}(k)\rangle \simeq D(\pt)$ for $k=0,1,\ldots, l-1$ and the sequence of subcategories $\langle  \sO_{\PP^{l-1}} \rangle , \ldots, \langle  \sO_{\PP^{l-1}}(l-1) \rangle$ gives rise to a semiorthogonal decomposition of $D(\PP^{l-1})$. 
\end{example}

For a triangulated subcategory $\shA$ of $\shT$, we define its \textbf{left} and \textbf{right orthogonals} to be
	\begin{align*}
	{}^\perp \shA &: = \{ T \in \shT ~|~ R\Hom(T, A) = 0, \quad \forall A\in \shA \}, \\
	\shA^{\perp} &:= \{T \in \shT~|~ R\Hom(A,T) = 0, \quad \forall A \in \shA\}.
	\end{align*}
It is clear these are triangulated subcategories of $\shT$ which are closed under taking direct summands. The following fact is simple but very useful.

\begin{lemma} \label{lem:component_van} 
Assume $\shT$ admits a semiorthogonal decomposition (\ref{SOD:T}), and let $\shA$ be a triangulated subcategory of $\shT$. Let $a \in \shT$, and $a_k$ be its component in $\shA_k$. If $a \in \shA^\perp$ (resp. ${}^\perp \shA$), and $a_k \in \shA^\perp$ (resp. ${}^\perp \shA$) for all $k \ne l$, then $a_l \in \shA^\perp$ (resp. ${}^\perp \shA$).
\end{lemma}

\begin{proof} The lemma follows directly from the fact that $\shA^\perp$ (resp. ${}^\perp \shA$) is a triangulated subcategory, i.e. is closed under shifts and taking cones. \end{proof}

Semiorthogonal decompositions of $\shT$ can be obtained from a subcategory $\shA$ and its orthogonals if it is admissible.
A full triangulated subcategory $\shA$ of $\shT$ is called \textbf{admissible} if the inclusion functor $i: \shA\to \shT$ has both a right adjoint $i^!: \shT \to \shA$ and a left adjoint $i^*: \shT \to \shA$. If $\shA\subset \shT$ is admissible, then $\shA^\perp$ and ${}^\perp \shA$ are both admissible, and $\shT = \langle \shA^\perp, \shA \rangle$ and $\shT = \langle \shA, {}^\perp \shA\rangle$.

\subsection{Mutations.} Given a semiorthogonal decomposition of $\shT$, one can obtain a whole collection of new decompositions by functors called \textbf{mutations}. 

\begin{definition} Let $\shA$ be an admissible subcategory of a triangulated category $\shT$. Then the functor $\LL_\shA: = i_{\shA^\perp} i^*_{\shA^{\perp}}$  (resp. $\RR_{\shA} : =  i_{{}^\perp \shA} i^!_{ {}^\perp \shA}$) is called the \textbf{left (resp. right) mutation through $\shA$}., where $ i_{\shA^\perp}$ (resp. $ i_{{}^\perp \shA}$) denotes the inclusion functor $\shA^\perp \to \shT$ (resp. ${}^\perp \shA \to \shT$), and $ i^*_{\shA^{\perp}}$ (resp. $i^!_{{}^\perp \shA}$) is its left (resp. right) adjoint.
\end{definition}

We will only focus on the left mutation functors in this paper; the statements on right mutations are similar. The following results are standard, cf. \cite{B}, \cite{BK}, or \cite{Kuz06Hyp}.
\begin{lemma} \label{lem:mut} Let $\shA$ and $\shA_1, \ldots, \shA_n$ be admissible subcategories of a triangulated category $\shT$ where $n \ge 2$ is an integer. Let $k$ be an integer, $2 \le k \le n$.
\begin{enumerate}
	\item $(\LL_{\shA})\,|_{\shA} = 0$ is the zero functor, $(\LL_{\shA})\,|_{{}^\perp \shA} : {}^\perp \shA \to \shA^\perp$ is an equivalence of categories, and $(\LL_{\shA})\,|_{\shA^\perp} = \Id_{\shA^\perp}: \shA^\perp \to \shA^\perp$ is equal to the identity functor.
	\item For any $b \in \shT$, there is a distinguished triangle
		$$ i_\shA i^!_{\shA} (b) \to b \to \LL_{\shA} b \xrightarrow{[1]}{}.$$
	\item \label{lem:mut:sod} If $\shA_1, \ldots, \shA_n$ is a semiorthogonal sequence, then $\LL_{\langle \shA_1,\shA_2, \ldots, \shA_n \rangle} = \LL_{\shA_1} \circ \LL_{\shA_2} \circ \cdots \circ \LL_{\shA_n}$.
	\item \label{lem:mut:span} If $\shA_1, \ldots, \shA_{k-1}, \shA_k, \shA_{k+1}, \ldots, \shA_n$ is a semiorthogonal sequence inside $\shT$, then 
		$$\shA_1, \ldots, \LL_{\shA_{k-1}} (\shA_k), \shA_{k-1}, \shA_{k+1}, \ldots, \shA_n$$
		 is also a semiorthogonal sequence, and it generates the same subcategory
		$$\langle \shA_1, \ldots, \shA_{k-1}, \shA_k, \shA_{k+1}, \ldots, \shA_n\rangle = \langle \shA_1, \ldots, \LL_{\shA_{k-1}} (\shA_k), \shA_{k-1}, \shA_{k+1}, \ldots, \shA_n \rangle.$$
	\item Let $F: \shT \to \shT$ be an autoequivalence, then $F \circ \LL_\shA  \simeq \LL_{F(\shA)} \circ F$.
\end{enumerate}
\end{lemma}

\subsection{Serre functor}
\begin{definition} For a triangulated category $\shT$, a \textbf{Serre functor} is a covariant functor $S: \shT \to \shT$ that is an autoequivalence, such that for any two objects $F,G \in \shT$, there is a bi-functorial isomorphism 
		$$\Hom(F,G) = \Hom(G, S(F))^*,$$
		where $^*$ denotes the dual vector space over $\kk$.
\end{definition}
If a Serre functor exists, then it is unique up to canonical isomorphisms. If $X$ is a smooth projective variety of dimension $n$, then $D(X)$ has a Serre functor given by $S_X (-) = - \otimes \omega_X [n]$, where $\omega_X$ is the dualizing sheaf on $X$. For example, when $X = \PP^{l-1}$ projective space, then $S_X = \otimes \sO(-l)[l-1]$. Serre functor $S_\shT$ commutes with any $\kk$-linear auto-equivalence of $\shT$.

The following are properties of Serre functors, cf. \cite{B, BK}.
\begin{lemma} \label{lem:serre} Let $\shT$ be a triangulated subcategory with a Serre functor $S$, and $\shA$ be an admissible subcategory. Then
\begin{enumerate}
	\item $S({}^\perp \shA) = \shA^\perp$, and $S^{-1}(\shA^\perp) = {}^\perp \shA$. In particular, if $\shT = \langle \shA, \shB\rangle$, then $\shB = {}^\perp \shA$ , $\shA = \shB^{\perp}$, and $\shT  = \langle S(\shB), \shA \rangle = \langle \shB, S^{-1}(\shA)\rangle$.
	\item $\shA$ also admits a Serre functor given by $S_{\shA} = i_{\shA}^! \circ S \circ i_{\shA}$.
\end{enumerate}
\end{lemma}

\subsection{Fourier-Mukai transform} Let $X$, $Y$ be smooth varieties throughout this section.
\begin{definition} A \textbf{Fourier-Mukai transform}  (or a Fourier-Mukai functor) given by a Fourier-Mukai \textbf{kernel} $\shP \in D(X \times Y)$, denoted by $\Phi_{\shP}$, is the functor
	$$\Phi_{\shP}: = \pi_{Y*} \,( \pi_X^*\,(-) \otimes \shP): D(X) \to D(Y),$$
where $\pi_{X}: X\times Y \to X$ and $\pi_Y: X \times Y \to Y$ are projection functors. Sometimes we use the notation $\Phi^{X \to Y}_{\shP} = \Phi_{\shP}$ if we need to emphasize it is a functor from $D(X)$ to $D(Y)$.
\end{definition}

Let $S$ another smooth variety, and $f:X \to S$, $g:Y \to S$ be proper maps. Denote the $p_X$, $p_Y$ the projection from $X\times_S Y$ to $X$, $Y$, and $i: X\times_S Y \hookrightarrow X \times Y$ the inclusion. 
\begin{definition} The relative Fourier-Mukai transform $\Phi^S_{\shP}: D(X) \to D(Y)$ given by a kernel $\shP \in D(X \times_S Y)$, is defined to be the functor
	$$\Phi^S_\shP (-) :=  p_{Y*} \, (p_{X}^*\, (-) \otimes \shP): D(X) \to D(Y).$$
\end{definition}

The situation is illustrated in diagram 
\begin{equation}
\begin{tikzcd}[row sep=1.5 em, column sep=3.0 em]
	& X\times_S Y \ar{ldd}[swap]{p_X} \ar{rdd}{p_Y} \ar[left hook->]{d}{i} \\ 
	& X\times Y \ar{ld}{\pi_X} \ar{rd}[swap]{\pi_Y} \\
	X && Y
\end{tikzcd}
\end{equation}

The Fourier-Mukai transform $\Phi^S_\shP$ is \textbf{$S$-linear} (cf. next section), and these are sometimes called \textbf{geometric $S$-linear} functors $D(X) \to D(Y)$. For $\shP \in D(X\times_S Y)$, the relative Fourier-Mukai transform $\Phi^S_\shP$ is nothing but the ordinary one with kernel $i_* \shP \in D(X \times Y)$, i.e.
	$$\Phi_{i_*\,\shP} = \Phi^S_{\shP}: D(X) \to D(Y).$$
In fact, by projection formula
	$$  p_{Y*} \, (p_{X}^*\, \shE \otimes \shP) = \pi_{Y*} i_* ( p_X^*\, \shE \otimes \,\shP) = \pi_{Y*} i_* (i^*\, \pi_X^*\, \shE \otimes  \shP)  =  \pi_{Y*} \,( \pi_X^*\, \shE \otimes i_*\,\shP).$$
Hence we abandon the notation $\Phi^S_\shP$ and use $\Phi_{i_*\,\shP}$ or simply $\Phi_{\shP}$ to denote this functor. 

\begin{lemma}[Adjoints]\label{lem:adjoint} Let $\shP \in D(X \times_S Y)$, and assume $X\times_S Y$ is smooth. Then the left and right adjoint of $\Phi_{\shP}: D(X) \to D(Y)$ can be given by Fourier-Mukai kernels
	\begin{align*}  \shP^L : =  \shP^\vee \otimes \omega_{p_X} [\dim p_X], \qquad \shP^R : = \shP^\vee \otimes \omega_{p_Y} [\dim p_Y] . 
	\end{align*}
Here $ \omega_{p_X} : = \omega_{X \times_S Y} \otimes p_X^*\,\omega_X^\vee$, $\omega_{p_Y} : = \omega_{X \times_S Y} \otimes p_Y^* \, \omega_Y^\vee$, and $\dim p_X = \dim X \times_S Y - \dim X$, $\dim p_Y = \dim X \times_S Y - \dim Y $, $(-)^\vee: = R\sHom(-, \sO_{X\times_S Y})$ is the dual on $X \times_S Y$.
\end{lemma}
 
\begin{proof} For $E \in D(Y), F \in D(X)$, let $\shP^L \in D(X\times_S Y)$ be defined as in the lemma, then	
	\begin{align*} & \Hom_X(\Phi_{\shP^L} (E), F) \\
		& = \Hom_X(p_{X*} (\shP^L \otimes p_Y^*\, E), F)   \\
		& =  \Hom_{X\times_S Y} (\shP^L \otimes  p_Y^*\,E, p_X^*\,F \otimes \omega_{p_X} [\dim p_X] )  \qquad \text{(Grothendieck-Verdier duality for $p_X$)} \\
		& = \Hom_{X\times_S Y}(\shP^\vee \otimes p_Y^*\,E \otimes \omega_{p_X} [\dim p_X], p_X^*\,F \otimes \omega_{p_X} [\dim p_X] ) \\
		& = \Hom_{X\times_S Y}(p_Y^*\,E, \shP \otimes p_X^*\,F) = \Hom_Y(E, \Phi_{\shP}(F)).
	\end{align*}
Similarly for $\shP^R$.
\end{proof}

Above defined kernel $\shP^L$ and $\shP^R$ is compatible with base-change when the fiber products $X\times_S Y$ are smooth. 

\begin{definition}\label{def:op}
Let $X$, $Y$ be smooth varieties, the \textbf{opposite} of a functor $\Phi: D(X) \to D(Y)$ is defined to be $\Phi^{op}: = (-)^\vee \circ \Phi \circ (-)^\vee$, where $(-)^\vee = R\sHom(-, \sO)$.
\end{definition}

\begin{lemma} \label{lem:op} Let $\shP \in D(X \times_S Y)$, and assume $X\times_S Y$ is smooth. Then the opposite of $\Phi_{i_* \shP}$ is given by a geometric Fourier-Mukai transform $ \Phi_{i_* (\shP^{op}) }$ with kernel
	$$\shP^{op} = R\sHom_{X\times_S Y} (\shP,  \omega_{p_Y} [\dim p_Y] ) \in D(X \times_S Y), $$	
\end{lemma}
\begin{proof} For $F \in D(X)$, we have
	\begin{align*} 
	 \Phi_{\shP}^{op}(F) & = R\sHom_Y(\Phi_{\shP}(F^\vee), \sO_Y) \\
	& = R\sHom_Y(p_{Y*} (p_X^*F^\vee \otimes \shP), \sO_Y) \\
	&= p_{Y*} R\sHom_{X\times_S Y}(p_X^*\,F^\vee \otimes \shP, p_Y^! \sO_Y)  \qquad \text{(local Groth.-Verdier duality for $p_Y$)} \\
	&= p_{Y*}  R\sHom_{X\times_S Y} (p_X^*\, F^\vee \otimes \shP, \omega_{p_Y}[\dim p_Y])	\qquad (p_Y^! (-) = p_Y^*(-) \otimes \omega_{p_Y}[\dim p_Y]) \\
	& = p_{Y*} ( p_X^* F \otimes R\sHom(\shP, \omega_{p_Y}[\dim p_Y]) ).
	\end{align*}
\end{proof}

In general, if $X\times_S Y$ is not smooth, the kernel of left adjoint $\Phi_{\shP}^L$, or the opposite $\Phi_\shP^{op}$, where $\shP \in D(X \times_S Y)$, may not be described easily as a complex on $X\times_S Y$, due to the fact that $R\sHom(-, \sO_{X\times_S Y})$ and $p_X^!$, $p_Y^!$ may not behave well when $X\times_S Y$ is singular. However, the kernel of their composition is very simple.

\begin{lemma}\label{lem:op_of_L}
If $\shP \in D(X \times_S Y)$, then $(\Phi_{\shP}^L)^{op}: D(Y) \to D(X)$ is a Fourier-Mukai transform with kernel (the pushforward to $D(X \times Y)$ of)
	$$(\shP^L)^{op} = \shP \in D(X \times_S Y).$$
\end{lemma}
\begin{proof} 
For $\shE \in D(X\times Y)$, denote $\pi_X$, $\pi_Y$ the projection from $X\times Y$ to $X$ and $Y$. Then composing the formula of Lem. \ref{lem:op} with Lem. \ref{lem:adjoint} for $S = \pt$, we get 
	$$(\shE^L)^{op} = {\shE}  \in D(X \times Y).$$
Apply to $\shE = i_* \shP$, we get the desired formula.
\end{proof}
 
 \begin{remark} Symmetrically we have $(\shP^{op})^R = \shP$.
 \end{remark}
 
The kernel of the composition of two Fourier-Mukai transforms are given by \textbf{convolutions}. The following direct generalization in relative setting of \cite[Prop. 5.10]{Huy} may be well-known, but we include a proof here for lack of reference.
\begin{lemma}[Convolution] \label{lem:rel_convolution} Let $X$, $Y$ be $S$-varieties, and $Y$, $Z$ be $T$-varieties, and $S \to U$, $T \to U$ are maps between smooth varieties, such that the induced $U$-variety structures on $Y$ agree. Consider $S$-linear resp. $T$-linear Fourier-Mukai transform $\Phi_\shP: D(X) \to D(Y)$, and $\Phi_\shQ: D(Y) \to D(Z)$, where $\shP \in D(X\times_S Y)$ and $\shQ \in D(Y \times_T Z)$. Then their composition is a $U$-linear Fourier-Mukai transform $\Phi_\shR = \Phi_\shQ \circ \Phi_\shP: D(X) \to D(Z)$, with Fourier-Mukai kernel $\shR$ given by the \textbf{convolution} of kernels 
	$$\shR = p_{XZ*}\, (p_{XY}^*\, \shP \otimes p_{YZ}^*\, \shQ) \in D(X \times_U Z),$$
where maps are indicated in diagram (\ref{diagram:convolution}), if the following squares are {Tor-independent} \footnote{For the definition and criteria of Tor-independence, see \S \ref{sec:base-change}.}:
\begin{equation} \label{eqn:convolution_sq}
\begin{tikzcd} 
	X \times_S Y \times_T Z  \ar{r}{p_{YZ}} \ar{d}{p_{XY}}& Y \times_T Z  \ar{d}{u}\\
	X\times_S Y \ar{r}{p} & Y.
\end{tikzcd}
\end{equation}
In particular, if $S=T=U = \Spec \kk$, the square (\ref{eqn:convolution_sq}) is automatically Tor-independent, and we obtain the usual convolution formula.
\end{lemma}

\begin{proof}
\begin{equation}\label{diagram:convolution}
\begin{tikzcd}[back line/.style={}]		
		&	& X\times_S Y \times_T Z   \ar{ld}[swap]{p_{XY}} \ar{d}{p_{XZ}}\ar{dr}{p_{YZ}}	 \ar[bend right]{ddll}[swap]{p_X} \ar[bend left]{ddrr}{p_Z} \\
		& X\times_S Y   \ar{ld}[swap]{q} \ar[back line]{rd}[swap]{p} & X\times_U Z \ar[crossing over]{lld}{s} \ar[back line]{rrd}[swap]{r} & Y \times_T Z  \ar[crossing over]{ld}{u} \ar{rd}{t} \\
	X \ar{rd}{a_X}	&	& Y\ar{ld}[swap]{a_Y} \ar{rd}{b_Y}	&	& Z \ar{ld}{b_Z}\\
		&S \ar{r}{c_S}	& U	&T  \ar{l}[swap]{c_T} 	
\end{tikzcd}
\end{equation}
The computation is similar to the proof in absolute case \cite[Prop. 5.10]{Huy}:
\begin{align*} \Phi_{\shR}(\shE) & = r_*\, (s^* \,\shE  \otimes p_{XZ*}\,(p_{XY}^* \,\shP \otimes p_{YZ}^*\,\shQ)) \\
	& = r_*\,(p_{XZ*} \, (p_X^*\,\shE \otimes p_{XY}^* \shP \otimes  p_{YZ}^*\,\shQ)) \\
	& =  p_{Z*}\,(p_X^*\,\shE \otimes p_{XY}^* \shP \otimes  p_{YZ}^*\,\shQ)\\
	& = p_{Z*}\,(p_{XY}^*\,(q^*\,\shE \otimes \shP) \otimes p_{YZ}^* \shQ) \\
	& = t_* \,(p_{YZ*}\, p_{XY}^*\,( q^* \, \shE \otimes \shP) \otimes \shQ) \\
	& = t_*\,( u^*\, p_* ( q^* \, \shE \otimes \shP) \otimes \shQ ) \qquad (\text{base-change by Tor-independence})\\
	& =t_*\,( u^*\,  \Phi_\shP (\shE) \otimes \shQ) =   \Phi_{\shQ} \circ \Phi_{\shP} (\shE).
\end{align*}
The only place we use the condition of the lemma is the base-change step, where $p_{YZ*}\, p_{XY}^* = u^*\,p_*$ is guaranteed by the Tor-independence of the square (\ref{eqn:convolution_sq}), cf. Lem. \ref{lem: Tor_ind}.
\end{proof}

\begin{remark} 
\begin{enumerate}[leftmargin=0 cm, itemindent=*, align=left]
	\item	 It should be noted that the Tor-independent condition of the square (\ref{eqn:convolution_sq}) are not automatic. For a given $U$, usually we have the freedom of chosing the bases $S$ and $T$. The condition may not be satisfied by $S=T=U$, but by other choices of $S,T$.
	\item  For given bases $S$, $T$, if the square (\ref{eqn:convolution_sq}) is Tor-independent, then the only condition on $U$ is the following diagram commutes:
	\begin{equation} \label{diagram:commute_Y}
	\begin{tikzcd} Y  \ar{r}{b_Y} \ar{d}[swap]{a_Y} & T \ar{d}{c_T} \\
	S \ar{r}{c_S} & U.
	\end{tikzcd}
	\end{equation}
	Hence the lemma holds for $U = \Spec \kk$, as well as the 'maximal choice' of $U$ such that diagram (\ref{diagram:commute_Y}) commute. 
\end{enumerate}
\end{remark}

\subsection{Derive categories over a base.} \label{sec:base-change} The following is essential in the first step of our proof of main theorem. Readers are recommended to refer to \cite{Kuz06Hyp}, \cite{Kuz07HPD} or \cite{Kuz11Bas} for more details. 

Assume $S$ is a smooth\footnote{The smoothness is not necessary. In general we need only to consider the derived category of perfect complexes $D^{perf}(S)$ on $S$ rather than the bounded derived category $D(S)$ of coherent sheaves.} algebraic variety, then a triangulated category $\shT$ is called $S$-linear if it admits a module structure over the tensor triangulated category $D(S)$. Let $f:X \to S$ be a map between smooth varieties, then $D(X)$ is naturally equipped with $D(S)$-modules structure from the (derived) pullback. 

\begin{definition} An admissible subcategory $\shA \subset D(X)$ is called \textbf{$S$-linear} if for any $a \in \shA$ and $F \in D(S)$, we have $a \otimes f^* F \in \shA$. 
\end{definition}

We may regard an $S$-linear triangulated subcategory as a family of categories over $S$. Under certain conditions we can pullback the the family through base change $\phi: T \to S$ to get a family of categories over $T$ with desired properties.

\begin{definition} \label{def:bc} A base change $\phi: T \to S$ is called \textbf{faithful} with respect to a morphism $f: X \to S$ if the cartesian square
	\begin{equation}\label{eqn:fiber}
	\begin{tikzcd}
	X_T \arrow{d}{}[swap]{f_T}  \arrow{r}{\phi_T} & X \arrow{d}{f} \\
	T\arrow{r}{\phi}  & S
	\end{tikzcd}
	\end{equation}
is \textbf{exact cartesian}, i.e., the natural transformation $ f^* \circ \phi_* \to \phi_{T *} \circ f_T^*: D(T) \to D(X)$ is an isomorphism. A base change $\phi: T \to S$ is called \textbf{faithful with respect to a pair $(X, Y)$} if $\phi$ is faithful with respect to morphisms $f: X \to S$, $g: Y \to S$, and $f\times_S g : X \times_S Y \to S$.
\end{definition}

The condition for faithful base-change is typically satisfied by so called $\Tor$-independent squares. A fiber square (\ref{eqn:fiber}) is called \textbf{Tor-independent} if for all $t \in T$, $x \in X$, and $s \in S$ with $\phi(t) = s = f(x)$, and all $i > 0$, $\Tor_i^{\sO_{S,\,s}}(\sO_{T,\,t}, \sO_{X,\,x}) = 0$.

\begin{lemma}\label{lem: Tor_ind}A base-change $\phi: T \to S$ for $X \to S$ is faithful if and only if the square (\ref{eqn:fiber}) is Tor-independent.
\end{lemma}
\begin{proof} \cite[Thm. 3.10.3]{Lip09}; This is also essentially proved in \cite{Kuz06Hyp}.
\end{proof}

Therefore we will interchange freely between the terms 'exact cartesian', 'Tor-Independent', 'faithful base change' to refer this same condition. 
The following are certain typical situations when the condition is satisfied.
\begin{lemma}\label{lemma-faithful-base-change}
 Let $f:X \to S$ be a morphism and $\phi: T \to S$ a base change.
	\begin{enumerate}
	\item If $\phi$ is flat, then it is faithful.
	\item If $T$ and $X$ are smooth and $X_T$ has expected dimension, i.e. $\dim X_T = \dim X + \dim T - \dim S$, then $\phi: T\to S$ is faithful with respect to the morphisms $f: X \to S$.
	\item If $\phi: T\to S$ is a closed embedding and $T\subset S$ is a locally complete intersection, and both $S$ and $X$ are Cohen-Macaulay, and $X_T$ has expected dimension, i.e. $\dim X_T  = \dim X + \dim T - \dim S$, then $\phi: T\to S$ is faithful with respect to the morphisms $f: X \to S$.
	\end{enumerate}
\end{lemma}
\begin{proof} \cite[Cor. 2.23, 2.27]{Kuz06Hyp}, \cite[Lem. 2.31, 2.35]{Kuz07HPD}. 
\end{proof}

The following will be useful later:

\begin{lemma} \label{lem:3squares} If in the following cartesian squares of varieties, the right square is exact cartesian:
\begin{equation*}
	\begin{tikzcd}
	X'' \ar{d} \ar{r} & X' \ar{d} \ar{r} & X \ar{d} \\
	S'' \ar{r}		& S' \ar{r}		& S.
	\end{tikzcd}
\end{equation*}
Then the ambient square is exact cartesian if and only if the left square is.

\end{lemma}
\begin{proof} \cite[Lem. 2.25]{Kuz06Hyp}.
\end{proof}

The power of faithful base change is that it preserves $S$-linear fully faithful Fourier-Mukai transforms and $S$-linear semiorthogonal decomposition. Let us fix $S=P$ be a smooth projective variety. 

\begin{proposition}[{\cite[Prop. 2.38]{Kuz07HPD}}] \label{prop:bcFM} If $\phi: T \to P$ is a faithful base change for a pair $(X, Y)$ where $f:X\to P$ and $g:Y\to P$, and varieties $X$ and $Y$ are projective over $P$ and smooth, and $K\in D(X\times_P Y)$ is a kernel such that $\Phi_K : D(X) \to D(Y)$ is fully faithful. Then $\phi_{K_T} : D(X_T) \to D(Y_T)$ is fully faithful, where the Fourier-Mukai kernel is  $K_T := \phi_T^*K$.\footnote{$K_T$ a priori only belongs to the bounded above derived category  $D^-(X_T \times_T Y_T)$ of quasi-coherent complexes with coherent cohomologies. \cite[Lem. 2.4]{Kuz06Hyp} guarantees that $\Phi_{K_T}$ preserves boundedness.}
\end{proposition}


\begin{proposition} [{\cite[Thm. 5.6]{Kuz11Bas}}]\label{prop:bcsod} If $f:X \to P$ a map between smooth varieties, $D(X) = \langle \shA_1, \ldots, \shA_n \rangle$ is a semiorthogonal decompositions by admissible $P$-linear subcategories. Let $\phi: T \to P$ is a faithful base change for $f$, then we have a $T$-linear semiorthogonal decomposition
	$$D(X_T) = \langle \shA_{1T}, \ldots, \shA_{nT} \rangle $$
where $\shA_{kT}$ is the base-change category of $\shA_k$ to $T$, which depends only on $\shA_k$, i.e. independent of the embedding $\shA_k \subset D(X)$, and satisfies $\phi_T^*(a) \in \shA_{kT}$ for any $a \in \shA_{kT}$, and $\phi_{T*}(b) \in \shA_k$ for $b \in \shA_{kT}$ with proper support over $X$. 
\end{proposition}

\subsection{Exterior Products.} The exterior products of admissible subcategories of bounded derived categories of general algebraic varieties can be defined conveniently using base-change, cf. \cite[\S 5.5, 5.6]{Kuz11Bas}. Here we restrict ourselves to the case of product $X \times Y$ of two smooth quasi-projective varieties, say $X$, $Y$. Denote by $\pi_X$ (resp. $\pi_Y$) the projection from $X\times Y$ to $X$ (resp. $Y$). Notice $D(X)$ and $D(Y)$ are isomorphic to their respective categories of perfect complexes. Assume we have semiorthogonal decompositions by admissible subcategories
	$$D(X) = \langle \shA_1, \ldots, \shA_m\rangle, \quad \text{and} \quad D(Y) = \langle \shB_1, \ldots, \shB_n \rangle.$$ 
Define 
	$$\shA_i \boxtimes D(Y) : = \langle \pi_X^* \shA_i \otimes \pi_Y^* D(Y)\rangle \subset D(X \times Y),\quad  i=1,2,\ldots, m$$
by which we mean $\shA_i \boxtimes D(Y)$ is the minimal triangulated subcategory of $D(X \times Y)$ which is closed under taking summands and contains all objects of the form $\pi_X^* a \otimes \pi_Y ^* F$ for $a \in \shA_i$ and $F \in D(Y)$. We define $D(X) \boxtimes \shB_j$ similarly. Then the results of \cite[\S 5.6]{Kuz11Bas} can be summarized:

\begin{proposition} \label{prop:product}
 With the notations and assumptions as above, there are $Y$-linear and resp. $X$-linear semiorthogonal decompositions
	$$D(X\times Y) = \big\langle \shA_i \boxtimes D(Y)\big\rangle_{1 \le i \le m}, \quad\text{and}\quad D(X\times Y) = \big\langle D(X)\boxtimes \shB_j \big\rangle_{1 \le j \le n}.$$	
Furthermore, 
	$$D(X\times Y) = \big\langle \shA_i \boxtimes \shB_j \big\rangle_{1\le i \le m, 1 \le j \le n},$$
where the \textbf{exterior product} is defined by $\shA_i \boxtimes \shB_j : = \shA_i \boxtimes D(Y) \cap D(X) \boxtimes \shB_j \subset D(X \times Y)$. Moreover, we have semiorthogonal decompositions
	$$ \shA_i \boxtimes D(Y) = \langle \shA_i \boxtimes \shB_1, \ldots, \shA_i \boxtimes \shB_n \rangle \quad \text{and} \quad D(X) \boxtimes \shB_j =\langle \shA_1 \boxtimes \shB_j, \ldots, \shA_m \boxtimes \shB_j\rangle. $$
\end{proposition}

\medskip
\subsection{Lefschetz decompositions} \label{sec:lef} The next two sections are devoted to Kuznetsov's theory of Homological Projective Duality. The HP-duality theory considers varieties admitting a special type of decompositions, called Lefschetz decompositions, introduced in \cite{Kuz07HPD}.
\begin{definition} Let $X$ be a variety and $\sO(1)$ be a line bundle on $X$. A \textbf{Lefschetz decomposition} of $D(X)$ is a semiorthogonal decomposition of the form
	\begin{equation}\label{lef:X} 
	D(X) = \langle \shA_0, \shA_1 (1), \ldots, \shA_{i-1}(i-1)\rangle,
	\end{equation}
	with $\shA_0 \supset \shA_1 \supset \cdots \supset \shA_{i-1}$ a descending sequence of admissible subcategories, where $(k)$ denotes $\otimes \sO(k)$. The number $i \in \ZZ_{>0}$ is called the \textbf{length} of the decomposition. The Lefschetz decomposition (\ref{lef:X}) is called \textbf{rectangular} if $\shA_{0} = \shA_{1} = \cdots = \shA_{i-1}$. 
\end{definition}

\begin{remark} A Lefschetz decomposition is totally determined by its first component $\shA_0$: all other $\shA_k$'s can be reconstructed from $\shA_0$ via $\shA_k = {}^\perp {\shA_0}(-k) \cap \shA_{k-1}$, cf. \cite[Lem. 2.18]{Kuz08Lef}. 
\end{remark}

We adopt the following notations from Kuznetsov \cite{Kuz07HPD}: let $\foa_k$ to be the right orthogonal of $\shA_{k+1}$ inside $\shA_{k}$ for $0 \le k \le i-1$. Then $\foa_k$'s are admissible subcategories, and 
	$$\shA_k = \langle \foa_{k}, \foa_{k+1}, \ldots, \foa_{i-1}\rangle.$$
Hence it makes sense to adopt conventions $\shA_k = \shA_0$ if $k \le 0$ and $\shA_k = 0$ if $k \ge i$. Let $\alpha_0: \shA_0 \to D(X)$ be the inclusion functor and $\alpha_0^*: D(X) \to \shA_0$ be its left adjoint. 

\begin{lemma} \label{lem:sod:A_0} Assume $1 \le k\le i$, then $\alpha_0^*(\foa_0(1)), \ldots, \alpha_0^*(\foa_{k-1}(k))$ is a semiorthogonal collection, and we have
	$$\langle \alpha_0^*(\foa_0(1)), \ldots, \alpha_0^*(\foa_{k-1}(k)), \shA_1(1), \ldots, \shA_{k}(k)\rangle = \langle \shA_0(1), \ldots, \shA_{k-1}(k)\rangle.$$
In particular, set $k=i$, then $\shA_0 = \langle \alpha_0^*(\foa_0(1)), \alpha_0^*(\foa_1(2)), \ldots, \alpha_0^*(\foa_{i-1}(i))\rangle$. This gives another semiorthogonal decomposition of $\shA_0$.
\end{lemma}

\begin{proof} For a proof using Serre functor, cf. \cite[Lem. 4.3]{Kuz07HPD}. We give another proof as an application of properties of mutation functors. By properties (\ref{lem:mut:sod}) and (\ref{lem:mut:span}) of Lem. \ref{lem:mut},
	\begin{align*}  &\langle \shA_0(1), \ldots, \shA_{k-1}(k)\rangle = \langle \foa_0(1), \shA_1(1), \foa_1(2), \shA_2(2), \ldots, \foa_{k-1}(k), \shA_k(k)\rangle \\
		=&  \langle \foa_0(1), \LL_{\shA_1(1)}\foa_1(2),\shA_1(1),  \shA_2(2), \foa_2(3),  \ldots, \foa_{k-1}(k), \shA_k(k) \rangle \\
		=& \langle  \foa_0(1), \LL_{\shA_1(1)} \foa_1(2), \LL_{\langle \shA_1(1),\shA_2(2) \rangle} \foa_2(3), \shA_1(1), \shA_2(2), \shA_3(3) \ldots,  \foa_{k-1}(k), \shA_k(k)\rangle \\
		=& \langle  \foa_0(1), \LL_{\shA_1(1)} \foa_1(2),  \ldots, \LL_{\langle \shA_1(2), \ldots, \shA_{k-1}(k-1)\rangle}\foa_{k-1}(k), \shA_1(1), \shA_2(2), \ldots, \shA_{k}(k) \rangle
	\end{align*}
	Notice since $\alpha_0^* = \LL_{\langle \shA_1(1), \ldots, \shA_{i-1}(i-1)\rangle}$, and we have semiorthogonal decomposition
		$$D(X) = D(X)(1) =  \langle \foa_0(1), \shA_1(1), \foa_1(2), \shA_2(2), \ldots, \foa_{i-1}(i)\rangle.$$
	Therefore $ \foa_0(1) = \alpha_0^* ( \foa_0(1))$, $ \LL_{\shA_1(1)} \foa_1(2) = \alpha_0^*(\foa_1(2))$, $\cdots$, and $\LL_{\langle \shA_1(2), \ldots, \shA_{k-1}(k-1)\rangle}\foa_{k-1}(k) = \alpha_0^*(\foa_{k-1}(k))$.
\end{proof}

Using the new decomposition of $\shA_0$, we can build up a series of ascending subcategories\footnote{The relation between our notation $\shB^k$'s and Kuznetsov's notation $\shB_{l}$'s in \cite{Kuz07HPD} are $\shB^k \equiv \shB_{N-k-1}$.}
	\begin{equation}
		\shB^k :=  \langle \alpha_0^* (\foa_0(1)), \ldots, \alpha_0^* (\foa_{k-1}(k))\rangle \subset \shA_0, \quad \text{for} \quad k =1, 2, 3, \ldots
	\end{equation} \label{def:B^k}
	Then $\shB^1 \subset \shB^2 \subset \cdots \subset \shB^{N-1} = \shA_0$. We denote $\shB^{k} = 0$ if $k \le 0$, and $\shB^k = \shA_0$ if $k\ge N-1$. Notice we have isomorphism
	$$\shB^k \simeq \langle \foa_0, \foa_1 \ldots, \foa_{k-1}\rangle,$$
hence we can regard $\shB^k$ as the complementary of $\shA_{k}$ inside $\shA_0$. See the first diagram in Figure \ref{Figure:lf_2}.

\subsection{Hyperplane Section and Homological Projective Duality.} HP-duality theory is mainly concerned about the derived category of linear sections of a given projective variety. One important idea, similar to the classical topological theory of Lefschetz, is to consider the universal family of hyperplane sections of a variety over the dual projective space, instead of only one hyperplane section. Let $X$ be a smooth projective variety, with a morphism to projective space $f: X \to \PP V$, and $X$ has a Lefschetz decomposition (\ref{lef:X}) with respect to $\sO_X(1) : = f^* \sO_{\PP V} (1)$. Note that $f$ is not necessarily an embedding. We will make the following assumption in all considerations of HP-duality theory:

\medskip\noindent \emph{Assumption $(\dagger)$.} For $f: X \to \PP V$
with a Lefschetz decomposition (\ref{lef:X}),
we assume $\overline{ f(X) }$ is of dimension at least $2$, and
the length of~(\ref{lef:X}) is less than $\dim_k V$.
  
\begin{definition} The \textbf{universal hyperplane section} $\shH_X$ is defined to be the subscheme $X \times_{\PP V} Q  \subset X \times \PP V^*$, where $Q \subset \PP V \times \PP V^*$ is the incidence quadric $Q=\{(x,H) ~|~ x \in H\}$.
\end{definition}

Consider the projection $\pi_X: \shH_X \to X$, then it is easy to see $\shH_X$ is a smooth projective variety of dimension $\dim X + N - 2$ (where $N = \dim_\kk V$). From Prop. \ref{prop:product}, we have $\PP V^*$-linear semiorthogonal decomposition
	$$D(X \times \PP V^*) = \langle \shA_0 \boxtimes D(\PP V^*), \shA_1(1) \boxtimes D(\PP V^*), \cdots, \shA_{i-1}(i-1)\boxtimes D(\PP V^*) \rangle.$$

\begin{lemma} \label{sod:H_X} Denote $i_{\shH_X}: \shH_X \to X \times \PP V^*$ the inclusion morphism and $i_{\shH_X}^*$ the derived pullback functor. Then $i_{\shH_X}^*$ is fully faithful on the subcategories $\shA_1 \boxtimes D(\PP V^*)$, $\ldots$,  $ \shA_{i-1}(i-1)\boxtimes D(\PP V^*) $, and the images induce a $\PP V^*$-linear decomposition
	\begin{equation} \label{sod:H_X}
	D(\shH_X) = \langle \sC,  \shA_1(1)\boxtimes D(\PP V^*), \cdots, \shA_{i-1}(i-1)\boxtimes D(\PP V^*)\rangle,
	\end{equation}
where the $\PP V^*$-linear subcategory $\sC\subset D(\shH_X)$ is defined as the right orthogonal of the images.
\end{lemma}
Here we use the same notations $\shA_k(k)\boxtimes D(\PP V^*)$ to denote their fully faithful images inside $D(\shH_X)$ under the pullback functor $i_{\shH_X}^*$, for simplicity of notations. 
\begin{proof} The result follow easily from computing $R\Hom_{\shH_X}$ in terms of $R\Hom_{X\times \PP V^*}$, cf. \cite[Lem. 5.3]{Kuz07HPD}, or the more general situation (\ref{cone-for-Hom}) where we can take $S = \PP V^*$.
\end{proof}

The derived category $\sC$ in the decomposition (\ref{sod:H_X}) is called the \textbf{HP-dual category} of $X\to \PP V$ with respect to (\ref{lef:X}). As a $\PP V^*$-linear category the HP-dual always exists. The most interesting case is when the category is realized \textbf{geometrically}. This leads to the concept of Homological Projective Duality.

\begin{definition} A variety \footnote{for generalizations to noncommutative schemes, see \S \ref{sec:ncHPD}.} $Y$ together with a morphism $Y \to \PP V^*$ is called \textbf{Homological Projective dual} or \textbf{HP-dual} to $X \to \PP V$ with respect to the Lefschetz decomposition (\ref{lef:X}), if there exists an object $\shE_X \in D(\shH_X \times_{\PP V^*} Y)$ such that the Fourier-Mukai transform $\Phi_{\shE_X} : D(Y) \to D(\shH_X)$ gives an equivalence of categories $D(Y)\simeq \sC \subset D(\shH)$ where $\sC$ is HP-dual category in (\ref{sod:H_X}).
\end{definition}

The simplest examples of HP-dual are given by dual linear projective subspaces.

\begin{example}[Dual linear subspaces] \label{ex:dual-linear-sections} Let $\PP^{l-1} \simeq L \subset \PP V^*$ be a projective linear subspace, $1 \le l \le N$, with the Beilinson decomposition (\ref{sod:beilinson}). Consider the universal hyperplane $\shH_L \subset \PP V \times L$ for $L$. The projection from $\shH_L$ to $\PP V^*$ has fiber dimension $l-2$ except from the base loci $L^\perp$, over which the fibers are the whole $L$. Recall the loci $L^\perp \subset \PP V$ is called \textbf{dual linear subspace} of $L$, and is defined by 		
$$L^\perp := \{ x \in \PP V ~|~ s(x) = 0 \quad\text{for any}\quad s \in L\}.$$
Therefore we have diagram
\begin{equation*}
	\begin{tikzcd}
	L^\perp \times L \ar{d}{p} \ar[hook]{r}{j} & \shH_L \ar{d}{\pi} \ar[hook]{r}{\iota} & \PP(V) \times L \ar{ld}\\
	L^\perp \ar[hook]{r}         & \PP(V) 
	\end{tikzcd}
\end{equation*}
The projection $\pi: \shH_L \to \PP V$ is like a 'combination' of blowing up and projective bundle. It is well known from this geometry we have Orlov type results: $j_*\,p^* : D(L^\perp) \to D(\shH_L)$ is fully faithful, and there is a decomposition
	$$D(\shH_L) = \langle j_*\,p^* D(L^\perp), \pi^*D(\PP V) (0,1), \ldots ,  \pi^*D(\PP V) (0,l-1)\rangle.$$
Cf. \cite[Prop. 3.6]{RT15HPD}, \cite[Thm. 8.2]{Kuz07HPD} or Prop. \ref{prop:orlov_rel}. Notice the functor $j_*\,p^*$, therefore the whole decomposition, is $\PP V$-linear, since the Fourier-Mukai kernel of $j_*\,p^*$ is $\sO_{L^\perp \times_{\PP V} \shH_L} \in D(L^\perp \times_{\PP V} \shH_L)$. Therefor $L^\perp$ is the HP-dual of $L$ with respect to Beilinson decomposition.

\end{example}

The main result of HP-duality theory is the following. Assume the image of $X \to \PP V$ is not contained in any hyperplane (i.e. nondegenerate).
\begin{theorem}[HP-duality theorem, Kuznetsov {\cite[Thm. 6.3]{Kuz07HPD}}] \label{thm:HPD} Suppose $Y \to \PP V^*$ is HP-dual to $X \to \PP V$ with respect to (\ref{lef:X}). Then
	\begin{enumerate}[leftmargin=*]
		\item There exists a semiorthogonal decomposition
			\begin{equation}\label{lef:Y} D(Y) = \langle \shB^0(1-N), \shB^1(2-N), \cdots \shB^{N-2}(-1), \shB^{N-1} \rangle.
			\end{equation}
			with $0 = \shB^0 \subset \shB^1 \subset \cdots \subset \shB^{N-1} = \shA_0$ defined by (\ref{def:B^k}) \footnote{This type of decomposition is called a \textbf{dual Lefschetz decomposition}. 
	Taking $R\sHom_Y(-,\sO_Y)$, one can canonically obtain a Lefschetz decomposition from a dual one, and vice visa. Cf. also \S \ref{sec:duality}.}. In particular, $Y$ is smooth.
		\item For any pair of dual projective linear subspace $L\subset \PP V$ and $L^\perp \subset \PP V^*$, if they are \textbf{admissible}, i.e. the intersections
			$$X_{L^\perp} : = X \times_{\PP V} L^\perp \quad \text{and} \quad Y_L : = Y \times_{\PP V^*} L$$
			are of expected dimensions. Then there are semiorthogonal decompositions
				\begin{align*}
				D(X_{L^\perp}) & = \langle \sC_L, \shA_l(l),  \ldots, \shA_{i-1}(i-1) \rangle, \\
				D(Y_L) & = \langle \shB^1(2-l), \ldots, \shB^{l-2}(-1), \shB^{l-1}, \sC_L \rangle
				\end{align*}
			with a same triangulated category $\sC_L$ as the 'primitive' parts.
	\end{enumerate}
\end{theorem}

Notice the linear sections $X_{L^\perp}$ and $Y_L$ need not to be smooth\footnote{actually when they are singular, there is an equivalence of their 'singularity categories', cf. \cite[\S 7.5]{Kuz07HPD}.}; They are only required to be dimensional transverse. The HP-duality is a reflective relation: $X \to \PP V$ is HP-dual to $Y \to \PP V^*$ with respect to (the dual of) decomposition (\ref{lef:Y}), cf. \cite[Thm. 7.3]{Kuz07HPD} and \S \ref{sec:duality}.

The power of the theorems is at least twofold. First it produces interesting semiorthogonal decompositions of all (complete) linear sections of the algebraic varieties. Almost all known examples of semiorthogonal decompositions of algebraic varieties fit into the framework of HP-duality or its variants, cf. \cite{Kuz14SODinAG}. Second it gives striking relations between derived categories of different varieties. See the rich amount of literatures by Kuznetsov and others listed in the introduction for various applications of HP-duality theorems (and its variants).

\begin{remark} The classical projective dual of a variety always exists, but usually singular. On the contrary, the HP-dual of $f: X \to \PP V$ always exists as a category, but it is very hard for it to be a variety. However, in many important situations the HP-dual is very 'close' to an honest variety and we still have geometric applications of the HP-dual theorem (see discussions in \cite{RT15HPD, Kuz14SODinAG} etc, also Ex. \ref{Pf-Gr}, and \S \ref{sec:ncHPD}).

The relations between HP-dual variety and the classical projective dual variety is described by \cite[Theorem 7.9]{Kuz07HPD}: the critical value of the HP-dual variety $Y \to \PP V^*$ of $X$ is classical projective dual of $X$. This hints the following heuristic but instructive principle: {\it the HP-dual is usually (or at leas closely related to) the {\em (noncommutative crepant) resolutions of singularities} (\cite{Kuz08Lef, Vdb04, BLVdb10, BLVdb11, SVdb15}) of the classical projective duals}. See for example \cite{Kuz06HPDGr, BBF16}.

The other way to construct HP-duals (also closely related to the above principle) is to use the thread of ideas of Gauge Linear Sigma Models (GLSMs), Landau-Ginzburg (LG) models, and Variation of GIT quotients (VGIT), mentioned in the introduction. See the work of Ballard et al. \cite{BDF+} for a powerful and systematic approach to construct HP-duals as LG models, and see their later work \cite{BDF+14}, Halpern-Leistner \cite{H-L15}, and works by Addington, Donovan, Rennemo, Segal, Thomas \cite{ADS15, ST14, Ren15, ReEd16} for related constructions and generalizations. 
\end{remark}

\begin{example}[Pfaffian-Grassmannian duality, {\cite{Kuz06HPDGr, BC09, ADS15}}]\label{Pf-Gr} Let $W$ be a $\kk$-linear space of dimension $m$ with $m = 6,7$. Then $X=\Gr(2,W)$ with the Pl\"ucker embedding $X \to \PP(\wedge^2 W)$ has a natural Lefschetz decomposition $D(X) = \langle \shA_0, \shA_1(1), \ldots, \shA_{m-1}(m-1) \rangle$, where 
	\begin{align*}
		&\shA_0 = \ldots = \shA_{6} = \langle S^2 \sU, \sU, \sO \rangle, & \quad  \text{if} \quad m = 7, \\
		& \shA_0 = \shA_1 = \shA_{2} =  \langle S^2 \sU, \sU, \sO \rangle, ~\shA_3 = \shA_4 = \shA_5 = \langle \sU, \sO \rangle, & \quad \text{if} \quad  m = 6,
	\end{align*}
where $\sU$ is the tautological rank $2$ bundle on $\Gr(2,W)$, and $S^2 \sU$ is the symmetric product.
Then \cite{Kuz06HPDGr} shows $X$ with the above decomposition is HP-dual to a non-commutative resolution $(Y,\sR_Y)$ of the Pfaffian variety $Y=\Pf(2 \lfloor 2m/2 \rfloor - 2, W) \subset \PP(\wedge^2 W^*)$, where $\Pf(2k, W)$ is the loci inside $\PP(\wedge^2 W^*)$ where the skew-symmetric form has rank smaller than or equal to $2k$ . If $m=7$, and we take $L$ to be a generic $\PP^6 \subset \PP(\wedge^2 W^*)$, then we get two Calabi-Yau threefolds $X_{L^\perp}$ and $Y_L$, which are non-birational equivalent but are derived equivalent $D(X_{L^\perp}) \simeq D(Y_L)$. This is the first example of a pair of provable non-birational Calabi-Yau threefolds which are derived equivalent. If $m=6$, and $L$ is a generic $\PP^5 \subset \PP(\wedge^2 W^*)$, then we get a decomposition of Pfaffian cubic fourfold $Y_L$ with the 'primitive part' of the decomposition equivalent to the derived category of the $K3$ surface $X_{L^\perp}$.
\end{example}

\section{Main results}\label{sec:main}
The goal of this chapter to prove Thm. \ref{thm:HPDgen}, which generalizes HP-duality theorem to the situation when $(L,L^\perp)$ is replaced by any other HP-dual pair $(S,T)$. The strategy of proving Thm. \ref{thm:HPDgen} is first to use base-change to put everything in a common playground, and then play the game of mutations 
in \cite{RT15HPD}. 

\subsection{Preparations.}\label{sec:set up} Let us first set up certain notations. We start with a HP-dual pair $(X,Y)$, where $X$ is a smooth projective variety, with a morphism $f: X \to \PP V$, and a Lefschetz decomposition (\ref{lef:X}) with respect to the line bundle $\sO_X(1) = f^* \sO_{\PP V} (1)$ satisfying assumption $(\dagger)$, and $Y \to \PP V^*$ is the HP-dual, i.e. we have $g:Y \to \PP V^*$ a morphism from a projective variety $Y$, with an equivalence of category $D(Y) \simeq \sC$ (where $\sC$ is the HP-dual category defined by (\ref{sod:H_X})), given by a Fourier-Mukai kernel $\shE_X \in D(\shH_X \times_{\PP V^*} Y)$.

Now consider another HP-pair $(S,T)$. Let $S$ be a smooth projective, with a morphism $g:S\to \PP V^*$, and a Lefschetz decomposition 
	  \begin{equation} \label{lef:S}
            D(S) = \langle \shC_{0}, \shC_1(1), \ldots, \shC_{l-1}(l-1)\rangle, \quad \shC_0 \supset \shC_1 \supset \cdots \supset \shC_{l-1}
   	 \end{equation}
with respect to the line bundle $\sO_S(1) = g^* \sO_{\PP V^*}(1)$. We also assume $g:S \to \PP V^*$ satisfies assumption $(\dagger)$, i.e. $\dim_k \overline{g(S)} \ge 2$, and $l \le N -1$. As in Sec. \ref{sec:lef},  we define $\foc_k$ to be the right orthogonal of $\shC_{k+1}$ inside $\shC_{k}$ for $k = 0,1,\ldots, l-1$, then
	\begin{equation} \label{def:C_k}
	\shC_k = \langle \foc_k, \foc_{k+1}, \ldots, \foc_{l-1}\rangle.
	\end{equation}
Lem. \ref{lem:sod:A_0} holds for $S$, with $\shA_k$ (resp. $\foa_{k}$) for $0 \le k \le i-1$ replaced by $\shC_k$ (resp. $\foc_{k}$) for $0 \le k \le l-1$. We also define
    \begin{equation} \label{def:D^k}
        \shD^k =  \langle \gamma_0^* (\foc_0(1)), \ldots, \gamma_0^* (\foc_{k-1}(k))\rangle \subset \shC_0, \quad \text{for} \quad k = 1,2, 3, 4, \ldots
    \end{equation}
where $\gamma_0: \shC_0 \to D(S)$ is the inclusion functor and $\gamma_0^*: D(S) \to \shC_0$ is its left adjoint. See the lower diagram of Figure \ref{Figure:lf_2}. We set $\shD^k = 0$ if $k \le 0$, and $\shD^k = \shC_0$ if $k \ge l$.

Let $T \to \PP V$ be a HP-dual to $S \to \PP V^*$ with respect to (\ref{lef:S}). This means there exists $\shE_S \in D(T \times_{\PP V} \shH_S)$ where $\shH_S \subset \PP V \times S$ is the universal hyperplane section for $S$, such that the Fourier-Mukai functor $ \Phi_{\shE_S} : D(T) \to D(\shH_S)$ is fully faithful, and there exists a $\PP V$-linear semi-orthogonal decomposition
	\begin{equation}\label{sod:H_S}
	D(\shH_S) =  \langle \Phi_{\shE_S} (D(T)),~ D(\PP V)\boxtimes  \shC_1(1), \ldots, D(\PP V) \boxtimes \shC_{l-1}(l-1) \rangle.
	\end{equation}

	From Thm. \ref{thm:HPD} we know $D(Y)$ (resp. $D(T)$) can be built up from $\shB^k$'s (resp. $\shD^k$'s). But we will not use this result. In fact we will reprove this in our main result Thm. \ref{thm:HPDgen}. By Lem. \ref{lem:serre}, we also have decomposition
	\begin{equation*}
	D(\shH_S) =  \Big\langle  \langle S_{\shH_S}( D(\PP V)\boxtimes  \shC_k(k)\rangle_{k=1,\ldots, l-1}~,~\Phi_{\shE_S} (D(T)) \Big\rangle,
	\end{equation*}
	where  $S_{\shH_S} := S_{D(\shH_S)} $ is the Serre functor of $D(\shH_S)$. Since $\shH_S \subset \PP V \times S$ is a smooth divisor in $|\sO_{\PP V \times S}(1,1)|$, then $S_{\shH_S} =\otimes \omega_{\shH_S} [\dim \shH_S]$, $S_{D(S)} = \otimes \omega_S [\dim S]$. But $\omega_{\shH_S} = \omega_{\PP V}\boxtimes \omega_S (1,1)|_{\sH_S}$, and triangulated subcategories are stable under shift, so we have
	$$  S_{\shH_S}( D(\PP V)\boxtimes  \shC_k(k)) = D(\PP V) \boxtimes S_{D(S)}(\shC_k)(k+1), \quad \text{for} \quad k=1, 2, \ldots, l-1.$$

	If the decomposition of $D(S)$ is rectangular $\shC_{0} = \ldots = \shC_{l-1} = \shC$, then $S_{D(S)} (\shC) = \shC(-l)$. This motivates us to introduce the following notation,	
	\begin{equation}\label{def:C^L} 
	\shC_{k}^L :=  S_{D(S)} (\shC_k) \otimes \sO_S(l), \quad \text{for} \quad k = 1, 2, \ldots, l-1.
	\end{equation}
Then $S_{D(S)}(\shC_k) = \shC^L_k (-l)$, and
	\begin{equation} \label{sod:serre:H_S}
	D(\shH_S) =  \langle  D(\PP V) \boxtimes \shC_{1}^L (2-l), \ldots, D(\PP V) \boxtimes\shC_{l-1}^L ,~\Phi_{\shE_S} (D(T)) \rangle,
	\end{equation}
 Notice $\shC_{k}^L \simeq \shC_k$ and $\shC_{1}^L \supset \shC_{2}^L \supset \ldots \supset \shC_{l-1}^L$. Using $\shC^L_k$ and Serre functor, we have following semiorthogonal decompositions for $D(S)$ for any $k=1,2,\ldots, l-1$,
	\begin{align}\label{sod:S:C^L}
	\begin{split} D(S) & = \langle \shC_0, \shC_1(1), \ldots, \shC_{k-1}(k-1), \shC_k(k), \ldots, \shC_{l-1}(l-1) \rangle	\\
		& = \langle S_{D(S)}(\shC_k(k)), \ldots, S_{D(S)}(\shC_{l-1}(l-1)), \shC_0, \shC_1(1), \ldots, \shC_{k-1}(k-1)\rangle \\
		& = \langle \shC_k^L(k-l+1), \ldots, \shC_{l-1}^L, \shC_0(1), \shC_1(2), \ldots, \shC_{k-1}(k)\rangle.
	\end{split}
	\end{align}	

\subsection{Base change} \label{sec:base-change}
We will be interested in the derived categories of the fibered products
		$$X_T : = X \times_{\PP V} T , \quad \text{and}\quad Y_S : = Y \times_{\PP V^*} S.$$
	In order to compare derived categories of $X_T$ and $Y_S$, we use the trick of base-change.

Recall $Q \subset \PP V \times \PP V^*$ is the incidence quadric. For two maps $f:X \to \PP V$ and $q:S \to \PP V^*$, we will use $Q(X,S) \subset X \times S$ to denote the incidence quadric inside $X \times S$, i.e. $Q(X,S): = X \times_{\PP V} Q \times_{\PP V^*} S$.

\begin{definition}\label{def:admissible} Two pairs of morphisms $(X \to \PP V, \,Y \to \PP V^*)$, and $(S \to \PP V^*, \,T \to \PP V)$ are called \textbf{admissible} if the morphism $X \to \PP V$ (resp. $S \to \PP V^*$) considered as a base-change is \textbf{faithful} (Def. \ref{def:bc}) with respect to $(\shH_S, T )$ (resp. $(\shH_X, Y)$), and $Q(X,S) \ne X \times S$.
\end{definition}

\begin{remark} This is a technical way to say the two pairs intersect properly, and the condition is satisfied typically for Cohen-Macaulay varieties inside smooth variety which have intersections of expected dimensions (cf. \cite[Prop. A.1]{Add}). The Lem. \ref{lem:admissible pairs} will give certain criteria for admissibility. We use this technical definition for the purpose of later generalizations to noncommutative cases.  As we will see it holds for almost all examples of HP-duals in considerations. Similar to the discussion in Kuznetsov's \cite{Kuz06Hyp}, we also expect these conditions which guarantee certain 'flatness' can be removed if we consider derived intersection $X\times^R_{\PP V} T$ and $Y \times^R_{\PP V^*} S$ and the corresponding $dg$-categories.
\end{remark}

A map $f: X \to \PP V$ is called \textbf{non-degenerate} if the image $f(X)$ is not contained in any proper linear subspace. This is equivalent to the natural map $V^* \to H^0(X,\sO_X(1))$ being an inclusion. If $f: X \to \PP V$ (resp. $q: S \to \PP V^*$) is degenerate, we will use $L_X \subset \PP V$ (resp. $L_S \subset \PP V^*$) to denote the minimal linear subspace containing $f(X)$ (resp. $q(S)$). 

\begin{lemma} \label{lem:H_flat} Suppose $X$ is integral, and $f: X \to \PP V$ is non-degenerate. Then the universal hyperplane $\shH_X$ is flat over the dual projective space $\PP V^*$.
\end{lemma}

\begin{proof} 
$\shH_X \subset X \times \PP V^*$ is an effective Cartier divisor, and $X \times \PP V^*$ is flat over $\PP V^*$. The non-degeneracy condition for $f$ exactly says for any $s \in \PP V^*$, the corresponding section $s \ne 0 \in \Gamma(X, \sO_X(1))$. Since $X$ is integral, this implies the divisor $\shH_{X} \subset X \times \PP V^*$ is cut out locally at any point $x \in \shH_{X} \times_{\PP V^*} \{s\}$ on the fiber $X_s: = X \times \{s\}$ by a non-zerodivisor $s_x \in \sO_{X_s,x} = \sO_{X \times \PP V^*} \otimes_{\sO_{\PP V^*, s}} k(s)$. By \cite[Def.-Prop. 1.11]{Kol} or \cite[Lem. 9.3.4]{FAG}, $\shH_X$ is a relative effective Cartier divisor over $\PP V^*$, i.e. it is flat over $\PP V^*$.
\end{proof}

\begin{lemma}[Criterion for admissibility of pairs] \label{lem:admissible pairs} 
Let $f: X \to \PP V$, $g: Y \to \PP V^*$ and $q: S \to \PP V^*$, $p: T \to \PP V^*$ be two pairs of morphisms.
\begin{enumerate} [leftmargin = *]
\item If both $f: X \to \PP V$ and $q: S \to \PP V^*$ are non-degenerate, then the two pairs are admissible if $X_T$ and $Y_S$ are of expected dimensions.
\item If $f: X \to \PP V$ is non-degenerate.
	\begin{enumerate} [leftmargin = 0.5 cm]
	\item If $Q(S,T) = S \times T$  (for example $S = L$, $T = L^\perp$ dual linear subspaces). Then the two pairs are admissible if $X_T$ and $Y_S$ are of expected dimensions.
	\item If $Q(S,T) \neq S \times T$, and $Q(X_T, S)$ is a divisor of $X_T \times S$ \footnote{by which we mean the quadric $Q(X_T, S)$ is of pure codimension one inside the equidimensional variety $X_T \times S$. This condition is to guarantee the position of the intersection $X_T$ is not too bizarre, i.e. the very unlikely situation when $X$ and $T$ intersects exactly inside $L_S^\perp$, will not happen.}. Then the two pairs are admissible if $X_T$ and $Y_S$ are of expected dimensions.
	\end{enumerate}
\item If $q: S \to \PP V^*$ is non-degenerate, the statements are exactly symmetric to (2).
\item If both $f: X \to \PP V$ and $q: S \to \PP V^*$ are degenerate. Then we assume $L_X \not \subset L_S^\perp$ \footnote{which is equivalent to $L_S \not \subset L_X^\perp$, which is also equivalent to $Q(X,S) \neq X \times S$.}. If one of the following is satisfied: whether $Q(X,Y) = X \times Y$, or $Q(X,Y) \ne X \times Y$ but $Q(X,Y_S) \subset X \times Y_S$ is a divisor; And similarly one of the following is satisfied: whether $Q(S,T) = S \times T$, or $Q(S,T) \ne S \times T$ but $Q(X_T,S) \subset X_T \times S$ is a divisor. Then the two pairs are admissible if $X_T$ and $Y_S$ are of expected dimensions.
\end{enumerate}
\end{lemma}

\begin{proof} The key is to check six squares are exact cartesian. For the first set of squares
	$$
	\begin{tikzcd}
	X_T \ar{d} \ar{r} &  T \ar{d}\\
	X \ar{r}{f}          &\PP V
	\end{tikzcd}
	\qquad\quad
	\begin{tikzcd}
	\shH \ar{d} \ar{r} & \shH_S \ar{d}\\
	X \ar{r}{f}          &\PP V
	\end{tikzcd}
	\qquad\quad
	\begin{tikzcd}
	Q(X_T,S) \ar{d} \ar{r} &  Q(T,S) \ar{d}\\
	X \ar{r}{f}          &\PP V
	\end{tikzcd}	
	$$
The first square is exact cartesian if $X_T$ are of expected dimension by Lem. \ref{lemma-faithful-base-change}, (2). The second is exact cartesian if $L_X \not \subset L_S^\perp$. For last square, if $Q(T,S) = T \times S$, then $Q(X_T,S) = X_T \times S$, and the square is exact cartesian if $X_T$ are of expected dimension. If $q:S \to \PP V^*$ is non-degenerate, then the family $\shH_S \to \PP V$ is a flat family, and so is map $Q(T,S) \to T$. Since we have	
	$$
	\begin{tikzcd}
	Q(X_T,S) \ar{d}{f_Q} \ar{r} & X_T \ar{d}{f_T} \ar{r} & X \ar{d}{f}\\
	Q(T,S)  \ar{r}   &T \ar{r}          &\PP V,
	\end{tikzcd}
	$$	
hence the ambient square is exact cartesian by Lem. \ref{lem:3squares}. If the $Q(T,S)$ is a divisor, then consider the squares 
	$$
	\begin{tikzcd}
	Q(X_T,S) \ar{d} \ar[hook]{r}&  X \times Q(T,S) \ar{d} \ar{r} & Q(T,S) \ar{d}\\
	X   \ar[hook]{r}{\Gamma_f}   &X \times \PP V \ar{r}          &\PP V
	\end{tikzcd}	
	$$
where $\Gamma_f$ is the graph embedding of $X$. The right square is exact cartesian since the projection $X \times \PP V \to \PP V$ is a smooth map. For left square, the condition $Q(X_T,S)$ is a divisor guarantees it is of expected dimension, and since $\Gamma_f$ is a locally complete intersection embedding, $X \times Q(T,S)$ is Cohen-Macaulay, the left square is also exact cartesian by (3) of Lem. \ref{lemma-faithful-base-change}. From Lem. \ref{lem:3squares}, the ambient square is exact cartesian. The argument is similar for the second set of squares:
	$$
	\begin{tikzcd}
	Y_S \ar[d] \ar[r] &  Y \ar[d]\\
	S \ar{r}{q}          &\PP V^*
	\end{tikzcd}
	\qquad\quad
	\begin{tikzcd}
	\shH \ar[d] \ar[r] & \shH_X \ar[d]\\
	S \ar{r}{q}          &\PP V^*
	\end{tikzcd}
	\qquad\quad
	\begin{tikzcd}
	Q(X,Y_S) \ar[d] \ar[r] &  Q(X,Y) \ar[d]\\
	S \ar{r}{q}          &\PP V^*
	\end{tikzcd}
	$$ \end{proof}

Back to the situation considered in Sec. \ref{sec:set up}. If the two HP-dual pairs $(X,Y)$ and $(S,T)$ are admissible, then we can base-change the $\PP V^*$-linear (resp. $\PP V$-linear) decomposition (\ref{sod:H_X}) of $\shH_{X}$ (resp. decomposition (\ref{sod:H_S}) of $\shH_S$) to $S$ (resp. $X$), and get a decomposition of $\shH_X \times_{\PP V^*} S$ (resp. $X \times_{\PP V} \shH_S$).  Notice the schemes $\shH_X \times_{\PP V^*} S$ and $X \times_{\PP V} \shH_S$ are isomorphic, and both isomorphic to the subscheme $\shH : = Q(X,S) \subset X \times S$ defined by incidence relation $\{(x, s)~|~ s(x)= 0\}$. Denote by $i_{\shH}: \shH \to X\times S$ the inclusion.

 \begin{proposition}\label{prop:base-change} If the two HP-dual pairs $(X,Y)$ and $(S,T)$ are admissible, then the funtors $\Phi_{\shE_{X}|S}:D(Y_S) \to D(\shH)$ and $\shA_r(r)\boxtimes D(S) \subset D(X\times S) \xrightarrow{i_{\shH}^*} D(\shH)$ ($1\le r \le i-1$), are all fully faithful, and give rise to a semiorthogonal decomposition
	\begin{equation}\label{sod:H for X}
	D(\shH) = \langle \shD_{Y_S}, \shA_1(1)\boxtimes D(S), \ldots, \shA_{i-1}(i-1)\boxtimes D(S) \rangle .
	\end{equation}
where $\shD_{Y_S} =  \Phi_{\shE_{X}|S}(D(Y_S))\simeq D(Y_S)$. Similarly, the functors $\Phi_{\shE_{S}|X}:D(X_T) \to D(\shH)$, $D(X)\boxtimes \shC_r(r) \subset D(X\times S) \xrightarrow{i_{\shH}^*} D(\shH)$, and $D(X)\boxtimes \shC^L_r(r)\subset D(X\times S) \xrightarrow{i_{\shH}^*} D(\shH)$, where $1\le r \le l-1$, are all fully faithful, and give rise to semiorthogonal decompositions
	\begin{align} 
	D(\shH) &= \langle \shD_{X_T}, ~D(X) \boxtimes \shC_1(1) , \ldots, D(X)\boxtimes \shC_{l-1}(l-1) \rangle,\label{sod:H for S}\\
	 &=  \langle D(X) \boxtimes \shC^L_1(2-l) , \ldots, D(X)\boxtimes \shC^L_{l-1}, ~ \shD_{X_T} \rangle,\label{sod:H for S:Serre}
	\end{align}
where $\shD_{X_T} =  \Phi_{\shE_{S}|X}(D(X_T))\simeq D(X_T)$ denotes the image.\footnote{Note we use the same symbol $\shA_k(k)\boxtimes D(S)$, $D(X)\boxtimes \shC_k$, etc to denote their full faithful images under the pullback functor $i_{\shH}^*$, for simplicity of notations.}
\end{proposition}
\begin{proof} Consider the $\PP V^*$-linear decomposition (\ref{sod:H_X}). Since the base-change $S \to \PP V^*$ is faithful for $\shH_X \to \PP V^*$, then from Prop. \ref{prop:bcsod} it induces a $S$-linear semiorthogonal decompositions of $D(\shH_X \times_{\PP V^*} S ) = D(\shH)$. It is clear that under base change the fully faithful $\PP V^*$-linear functor $\shA_r(r)\boxtimes D(\PP V^*) \subset D(X\times \PP V^*) \xrightarrow{i_{\shH_X}^*} D(\shH_X)$ induces fully faithful $S$-linear functor $\shA_r(r)\boxtimes D(S) \subset D(X\times S) \xrightarrow{i_{\shH}^*} D(\shH)$, where $1\le r \le i-1$, and the images of these coincide with the components obtained from base-change in Prop. \ref{prop:bcsod}, i.e. $[i_{\shH_X}^*(\shA_r(r)\boxtimes D(\PP V^*))]_{S} = i_{\shH}^* (\shA_r(r)\boxtimes D(S))$. For the first component, since $S\to \PP V^*$ is faithful for the pair $(\shH_X, Y)$, hence the fully faithful embedding $\Phi_{\shE_{X}}:D(Y) \to D(\shH_X)$ induces a fully faithful embedding $\Phi_{\shE_{X}|S}:D(Y_S) \to D(\shH)$ by Prop. \ref{prop:bcFM}, where the Fourier-Mukai kernel is given by $\shE_{X}|S := \phi ^* \shE_X$, where $\phi$ is the natural map $Y_S \times_S \shH \to Y \times_{\PP V^*} \shH_X$. It is not hard to see the image of the Fourier-Mukai transform coincide with the first component of the induced decomposition from  (\ref{sod:H_X}), cf. Thm. 6.4, \cite{Kuz11Bas}.

Similarly, from $\PP V$-linear decomposition (\ref{sod:H_S}) and its mutated decompositions using Serre functor, we have the following decompositions induced from base-change along $X \to \PP V$: for all $k = 0, 1, \ldots, l-1$,
	\begin{align} 
	\begin{split}
	D(\shH) =& \big\langle D(X) \boxtimes \shC^L_{k+1}(k+2-l) , \ldots, D(X)\boxtimes \shC^L_{l-1}, \\
		& ~ \shD_{X_T}~, D(X)\boxtimes \shC_1(1), \ldots,\ D(X)\boxtimes\shC_{k}(k) \big \rangle,\label{sod:H:k} 
	\end{split}
	\end{align}
The theorem then follows.
\end{proof}

\noindent\emph{Notations.} Denote $\shD_{X_T} =  \Phi_{\shE_{S}|X}(D(X_T))$ and $\shD_{Y_S} =  \Phi_{\shE_{X}|S}(D(Y_S))$ are the fully faithful images, and let $i_T:  \shD_{X_T} \hookrightarrow D(\shH)$ and $i_S : \shD_{Y_S} \hookrightarrow D(\shH)$ be the inclusions, i.e. 
\begin{equation*}
 	\begin{tikzcd}
		 D(X_T) \ar{r}{\Phi_{\shE_{S}|X}}[swap]{\sim}  &\shD_{X_T} \ar[hook]{r}{i_T} &  D(\shH)  & \shD_{Y_S} \ar[swap, hook']{l}{i_S} & D(Y_S) \ar{l}{\sim}[swap]{\Phi_{\shE_{X}|S}} .
	\end{tikzcd}
 \end{equation*}
Denote by $\pi_T = i_T^*: D(\shH) \to \shD_{X_T}$ (resp. $\pi_S = i_S^*:D(\shH) \to \shD_{Y_S}$) the left adjoint functor of the inclusion $i_T$ (resp. $i_S$), and by $i_T^!$ (resp $i_T^!$) the right adjoint. Note $\pi_T$ (resp. $\pi_S$) is given by the left mutations through the orthogonal of $\shD_{X_T}$ in (\ref{sod:H for X}) (resp. of $\shD_{Y_S}$ in (\ref{sod:H for S})). 

 %
 %
 
\subsection{Main result}

The base-change Prop. \ref{prop:base-change} enables us to compare $D(X_T)$ and $D(Y_S)$ inside $D(\shH)$. Notice theirs orthogonal in $D(\shH)$ are 'linear', given by box-tensor categories $i_{\shH}^*(\shA_{\alpha}(\alpha) \boxtimes \shC_{\beta}(\beta))$. We will write $\shA_{\alpha}(\alpha) \boxtimes \shC_{\beta}(\beta)$ for simplicity. The situation is illustrated in Figure \ref{Fig_chessboard}, which is direct analogue of the diagram in \cite{RT15HPD}. Every square of the chessboard, which we call a \emph{box}, correspond to a box-tensor $\shA_{\alpha}(\alpha) \boxtimes \shC_{\beta}(\beta)$. Then everything reduces to playing the game of mutations 
on this 'chessboard'.

Later we will consider the extended 'chessboard', Figure \ref{Figure:l<=i} and Figure \ref{Figure:l>=i}, where not only the orthogonal components of decomposition (\ref{sod:H for X}) 
but also of mutated decompositions (\ref{sod:H for S:Serre}) and (\ref{sod:H:k}) 
are all indicated in the diagram. Notice in these 'chessboards', we draw every boxes as if they are of the same size. But actually this is only for simplicity of drawing, and in reality the boxes are of different sizes. For example, a more realistic version of Figure \ref{Fig_chessboard} is illustrated in Figure \ref{Fig_real}.

\begin{figure}
\begin{center}
\includegraphics[height=3.2in]{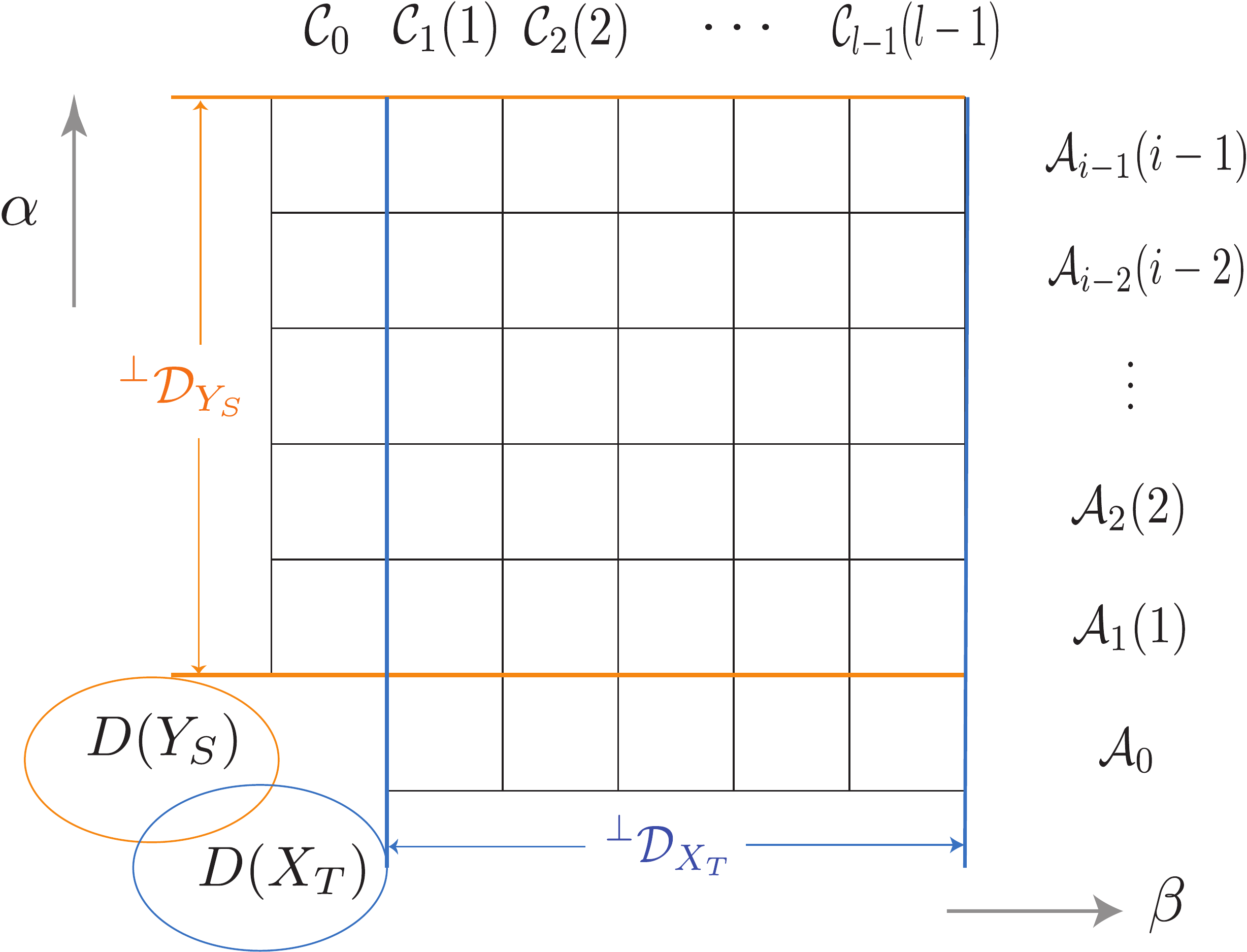}
\caption{The 'chessboard' which describes the 'differences' of the embeddings of $D(X_T)$ and $D(Y_S)$ into $D(\shH)$ (in the case $l = i =6$). The horizontal (resp. vertical) direction corresponds to $\sO_S(1)-$ (resp. $\sO_X(1)-$) action. Each square of the 'chessboard' corresponds to a fully faithful subcategory $i_{\shH}^*(\shA_{\alpha}(\alpha) \boxtimes \shC_{\beta}(\beta)) \subset D(\shH)$,  the factor $\shA_{\alpha}(\alpha)$, $\alpha \in [0,i-1]$ (resp. $\shC_{\beta}(\beta)$, $\beta \in [0, l-1]$) of which is indicated on the right of (resp. above) the diagram.
} \label{Fig_chessboard}
\end{center}
\end{figure}

\begin{figure}
\begin{center}
\includegraphics[height=2.4in]{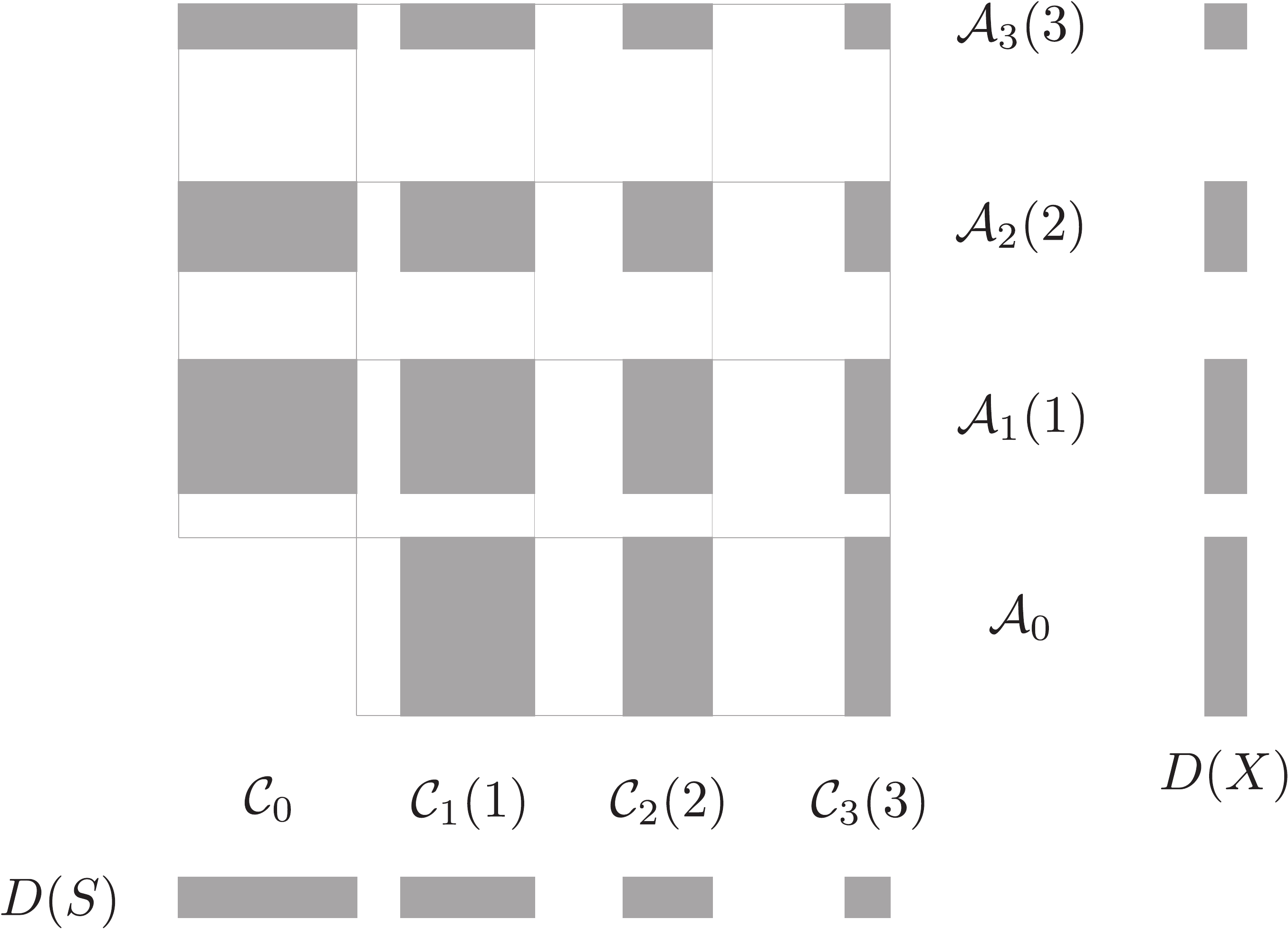}
\caption{A more realistic picture for Figure \ref{Fig_chessboard} in the case $l=i=4$. 
The shaded regions inside each boxes of the 'chessboard' indicate the 'real sizes' of the boxes $\shA_\alpha(\alpha) \boxtimes \shC_{\beta}(\beta)\subset D(\shH)$.} \label{Fig_real}
\end{center}
\end{figure}

\newpage
\begin{theorem}[HP-duality theorem for general sections] \label{thm:HPDgen} If the two HP-dual pairs 
	$$X \to \PP V \quad \text{HP-dual to} \quad Y \to \PP V^*$$
and 
	$$S \to \PP V^* \quad \text{HP-dual to} \quad T \to \PP V$$
with respect to Lefschetz decompositions (\ref{lef:X}) and (\ref{lef:S}) are \textbf{admissible} (Def. \ref{def:admissible}). Then
\begin{enumerate}[leftmargin=*]
\item We have semiorthogonal decompositions
	\begin{align}
	D(Y) & = \langle \shB^1 (2-N), \cdots, \shB^{N-2}(-1),\shB^{N-1} \rangle, \quad \shB^1 \subset \shB^2 \subset \cdots \subset\shB^{N-1} = \shA_0 \label{sod:Y},\\
	D(T) & = \langle \shD^1(2-N), \cdots, \shD^{N-2}(-1), \shD^{N-1} \rangle, \quad \shD^1 \subset \shD^2 \subset \cdots \subset\shD^{N-2} = \shC_0  \label{sod:T},
	\end{align} 
where $\shB^k$ and resp. $\shD^k$ are defined by (\ref{def:B^k}) and resp. (\ref{def:D^k}).
\item There exists a full triangulated subcategory $\sE \subset D(\shH)$, such that the restrictions of $\pi_T: D(\shH) \to \shD_{X_T}$ (resp. $\pi_S: D(\shH) \to \shD_{Y_S} $) to $\sE$ and $(\shA_{k}(k) \boxtimes \shD^k)|_{\shH}$ (resp. $\sE$ and $(\shB^k \boxtimes \shC_{k}^L(k+1-l) )|_{\shH}$) are fully faithful for all $ k \in \ZZ$, and their images give rise to semiorthogonal decompositions
\begin{align}
		D(X_T) \simeq \shD_{X_T} &= \langle \sE, ~\shA_1(1)\boxtimes \shD^1, \shA_2(2) \boxtimes \shD^2, \ldots, \shA_{i-1}(i-1) \boxtimes \shD^{i-1} \rangle. \label{sod:X_T} \\
		D(Y_S )\simeq \shD_{Y_S} & =  \langle \shB^1 \boxtimes \shC_{1}^L(2-l) , \shB^2 \boxtimes \shC_{2}^L(3-l), \ldots, \shB^{l-1}\boxtimes \shC_{l-1}^L, ~\sE \rangle. \label{sod:Y_S}
\end{align}
where recall $\shC^L_k = S_{D(S)}(\shC_k) \otimes \sO_S(l) \simeq \shC_k$. If $D(S)$ is rectangular, then $\shC^L_k = \shC_k$.
\end{enumerate}
\end{theorem}

\begin{figure}
\begin{center}
\includegraphics[height=2.8in]{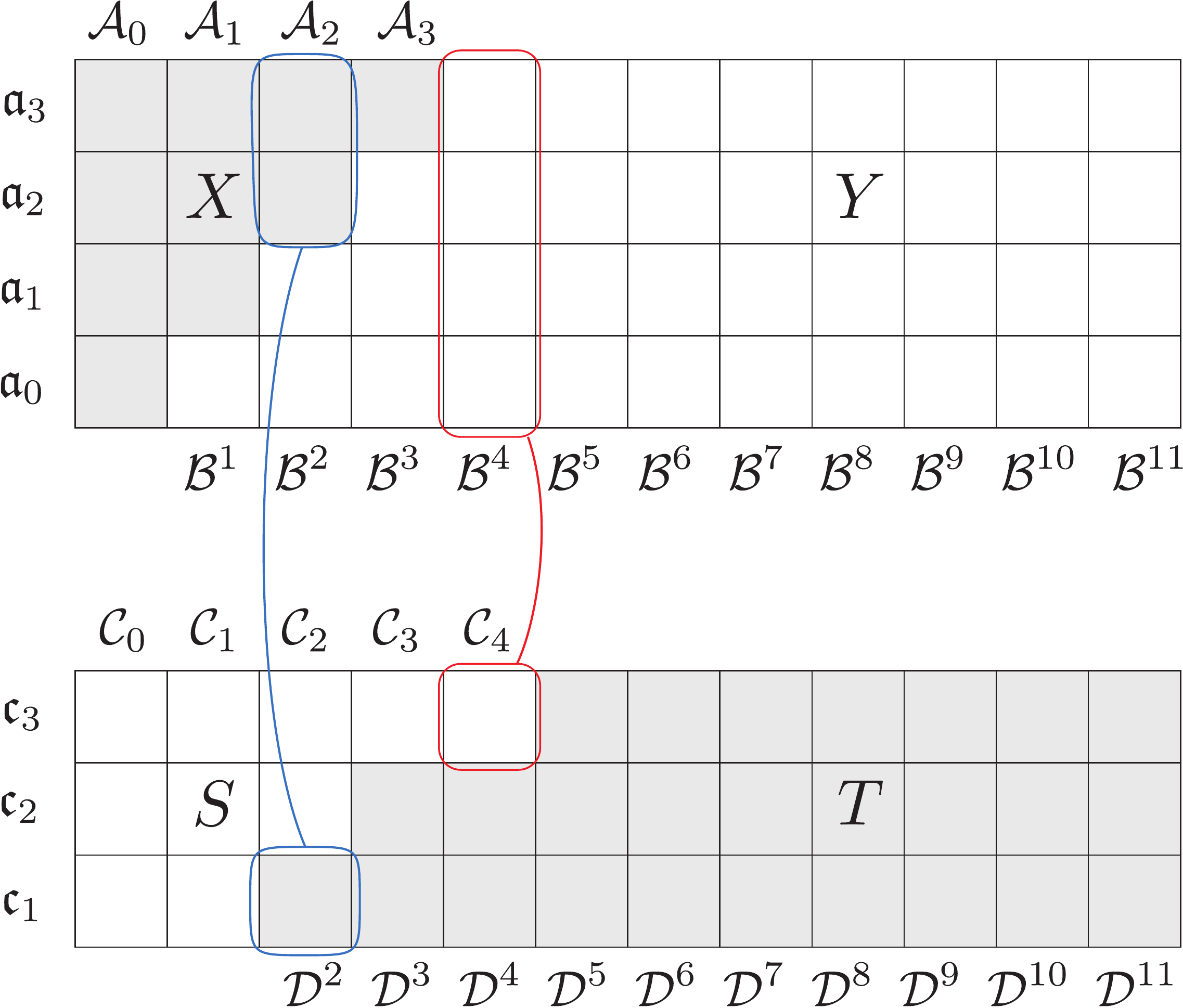}
\caption{Lefschetz decompositions for $D(X)$ and $D(S)$, and the dual decompositions for their HP-duals $D(Y)$ and $D(T)$ (in the figure $N=12$, $i=4$, $l=5$). 
Then Thm. \ref{thm:HPDgen} implies the ambient part of $D(X_T)$ (resp. $D(Y_S)$) is generated by box-tensors of the pairs of categories inside shaded (resp. unshaded) region connected by vertical lines (after twist). In the figure, two vertical lines can be drawn connecting boxes of shaded regions (e.g blue line), and therefore the ambient part of $D(X_T)$ is $\langle \shA_2 \boxtimes \shD^2(2), \shA_3 \boxtimes \shD^3(3)\rangle$. Similarly four vertical lines can be drawn connecting boxes of unshaded regions (e.g red line), therefore the ambient part of $D(Y_S)$, generated by the corresponding four terms after twist, is $\langle \shB^1 \boxtimes \shC_1 (-3), \shB^2 \boxtimes \shC_2 (-2), \shB^3 \boxtimes \shC_3 (-1),  \shB^4  \boxtimes  \shC_4\rangle$.
}\label{Figure:lf_2}
\end{center}
\end{figure}

The theorem takes a more symmetric than linear section case and is quite easy to remember: for example, the 'ambient' part of $D(X_T)$, if written suggestively as $\sum (\shA_k \boxtimes \shD^k) \otimes \sO_X(k)$, are simply generated by box-tensors $\shA_k \boxtimes \shD^k$'s of all the components $\shA_k \subset D(X)$ and $\shD^k \subset D(T)$ with the same index $k$, twisted by $\sO_{X_T}(k)$. Notice $\shA_k \boxtimes \shD^k$ can only be non-zero if $k  \in [1, i-1]$, hence we can take sum over all $k \in \ZZ$, with the range automatically being taken care of.
Similarly for $D(Y_S)$. 
The situation are intuitively illustrated in Figure \ref{Figure:lf_2}.

\begin{remark} Totally symmetrically and using the exactly the same method, we have
	\begin{align*} D(X_T) & = \Big\langle \big\langle (\shA_k^L \boxtimes \shD^k) \otimes \sO_{X_T}(k) \big\rangle_{k=1, \ldots, i-1}, \, \sE_{X_T}' \Big\rangle\\
		D(Y_S) & = \Big\langle \sE_{Y_S}' ,\, \big\langle  (\shB^k \boxtimes \shC_k) \otimes \sO_{Y_S}(1-l+k) \big\rangle_{k=1, \ldots, l-1} \Big\rangle,
	\end{align*}
and a derived equivalence $\sE_{X_T}' \simeq \sE_{Y_S}'$. The $\sE_{X_T}'$ and $\sE_{Y_S}'$ are equivalent to the privious ones given in the theorem by Serre functors. 

\end{remark}

\begin{remark}[Fourier-Mukai kernels]\label{rmk:FM}We omit the Fourier-Mukai functors in the expressions of the theorem for simplicity of notations. They can be described explicitly as follows:

\medskip\noindent $(1)~$ The items appearing in the decomposition (\ref{sod:Y}) of $D(Y)$ are actually
		$$\Phi^L_{\shE_X} ((\shB^k \boxtimes  \sO_{\PP V^*}(1 + k -N) )|_{\shH_X}) = \Phi^L_{\shE_X} (\shB^k|_{\shH_X})  \otimes \sO_{Y} (1+k-N), $$
		where $\Phi^L_{\shE_X} : D(\shH_X) \to D(Y)$ is the $\PP V^*$-linear left adjoint functor of the fully faithful functor $\Phi_{\shE_X}: D(Y) \to D(\shH_X)$, $k \in [1,N-1]$, and $(\shB^k \boxtimes  \sO_{\PP V^*}(1 + k -N))|_{\shH_X} \subset (\shA_0 \boxtimes D(\PP V^*))|_{\shH_X} \subset D(\shH_X)$. Similarly the items in the decomposition (\ref{sod:T}) of $D(T)$ are actually
		$$\Phi^L_{\shE_S} ((\sO_{\PP V}(1 + k -N)  \boxtimes \shD^k)|_{\shH_S}) = \Phi^L_{\shE_S} (\shD^k|_{\shH_S})  \otimes \sO_{T} (1+k-N), $$ 
		where $\Phi^L_{\shE_S} : D(\shH_S) \to D(T)$ is the $\PP V$-linear left adjoint functor of the fully faithful functor $\Phi_{\shE_S}: D(T) \to D(\shH_S)$, $k \in [1,N-1]$, and $(\sO_{\PP V}(1 + k -N) \boxtimes \shD^k)|_{\shH_X} \subset 
		D(\shH_S)$.
		
\medskip\noindent $(2)~$ The decomposition (\ref{sod:X_T}) of $D(X_T)$ is 
		\begin{align*} D(X_T)  & = \big\langle \sE_{X_T} ,~ \langle \Phi^L_{\shE_{S}|X} ((\shA_{k}(k) \boxtimes \shD^k)|_{\shH}) \rangle _{k=1,2,\ldots} \big\rangle \\
		& = \big\langle \sE_{X_T} ,~ \langle ( \shA_k \boxtimes \Phi^L_{\shE_S} (\shD^k) )|_{X_T} \otimes \sO_{X_T}(k)\rangle_{k = 1,2,\ldots}\big\rangle,
		\end{align*}
		where $\shE_{S}|X = f^* \shE_S$ is the pullback of $\shE_{S}$ along base-change $f: X \to \PP V$, and $\Phi^L_{\shE_{S}|X}$ is the left adjoint of the $X$-linear functor $\Phi_{\shE_{S}|X}$, $\sE_{X_T}$ is the fully faithful image of $\sE$. The last equality follows from the faithfulness of the base-change $f: X\to \PP V$, and that $\PP V$-linear ($X$-linear) functor intertwines with base-change $f^*$. Therefore $\Phi^L_{\shE_{S}|X}$ respects the $X$-linear structure, and we can regard $\shA_k \boxtimes \shD^k$ as coming from the product  $X \times T$. Similarly, decomposition (\ref{sod:Y_S}) of $D(Y_S)$ is 
		\begin{align*} D(Y_S)  & = \big\langle  \langle \Phi^L_{\shE_{X}|S} ((\shB^k \boxtimes S_{D(S)}(\shC_k)(1+k))|_{\shH}) \rangle _{k=1,2,\ldots}, ~\sE_{Y_S}  \big\rangle \\
		& = \big\langle  \langle ( \Phi^L_{\shE_X} (\shB^k) \boxtimes S_{D(S)}(\shC_k)) |_{Y_S} \otimes \sO_{Y_S}(1+k)\rangle_{k = 1,2,\ldots}, ~\sE_{Y_S}\big\rangle.
		\end{align*}
		This gives the formulation of {\em Main theorem} in the introduction.
		
\medskip\noindent $(3)~$ The Fourier-Mukai functor $D(X_T) \to D(Y_S)$ which induces $\sE_{X_T} \xrightarrow{\sim} \sE_{Y_S}$ are
		$$ \Phi^L_{\shE_{X}|S} \circ \Phi_{\shE_{S}|X}: D(X_T) \to D(Y_S),$$
		and the inverse $\sE_{Y_S} \xrightarrow{\sim} \sE_{X_T}$ is induced from
		$$ \Phi^L_{\shE_{S}|X} \circ \Phi_{\shE_{X}|S}: D(Y_S) \to D(X_T).$$
		The Fourier-Mukai kernel for $\Phi^L_{\shE_{X}|S}$ can be described explicitly when the singularities of $Q(X,Y)$ and $Q(S,T)$ are mild (e.g. when they are smooth, cf. Lem. \ref{lem:op}). However, the Fourier-Mukai kernel of the 'dual' of the left adjoint, which still induces the equivalence, has a very simple form, cf. e.g. \S \ref{sec:duality}.				
\end{remark}

%
%

\begin{remark}[Decomposition of HP-dual]\label{rmk:dual-decomposition} The second statement $(2)$ implies the first $(1)$ by taking $S = \PP V^*$, $T = \emptyset$ or $X = \PP V$, $Y  = \emptyset$. Let us first give more details and illustrate our strategy of proof in the case when $S= \PP V^*$, $T = \emptyset$. It is easy to see the admissibility conditions are always satisfied by any HP-dual pair $(X,Y)$ together with $(\PP V^*, \emptyset)$, and the base-change results are nothing but the fact $Y$ being HP-dual of $X$:
	$$D(\shH_X) = \langle \sC, \shA_1(1) \boxtimes D(\PP V^*), \ldots, \shA_{i-1}(i-1) \boxtimes D(\PP V^*) \rangle,$$ 
where $\sC \simeq D(Y)$ is given by a Fourier-Mukai functor, and the Orlov's results for the projective bundle $\sH_X \to X$:
	\begin{equation}\label{sod:DHX} D(\shH_X) = \langle D(X)(0, 2-N), D(X)(0, 3-N), \ldots, D(X)(0,0)\rangle.\end{equation}
Denote $\pi_Y$ the projection functor from $D(\shH_X)$ to $\sC$. Then our calculation in Sec. \ref{sec:fullyfaithful} shows the restriction of $\pi_Y$ to the subcategories $\shB^1(0,2-N), \ldots, \shB^{N-1}(0,0)$ are fully faithful, and their image remains a semiorthogonal sequence. This agrees with Prop. 5.10 in \cite{Kuz07HPD}. Then to show they generate $\sC$, we show the right orthogonal of their image inside $\sC$ is zero. For every element of this right orthogonal, all its components with respect to the refined decomposition of (\ref{sod:DHX}) using decomposition of $D(X)$ can be shown to be zero inductively following a 'staircase' pattern, c.f. Sec. \ref{sec:generation}. Hence our proof in a way answers the question of Kuznetsov after \cite[Prop. 5.10]{Kuz07HPD} of 'finding a direct proof' of the fact that $\shB^1(2-N), \ldots, \shB^{N-1}$ generate $\sC \simeq D(Y)$.
\end{remark}

\begin{remark}[When one is rectangular] \label{rmk:rectangular}
 If we assume decomposition (\ref{lef:S}) for $D(S)$ is \text{rectangular}, i.e. $\shC_k =\shC$ for all $k\in [0,l-1]$. Then $\shD^k  = 0$ for $k \le l-1$, and $\shD^{k}  = \shC_0$ for $k\ge l$, and the theorem reduces to
	\begin{align*}
	D(X_T)&= \langle \sE, \shA_{l}(l) \boxtimes \shC, \ldots, \shA_{i-1}(i-1)\boxtimes \shC \rangle, \\
	D(Y_S) &= \langle \shB^1 \boxtimes \shC(2-l), \ldots, \shB^{l-1} \boxtimes \shC, \sE \rangle.
	\end{align*}
This takes the similar form as Thm. \ref{thm:HPD} for linear sections.
\end{remark}

\begin{remark}[Kuznetsov's HP-duality theorem]\label{rmk:KuznetsovHPD} If $f: X \to \PP V$ is non-degenerate \footnote{In \cite{Kuz07HPD} it was assumed that $V^* \subset \Gamma(X, \sO_X(1))$, which is equivalent to $f: X \to \PP V$ is non-degenerate. But as we will see, it also holds for at least generic linear sections even when $f$ is degenerate.}, and $S =L$, $T = L^\perp$ be a pair of dual linear spaces. Then $Q(L^\perp, L) = L^\perp \times L$. Therefore $(X,Y)$ and $(L,L^\perp)$ are admissible as long as $X_{L^\perp}$ and $Y_L$ are of expected dimensions. Notice $D(S)$ is rectangular, hence our takes the form of Rmk. \ref{rmk:rectangular}, recovers Kuznetsov's HPD theorem Thm. \ref{thm:HPD} in non-degenerate case.

If $f: X \to \PP V$ is degenerate, and the image spans the linear space $L_X$. Then if $X_{L^\perp}$ and $Y_L$ are of expected dimensions, the condition $L \not\subset L_X^\perp$ is automatically satisfied since $L_X \not\subset L^\perp$. If $Q(X,Y) = X \times Y$, for example when $X,Y$ are dual linear spaces, then the pair are admissible. If $Q(X,Y) \neq X \times Y$, then we need one more condition: $Q(X,Y_L) \subset X \times Y_L$ is a divisor\footnote{It is very likely this condition can be removed when $X_{L^\perp}$ and $Y_L$ already have expected dimensions since we know $g_L(Y_L)=L \not\subset L_X^\perp$. But the authors do not know how to show this yet.}. Then the pairs are admissible. This is certainly satisfied for example if $Y_L$ are integral, and therefore is satisfied by a generic linear section $L$.
\end{remark}

		
 
The second statement of the theorem, hence the whole theorem, will be proved in the next two sections. Before that, let's first illustrate why it is true on Hochschild homology level.

\medskip \noindent \textbf{Hochschild homology and Pl\"ucker formula.} Prop. \ref{prop:base-change} allows us to compare invariants of $X_T$ and $Y_S$ inside $\shH$, and this can be easily done on Hochschild homology level. 

The Hochschild homology of a smooth projective variety $X$ is defined to be
	$$\HH_\bullet(X) : = \Ext^\bullet_{X \times X}(\Delta_* \sO_X, \Delta_* \omega_X),$$
where $\Delta: X \to X \times X$ is the diagonal embedding, $\omega_X$ is the dualizing sheaf. The Hochschild homology is also defined for any differential graded (dg) categories, hence also defined for triangulated categories with dg-enhancement, for example, $D(X)$; and the two definitions agree, cf. \cite{Keller1, Keller2} and also the discussion in \cite{Kuz09HHSOD}.

For an admissible subcategory $\shA \subset D(X)$, where $X$ is a smooth projective variety, Kuznetsov \cite{Kuz09HHSOD} defines the Hochschild homology $\HH_*(\shA)$ using the Fourier-Mukai kernel $\shP \in D(X \times X)$ of projection functor to $\shA$, and shows that it is independent of the embedding $\shA \subset D(X)$, and agrees with definition using the natural dg-enhancement of $\shA$. One of the most important features of Hochschild homology is that it is additive for semiorthogonal decompositions: suppose $D(X) = \langle \shA_1, \shA_2, \ldots, \shA_n\rangle$ is a semiorthogonal decomposition into admissible subcategories, then
	$$\HH_\bullet(X) = \HH_\bullet(\shA_1) \oplus \HH_\bullet (\shA_2) \oplus \ldots \oplus \HH_\bullet (\shA_n). $$
Cf. \cite[Thm. 7.3]{Kuz09HHSOD}. Also since K\"unneth formula holds for dg-categories, cf. \cite[Prop. 1.1.4]{PV12}, or \cite[\S 2.4]{Sh07}, and Hochschild homologies of admissible subcategories agree with the definition via their dg-enhancements. Therefore we have for any $\shA \subset D(X)$ and $\shB \subset D(Y)$,
	$$\HH_*(\shA \boxtimes \shB) = \HH_*(\shA) \otimes \HH_*(\shB),$$
where the tensor is regarded as a tensor product of graded vector spaces. For an admissible subcategory $\shA$, we define its (Hochschild) homological Euler characteristic to be
	$$\chi^H(\shA) : = \sum_k (-1)^k \rank \HH_k(\shA).$$
We denote $\chi^H(X) := \chi^H(D(X))$ for a possibly noncommutative variety $X$. Recall for a smooth projective (commutative) variety $X$, the famous Hochschild-Kostant-Rosenberg (HKR) isomorphism states (cf. \cite{CalHKR}):
	$$\HH_i(X) \simeq \bigoplus_{p-q = i} H^q(X,\Omega^p).$$
Hence $\chi^H(X) = \chi(X)$, where $\chi(X)$ is the topological Euler characteristic. The homological Euler characteristic is additive for semiorthogonal decompositions: $\chi(X) = \sum_k \chi^H(\shA_k)$ if $D(X) = \langle \shA_0, \ldots, \shA_n \rangle$, and multiplicative for exterior products: $\chi^H(\shA \boxtimes \shB) = \chi^H(\shA) \cdot \chi^H(\shB)$ by K\"unneth formula.

Now we apply these to our situation. Assume $X_T$ and $Y_S$  are smooth varieties. Then from (\ref{sod:H for X}) and (\ref{sod:H for S}) we have
	$$\chi(\shH) = \chi(X_T) + \chi(X) \cdot (\chi(S) - \chi^H(\shC_0)) = \chi(Y_S) + \chi(S) \cdot (\chi(X) - \chi^H(\shA_0)).$$
From the first statement (1) of our main theorem Thm. \ref{thm:HPDgen} we have $\chi(X) + \chi(Y)  = N\cdot\chi^H(\shA_0)$ and $\chi(S) + \chi(T) = N\cdot \chi^H(\shC_0)$, where recall $N = \dim_{\kk} V$. Direct computations lead to

\begin{corollary}[Pl\"ucker formula for HP-dual]\label{cor:plucker} In the same situation as Thm. \ref{thm:HPDgen}, we further assume $X_T$ and $Y_S$ are smooth, then
	$$\chi(X_T)  - \frac{\chi(X) \cdot \chi(T)}{N} = \chi(Y_S) - \frac{\chi(Y)\cdot \chi(S)}{N}.$$
\end{corollary}

\begin{example} As in Example \ref{Pf-Gr}, if we take $X = \Gr(2,7)$, and let $\tilde{Y}$ be the non-commutative resolution of the Pfaffian loci $Y=\Pf(4,7)$, $S = \PP^6$ and $T= \PP^{13}$ be a generic pair of dual linear sections. Then $N = 21$, $\chi(S) = 7$, $\chi(T) = 14$, $\chi(X) = 21$, $\chi^H(\tilde{Y}) = 42$. Then the corollary implies $\chi(X_T) = \chi(Y_S)$, which agrees with they having the same Hodge diamond. If we take $X = \Gr(2,6)$, and let $\tilde{Y}$ be the non-commutative resolution of the Pfaffian loci $Y=\Pf(4,6)$,  $S = \PP^5$ and $T = \PP^8$ be a generic pair of dual linear sections. Then $N=15$, $\chi(S)=6$, $\chi(T) = 9$, $\chi(X) = 15$, $\chi^H(\tilde{Y}) = 30$. Then the corollary implies $\chi(X_T) + 3 =  \chi(Y_S)$. Since $X_T$ is a $K3$ surface, $\chi(X_T) = 24$. Therefore $\chi(Y_S) = 27$. This agrees with $Y_S$ being a cubic fourfold. 
\end{example}

\medskip\noindent \textbf{Vanishing results}. As a preparation for playing the game of mutations, we show certain vanishing results on $D(\shH)$. These results are analogous to \cite[Lem. 5.1]{Kuz07HPD} and \cite[Lem. 4.3]{RT15HPD}.
Since $i_{\shH}: \shH \subset X\times S$ is a divisor in $|\sO_{X\times S}(1,1)|$, we have a short exact sequence
    $$ 0 \to \sO_{X\times S}(-1,-1) \to \sO_{X\times S} \to i_{\shH *} \sO_{\shH} \to 0.$$
This enables us to express $R\Hom_{\shH}$ in terms of $R\Hom$s on $X$ and $S$. 

\begin{lemma}[Vanishing lemma] \label{lem:van} For any $F_1,F_2 \in D(X)$, $G_1, G_2 \in D(S)$, 
    \begin{align}\label{cone-for-Hom}
    \begin{split}
     R\Hom_{\shH}(i_\shH^* (F_1\boxtimes\, &G_1), i_\shH^*(F_2 \boxtimes \,G_2))  =  \cone\big[R\Hom_X(F_1,F_2(-1))\otimes \\
    & R\Hom_S(G_1,G_2(-1)) \to R\Hom_X(F_1,F_2)\otimes R\Hom_S(G_1,G_2) \big]
    \end{split}
    \end{align}
\end{lemma}

\begin{proof} For any $E_1, E_2$ on $X \times S$, we have exact triangle
	$$R\sHom(E_1, E_2(-1,-1)) \to R \sHom(E_1, E_2) \to R \sHom(E_1, E_2 \otimes i_{\shH *} \sO_{\shH}) \xrightarrow{[1]}.$$
Apply global section functor, and notice by adjunction and projection formula we have $R \Hom_{X\times S}(E_1, E_2 \otimes  i_{\shH *} \sO_{\shH} ) = R\Hom_{X\times S}(E_1, i_{\shH *} i_\shH^* E_2) =R\Hom_{\shH}(i_\shH^* E_1, i_\shH^* E_2)$, one gets an exact triangle
 	$$R\Hom_{X\times S}(E_1, E_2(-1,-1)) \to R \Hom_{X \times S}(E_1, E_2) \to R \Hom_{\shH}(i_\shH^* E_1, i_\shH^* E_2) \xrightarrow{[1]}.$$ 	
Apply to the case $E_1 = F_1 \boxtimes G_1$, $E_2 = F_2 \boxtimes  G_2$ and use K\"unneth formula, we are done.\footnote{Note the adjunction, projection formula, etc hold for the derived category of all locally ringed spaces, cf. \cite{Lip09}, hence we do not need $\shH$ to be smooth. }
\end{proof}

\begin{remark}['$\alpha$- and $\beta$-vanishing']\label{rmk:alpha&beta_van}
The lemma will be repeatedly used later. In order to show a vanishing of the form
        \begin{equation}\label{formula:van}
        R\Hom_\shH(\shA_{\alpha_1}(\alpha_1)\boxtimes \shC_{\beta_1}(\beta_1), \shA_{\alpha_2}(\alpha_2)\boxtimes \shC_{\beta_2}(\beta_2)) = 0,
        \end{equation}
it suffices to show each term inside the cone in (\ref{cone-for-Hom}) has a factor which vanishes. We will refer to the first term as 'twisted' term (twisted by $\sO(-1,-1)$), and the second term as 'untwisted' term. There will be three typical situations:

\begin{enumerate}
\item ($\alpha$-vanishing). The $R \Hom_X$ factors of each term of the cone vanish. This happens for example if $ i-1 \ge  \alpha_1 > \alpha_2 \ge 1$.
\item ($\beta$-vanishing). The $R \Hom_S$ factors of each term of the cone vanish. This happens for example, if $ l-1 \ge \beta_1 > \beta_2  \ge 1$.
\item (Mixed type). One of the terms inside the cone has a $R \Hom_X$ factor which vanishes (we will refer this term as 'having $\alpha$-vanishing') and the other term has a vanishing $R \Hom_S$ factor (which we will refer as 'having $\beta$-vanishing'). This happens for example, if  $\alpha_1 = \alpha_2 \ge 1,~ \beta_1 > \beta_2 = 0$ or  $\alpha_1 >  \alpha_2 = 0,~ \beta_1 = \beta_2 \ge 1$.
\end{enumerate}
\end{remark}



\subsection{Fully-faithfulness}\label{sec:fullyfaithful} We first prove the statements about fully-faithfulness in Thm \ref{thm:HPDgen}. The strategy, following \cite{RT15HPD}, is as follows.

First, suppose we want to show $\pi_S: D(\shH) \to \shD_{Y_S}$ is fully faithful when restricted to some full triangulated subcategory, say $\shE \subset D(\shH)$, that is to show for all $a,b \in \shE$,
    $$R\Hom_{Y_S}(\pi_S \,a, \pi_S\, b) = R\Hom_{\shH}(a,b).$$
Since $\pi_S$ is left adjoint to the inclusion functor $i_S$, the left hand side is equal to $R\Hom_{\shH}(a, i_S \,\pi_S\, b) = R\Hom_{\shH}(a, \pi_S \,b)$\footnote{We omit the inclusion functor $i_S$ when regarding $\pi_S b$ as an object in $D(\shH)$, for simplicity of notations.}. Hence it is equivalent to show
        \begin{equation}\label{eqn:Van} R\Hom_{\shH}(a, \cone(b \to \pi_S \,b)) = 0,
        \end{equation}
where $b \to  \pi_S \,b$ is the counit map of the adjunction. If we can show for all $b \in \shE$, $ \cone(b \to \pi_S \,b) \in \shE^\perp$, we are done. Hence the key is to compute $\cone(b \to \pi_S \,b)$ for $b$ in the subcategories of interest. Similarly for $\pi_T$. 
    
Next, observe that $\pi_S$ (resp. $\pi_T$), as the projection to the component $\shD_{Y_S}$ (resp. $\shD_{X_T}$) of  of (\ref{sod:H for X}) (resp. (\ref{sod:H for S})), is the left mutation functor passing through the remaining components on its right  in (\ref{sod:H for X}) (resp. (\ref{sod:H for S})). Hence the computation follows from playing with of properties of mutation functors.
    
\medskip\noindent \textbf{Fully-faithfulness about $\pi_T$.}  We first show statements about fully-faithfulness of $\pi_T$.

\begin{proposition} \label{prop:pi_T} The restrictions of $\pi_T$ to the subcategories $\shA_1(1)\boxtimes \shD^1, \ldots,\shA_{i-1}(i-1)\boxtimes \shD^{i-1} \subset D(\shH)$ are fully faithful, and their images remain a semiorthogonal sequence in $\shD_{X_T}$. 
\end{proposition}

This will follow from the following lemma, which we will show in more detail. The lemma, which although appears complicated, is actually very intuitive if one visualizes the process and the regions in Figure \ref{Figure:l<=i}, \ref{Figure:l>=i}. For example, in Figure \ref{Figure:l<=i}, the shaded region inside the 'middle column' $D(X) \boxtimes \shC_0$ is where $b \in \shA_{k}(k) \boxtimes\, \shD^k$ belongs, and the shaded 'staircase region' on the right (inside the 'columns' indicated by ${}^\perp \shD_{X_T}$) is where the $\cone(b \to \pi_T(b))$ belongs.
 
 \begin{figure}
\begin{center}
\includegraphics[height=3.6 in]{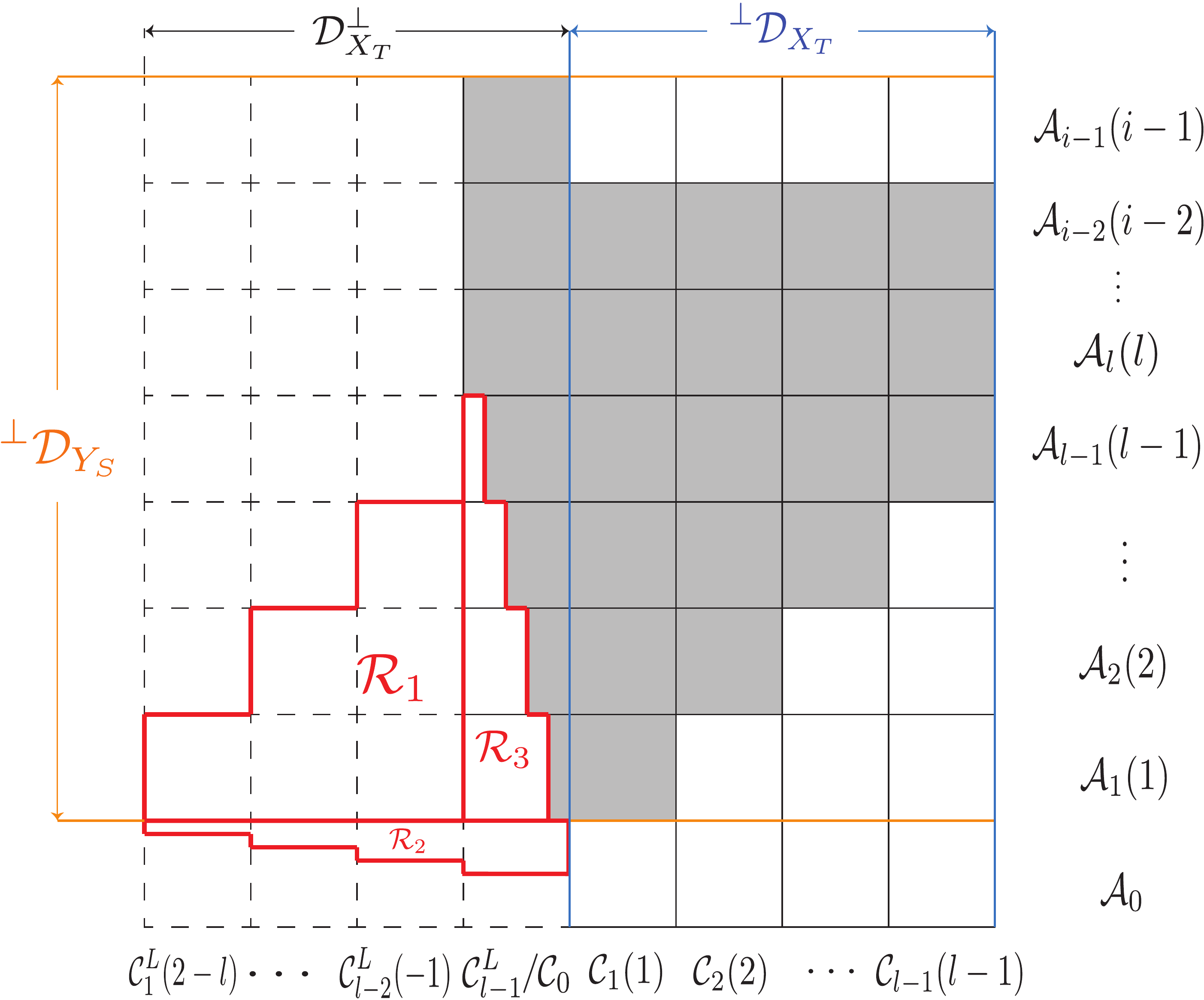}
\caption{The extended 'chessboard', where we glue the two decompositions $D(\shH) = \langle \shD_{X_T}, {}^\perp \shD_{X_T}\rangle$ and $D(\shH) = \langle \shD_{X_T}^{\perp}, \shD_{X_T}\rangle$ together. The blocks correspond whether $\shA_{\alpha}(\alpha) \boxtimes \shC_{\beta}(\beta)$ or $\shA_{\alpha}(\alpha) \boxtimes \shC^L_{\beta}(\beta+1-l)$ as indicated by the labels on axes, embedded into $D(\shH)$ by pullback $i_{\shH}^*$. Then the overlapping 'middle column' will correspond to whether the larger category $D(X) \boxtimes \shC_0$ or the smaller category $D(X) \boxtimes \shC^L_{l-1}$, depending on which decompositions we are considering. The staircases regions inside the diagram (whether the shaded ones or the ones enclosed by red lines) play essential roles in the proof of fully-faithfulness, as well as generation. 
} \label{Figure:l<=i}
\end{center}
\end{figure}
 
    \begin{lemma} \label{lem:region:pi_T} Let $b \in \shA_{k}(k) \boxtimes\, \shD^k$, $k = 1, \ldots, i-1$.  
If $k \in [1, l]$, then $\cone(b \to \pi_T(b))$ belongs to the subcategory
	\begin{equation*} \label{eqn:pi_T1}
	\left\langle\begin{array}{l}
	\shA_{k-1}(k-1)\boxtimes\shC_{1}(1), ~~
	\shA_{k-1}(k-1)\boxtimes \shC_{2}(2), \quad \dots \quad
 	\shA_{k-1}(k-1)\boxtimes\shC_{k-1}(k-1)\\
	\shA_{k-2}(k-2)\boxtimes\shC_{1}(1), \qquad \dots \qquad
	\shA_{k-2}(k-2)\boxtimes\shC_{k-2}(k-2)\\
	\qquad\qquad \vdots\\
	\shA_{1}(1)\boxtimes\shC_{1}(1)
	\end{array}\right\rangle \footnote{The the order of the semiorthogonal sequence is from bottom to top, and from left to right.}.
	\end{equation*}
If $k \in [l, i-1]$, then $\shD^k = \shC_0$, and $\cone(b \to \pi_T(b))$ belongs to the subcategory
	\begin{equation*} \label{eqn:pi_T2}
	\left\langle\begin{array}{l}
 	\shA_{k-1}(k-1)\boxtimes\shC_{1}(1), ~~
	\shA_{k-1}(k-1)\boxtimes \shC_{2}(2), \quad \dots \quad
 	\shA_{k-1}(k-1)\boxtimes\shC_{l-1}(l-1)\\
	\shA_{k-2}(k-2)\boxtimes\shC_{1}(1), \qquad \dots \qquad
	\shA_{k-2}(k-2)\boxtimes\shC_{l-2}(l-2)\\
	\qquad\qquad \vdots\\
	\shA_{k-l+1}(k-l+1)\boxtimes\shC_{1}(1)
	\end{array}\right\rangle.
	\end{equation*}
Therefore, the image of $\langle \shA_1(1)\boxtimes \shD^1, \ldots, \shA_{i-1}(i-1) \boxtimes \shD^{i-1}\rangle$ under $\pi_T$ is contained in the subcategory generated by themselves and
		$$\shA_{\alpha}(\alpha)\boxtimes \shC_{\beta}(\beta) \quad \text{for} \quad \beta \in [1, l-1], ~\alpha \in [\beta, i-2].$$
See the shaded region in Figure \ref{Figure:l<=i}. In particular, the image is contained in ${}^\perp \shD_{Y_S} = \langle  \shA_1(1) \boxtimes D(S), \ldots, \shA_{i-1}(i-1)\boxtimes D(S)\rangle$.
\end{lemma}

\begin{proof}
We show the case $1 \le k \le l$, the argument for $k \ge l$ are similar. In order to compute $\pi_T \, b$ and $\cone(b \to \pi_T\,b)$ for $b \in \shA_{k}(k) \boxtimes \shD^{k}$, where $\pi_T$ is the left mutation functor passing through the semiorthogonal sequence
	$$ D(X) \boxtimes \shC_1(1) , \ldots, D(X) \boxtimes \shC_{k-1}(k-1), D(X)\boxtimes \shC_k (k), \ldots, D(X)\boxtimes \shC_{l-1}(l-1),$$
we first observe that the subcategories $D(X)\boxtimes \shC_k (k), \ldots, D(X)\boxtimes \shC_{l-1}(l-1)$ has no (derived) $\Hom$s to $b$. This is from what we call '$\beta$-vanishing'. In fact, let $\beta \in [k,l-1]$, then the untwisted term is zero by $R\Hom_S(\shC_{\beta}(\beta), \shD^{k}) = 0$ since $\shD^{k} \subset \shC_0$, and twisted term is zero since
	\begin{align*} & R\Hom_S(\shC_{\beta}(\beta), \shD^k(-1)) =   R\Hom_S(\shC_{\beta}(\beta+1), \shD^k)\\
 	=&  R\Hom_{\shC_0}((\gamma_0^*( \shC_{\beta}(\beta+1)), \shD^k)  = R\Hom_{\shC_0} (\gamma_0^*( \foc_{\beta}(\beta+1)), \shD^k) = 0
	\end{align*}
exactly by the way we define $\shD^k$ in (\ref{def:D^k}) and Lem. \ref{lem:sod:A_0}. Therefore by Lem. \ref{lem:mut}, 
 	$$ \pi_T\,b = \LL_{D(X) \boxtimes \shC_1(1)} \circ \cdots \circ \LL_{D(X) \boxtimes \shC_{k-1}(k-1)} \, b.$$
Let $b^{(0)} = b$, and $b^{(\gamma)} = \LL_{D(X)\boxtimes\,\shC_{k-\gamma}(k-\gamma)}\,b^{(\gamma-1)}$ for $\gamma =  1, \ldots, k-1$. Then $b^{k-1} = \pi_T(b)$. Notice if $k=1$, $\pi_T(b) = b$, hence we are already done. Therefore only need to consider when $k\ge 2$. To prove $\cone(b \to \pi_T(b))$ belongs to the desired region, we prove by induction on $\gamma$ that $\cone(b \to b^{(\gamma)})$ belongs to the 'subregion' generated by
	\begin{equation} \label{ind:pi_T}
	\shA_{\alpha}(\alpha) \boxtimes \shC_\beta(\beta), \quad \text{for} \quad \beta \in [k-\gamma, k-1], ~\alpha \in [\beta, k-1].
	\end{equation}

\medskip \noindent \textit{Base case.} For $\gamma = 1$, to compute $b^{(1)} = \LL_{D(X) \boxtimes \shC_{k-1}(k-1)} \,b$, we use the following mutated decomposition for $D(X)$:
	$$D(X) = \langle \shA_{k-1}(k-1), {}^\perp (\shA_{k-1}(k-1))\rangle. 
	\footnote{We know explicitly that the second component is 
		$${}^\perp (\shA_{k-1}(k-1)) = \shA_{k}(k), \ldots, \shA_{i-1}(i-1), S_X^{-1}(\shA_0), \ldots, S_X^{-1}(\shA_{k-1}(k-1))$$ but this is actually not relevant for our computation.}	$$
Note of all components of $D(X)\boxtimes \shC_{k-1}(k-1)$ induced by this decomposition, only $\shA_{k-1}(k-1) \boxtimes \shC_{k-1}(k-1)$ has $\Hom$s to $b$. This is a vanishing of 'mixed type'. In fact, for the untwisted term, $R\Hom_S(\shC_{k-1}(k-1), \shD^k) = 0$ since $k \ge 2$, and for the twisted term, 
	$$R\Hom_X({}^\perp (\shA_{k-1}(k-1)) (1), \shA_{k}(k)) = R\Hom_X({}^\perp (\shA_{k-1}(k-1)), \shA_{k}(k-1)) = 0$$
by $\shA_{k}(k-1)\subset \shA_{k-1}(k-1)$. Thus $\cone(b \to b^{(1)}) \in \shA_{k-1}(k-1)\boxtimes \shC_{k-1}(k-1)$ as required.

\medskip \noindent \textit{Inductive step.} Next, suppose $\cone(b \to b^{\gamma})$ belongs to (\ref{ind:pi_T}), then $b^{\gamma}$ belongs the category generated by $\shA_k(k)\boxtimes \shD^k$ and (\ref{ind:pi_T}). To analyze $b^{(\gamma+1)} =\LL_{D(X)\boxtimes\,\shC_{k-\gamma-1}(k-\gamma-1)}\, b^{(\gamma)}$, we use the following decomposition of $D(X)$:
	$$D(X) = \big\langle \shA_{k-\gamma-1}(k-\gamma-1), \ldots, \shA_{k-1}(k-1), {}^\perp \langle \shA_{k-\gamma-1}(k-\gamma-1), \ldots, \shA_{k-1}(k-1)\rangle \big\rangle.$$
We claim that for this decomposition, only the fist $\gamma+1$ terms 
	\begin{equation} \label{eqn:ind:pi_T}\shA_{k-\gamma-1}(k-\gamma-1) \boxtimes \shC_{k-\gamma-1} (k-\gamma -1),\ldots, \shA_{k-1}(k-1)\boxtimes \shC_{k-\gamma-1}(k-\gamma -1)\end{equation}
has $\Hom$s to $b^{(\gamma)}$. First notice the rest of (\ref{eqn:ind:pi_T}),
	\begin{equation} \label{eqn:ind:pi_T:rest}  {}^\perp \langle \shA_{k-\gamma-1}(k-\gamma-1), \ldots, \shA_{k-1}(k-1)\rangle  \boxtimes \shC_{k-\gamma-1}(k-\gamma-1)\end{equation}	
has no $\Hom$s to $b$. The reason is similar as base case: the untwisted term is always zero from $\beta$-vanishing, and the twisted term is zero by $\alpha$-vanishing since $\shA_{k}(k-1) \subset \shA_{k-1}(k-1)$.
Then we only need to show there are no $\Hom$s from (\ref{eqn:ind:pi_T:rest}) to $\cone(b \to b^{(\gamma)})$. This is now from what we call '$\alpha$-vanishing': since the $\shA_{\alpha}(\alpha)$ appearing in (\ref{ind:pi_T}) has range $k - \gamma  \le \alpha \le k-1$, then the corresponding $\shA_{\alpha}(\alpha)$ and $\shA_{\alpha}(\alpha - 1)$ are both contained in $\langle  \shA_{k-\gamma-1}(k-\gamma-1) \ldots, \shA_{k-1}(k-1) \rangle$, therefore the $R\Hom_X$-factors of both twisted and untwisted terms are zero from above decomposition of $D(X)$. Therefore $\cone(b^{(\gamma)} \to b^{(\gamma+1)})$ belongs to (\ref{eqn:ind:pi_T}). Now from the distinguished triangle 
	$$\cone(b \to b^{(\gamma)}) \to \cone(b \to b^{(\gamma+1)}) \to \cone(b^{(\gamma)} \to b^{(\gamma+1)}) \xrightarrow{[1]}, $$ 
$\cone(b \to b^{(\gamma+1)})$ belongs to the desired region (\ref{ind:pi_T}). This completes the induction and hence the proof of the lemma.
\end{proof}

\begin{remark} \label{rmk:staircase} We see from the induction step, it is the twisted term $R\Hom_X(F_1, F_2(-1))$ (resp. $R\Hom_S(G_1, G_2(-1))$ for $\pi_S$) which is responsible for the \textbf{'staircase shape'} of the region. This interesting phenomenon will occur repeatedly: once we show the vanishing (resp. non-vanishing) for a box in the diagram in the process of mutations or by vanishing conditions, then the rest of vanishing (resp. non-vanishing) area will automatically follows a staircase pattern.
\end{remark}

\medskip\noindent\textit{Proof of Prop. \ref{prop:pi_T}}. This directly follows from Lem. \ref{lem:region:pi_T} as follows. For any $a \in \shA_m(m) \boxtimes \shD^m$, and $b \in \shA_k(k) \boxtimes \shD^k $ with $1 \le k \le m \le i-1$, from Lem. \ref{lem:region:pi_T}, $\cone(b \to \pi_T\,b)$ belongs to the subcategory generated by $\shA_{1}(1)\boxtimes D(S)$, $\ldots$, $\shA_{k-1}(k-1)\boxtimes D(S)$, but $a \in \shA_{m}(m)\boxtimes D(S)$ with $k-1 < m \le i-1$, hence $R\Hom(a, \cone(b \to \pi_T\,b)) = 0$ by semiorthogonality of (\ref{sod:H for X}). From the discussion at the beginning of the section, we are done. \hfill$\square$

\medskip
\begin{definition} We define $\sE \subset D(\shH)$ be the image of  the right orthogonal $\langle \pi_T(\shA_1(1)\boxtimes \shD^1), \ldots, \pi_T(\shA_{i-1}(i-1) \boxtimes \shD^{i-1}) \rangle^\perp \subset \shD_{X_T}$ under the inclusion $i_T$, i.e.,
	\begin{equation}\label{eqn:X_T}
	D(X_T) \simeq \shD_{X_T} =\big\langle \sE,  \pi_T(\shA_1(1)\boxtimes \shD^1), \ldots, \pi_T(\shA_{i-1}(i-1) \boxtimes \shD^{i-1}) \big\rangle
	\end{equation}
\end{definition}

\medskip\noindent \textbf{Fully-faithfulness about $\pi_S$.} We next show fully-faithfulness statements about $\pi_S$.

\begin{proposition} \label{prop:pi_S} The restrictions of $\pi_S$ to the subcategories $\shB^1 \boxtimes \shC^L_1(2-l) , \ldots, \shB^{l-1}\boxtimes \shC^L_{l-1}$ and $\sE$ are fully faithful, and their images form a semiorthogonal sequence in $\shD_{Y_S}$. 
\end{proposition}

This will follow from the next two lemmas. The first lemma is exactly symmetric to the situation of Lem. \ref{lem:region:pi_T}. The process and results can be intuitively visualized using the 'staircase regions' enclosed by red lines in Figure \ref{Figure:l<=i}, \ref{Figure:l>=i}: the 'staircase region $\shR_2$' inside the first row is where $b \in \shB^k \boxtimes\, \shC^L_{k}(k+1-l)$ belongs, and 'staircase region $\shR_1$' above the first row (contained in the rows indicated by ${}^\perp \shD_{Y_S}$) is where the $\cone(b \to \pi_S(b))$ belongs.

\begin{lemma}\label{lem:region:pi_S} Let $b \in \shB^k \boxtimes\, \shC^L_{k}(k+1-l)$, $k = 1, \ldots, l-1$. If $k \in [1, i]$, then $\cone(b \to \pi_S(b))$ belongs to
	$$
	\left\langle\begin{array}{r}
	\shA_{k-1}(k-1)\boxtimes\shC_{k-1}^L(k-l) \\
	\vdots \qquad \qquad\\
	\shA_2(2) \boxtimes \shC^L_2(3-l), \quad \dots \quad 
	\shA_{2}(2) \boxtimes \shC^L_{k-1}(k-l) \\
	\shA_{1}(1)\boxtimes\shC^L_{1}(2-l), ~~\shA_1(1) \boxtimes \shC^L_2(3-l), \quad \dots \quad  
	\shA_{1}(1) \boxtimes \shC^L_{k-1}(k-l)
	\end{array}\right\rangle.
	$$
If $k \in [i,l-1]$, then $\shB^k = \shA_0$, and $\cone(b \to \pi_S(b))$ belongs to
	$$
	\left\langle\begin{array}{r}
	\shA_{i-1}(i-1)\boxtimes\shC_{k-1}^L(k-l) \\
	\vdots \qquad \qquad\\
	\shA_2(2) \boxtimes \shC^L_{k-i+2}(k-i+3-l), \quad \dots \quad 
	\shA_{2}(2) \boxtimes \shC^L_{k-1}(k-l) \\
	\shA_{1}(1)\boxtimes\shC^L_{k-i+1}(k-i+2-l),
	 ~~\shA_1(1) \boxtimes \shC^L_{k-i+2}(k-i+3-l), \quad \dots \quad  
	\shA_{1}(1) \boxtimes \shC^L_{k-1}(k-l)
	\end{array}\right\rangle.
	$$
Therefore the image of $\langle \shB^1\boxtimes \shC^L_{1}(2-l), \ldots, \shB^{l-1}\boxtimes \shC^L_{l-1}\rangle$ under $\pi_S$ is contained in the subcategory generated by themselves and 
		$$\shA_{\alpha}(\alpha) \boxtimes \shC^L_{\beta} (\beta+1-l) \quad \text{for} \quad \beta \in [1, l-2], ~\alpha \in [1, \min\{i-1, \beta\}].$$
	See the union of staircase regions enclosed by red lines ('$\shR_1 \cup \shR_2$') in Figure \ref{Figure:l<=i} and \ref{Figure:l>=i}. In particular, the image is contained in $\shD_{X_T}^\perp$.
\end{lemma}

\begin{proof} This proof is similar to Lem. \ref{lem:region:pi_T}. We shall only sketch the key steps for $k \in [1, i]$. The case $k \in [i, l-1]$ is similar. $\pi_S$ is the left mutation past through
	$$\shA_1(1)\boxtimes D(S), \ldots, \shA_{k-1}(k-1)\boxtimes D(S), \shA_{k}(k)\boxtimes D(S), \ldots, \shA_{i-1}(i-1)\boxtimes D(S).$$
From similar argument as, $b$ receives no $\Hom$s from $\shA_{k}(k)\boxtimes D(S), \ldots, \shA_{i-1}(i-1)\boxtimes D(S)$ by $\alpha$-vanishing, which is a consequence of the way we define $\shB^k$ in (\ref{def:D^k}), therefore
	$$\pi_S\,b = \LL_{\shA_1(1)\boxtimes D(S)} \circ \cdots \circ \LL_{\shA_{k-1}(k-1)\boxtimes D(S)} \,b$$
Let $b^{(0)} = b$ and $b^{(\gamma)} = \LL_{\shA_{k-\gamma}(k-\gamma) \boxtimes D(S)} \, b^{(\gamma-1)}$ for $\gamma = 1, 2, \ldots, k-1$. Then $b^{(k-1)} = \pi_S\,b$, and one can show inductively $\cone(b \to b^{(\gamma)})$ belongs to the subcategory generated by
	$$\shA_{\alpha}(\alpha)\boxtimes \shC^L_{\beta}(\beta+1-l), \quad \text{for} \quad \beta \in [k-\gamma, k-1], ~\alpha\in [k-\gamma,\beta] .$$
Assume it is true for $\gamma \in [0, k-2]$,\footnote{Here $\gamma=0$ corresponds to the base case, for which the inductive assumption is trivial.} then to compute $b^{(\gamma+1)} =  \LL_{\shA_{k-\gamma-1}(k-\gamma-1) \boxtimes D(S)} \, b^{(\gamma)}$, we consider the following decomposition of $D(S)$, obtained from (\ref{sod:S:C^L}):
	$$D(S) = \big\langle \langle\shC^L_{k-\gamma-1}(k-\gamma-l), \ldots, \shC^L_{k-1}(k-l) \rangle, {}^\perp \langle\shC^L_{k-\gamma-1}(k-\gamma-l), \ldots, \shC^L_{k-1}(k-l) \rangle \big\rangle.$$
It's easy to see for this decomposition, only the part $\shA_{k-\gamma-1}(k-\gamma-1) \boxtimes \langle\shC^L_{k-\gamma-1}(k-\gamma-l), \ldots, \shC^L_{k-1}(k-l) \rangle$ has $\Hom$s to $b^{(\gamma)}$, and therefore completes the induction step.
\end{proof}

\begin{lemma}\label{lem:region:pi_S:sE} Let $b \in \sE$, then $\cone(b \to \pi_S(b))$ belongs to the subcategory $\shC_\sE$ generated by
		$$\shA_{\alpha}(\alpha) \boxtimes \shC_{\beta}^L (\beta+1-l) \quad \text{for} \quad \beta \in [1, l-1], ~\alpha \in [1, \min\{i-1, \beta\}].$$
In particular, $\cone(b \to \pi_S(b)) \in \shD_{X_T}^\perp$ by the semiorthogonal decomposition (\ref{sod:H for S:Serre}). 
\end{lemma}

In Figure \ref{Figure:l<=i} and \ref{Figure:l>=i}, $\shC_\sE = \shR_1 \cup \shR_3'$, where $\shR_3' = \langle \shA_{\alpha}(\alpha) \boxtimes \shC^L_{l-1}~|~\alpha \in [1, l-1]\rangle$ contained in the \emph{smaller} middle column $D(X) \boxtimes \shC^L_{l-1}$. However, if we consider the decomposition (\ref{sod:S:C^L}) of $D(S)$, one can regard the part $\shR_3'$ as \emph{placed} in the \emph{the complement} of the shaded regions $\shD^k \boxtimes \shC_0$ (i.e. the place of $\shR_3$) inside the \emph{larger} middle column $D(X)\boxtimes \shC_0$. Therefore we may regard $\shC_\sE$ as the union $\shR_1 \cup \shR_3$ in the diagrams. This point of view will useful in the proof.

\begin{proof} The method is the same as before. The only thing we need to pay attention is the vanishing for $\sE$, which is not obtained through computing the cone as in (\ref{cone-for-Hom}). Let's show the case when $l \le i$, the other case $i \le l$ is similar with only some changes of subscripts. First notice since there are no Homs from $\shA_l(l)\boxtimes D(S), \dots, \shA_{i-1}(i-1)\boxtimes D(S)$ to $b$, 
 $$\pi_S(b)= \LL_{\shA_1(1)\boxtimes D(S)}\circ \cdots \circ \LL_{\shA_{l-1}(l-1)\boxtimes D(S)} \, b,$$
Then similarly let  $b^{(0)}=b$ and $b^{(\gamma)}=\LL_{\shA_{l-\gamma}(l-\gamma)\boxtimes D(S)} \,b^{(\gamma-1)}$ where $\gamma\in [1, l-1]$. We show by induction on $\gamma$ that $\cone(b \to b^{(\gamma)})$ belongs to the subcategory of $D(\shH)$ generated by
	$$\shA_{\alpha}(\alpha) \boxtimes \shC^L_{\beta}(\beta+l-1) \quad \text{for} \quad \beta \in [l-\gamma, l-1],~\alpha \in [l-\gamma, \beta].$$
Assume true for $\gamma \in [0, l-2]$. Then to compute $b^{(\gamma+1)} = \LL_{\shA_{l-\gamma-1}(l-\gamma-1)\boxtimes D(S)} \,b^{(\gamma)}$, consider the decomposition (\ref{sod:S:C^L}):
	$$D(S) = \langle \shC^L_{l-\gamma-1}(-\gamma), \shC^L_{l-\gamma}(1-\gamma), \ldots, \shC^L_{l-1}, \shC_0(1), \ldots, \shC_{l-\gamma-2}(l-\gamma-1)\rangle.$$
As before, $\cone(b \to b^{(\gamma)})$ receives no $\Hom$ from 
	\begin{equation}\label{eqn:ind:E} \shA_{l-\gamma-1}(l-\gamma-1) \boxtimes \langle \shC_0(1), \ldots, \shC_{l-\gamma-2}(l-\gamma-1)\rangle.
	\end{equation}
The only difference from before is, in order to show this holds for $b^{(\gamma)}$, hence make the induction to work, we need to show $b$ also receives no $\Hom$s from (\ref{eqn:ind:E}) for all $\gamma = 0,1, \ldots, l-2$. Since $b \in \sE \subset {}^\perp \shD_{X_T}$, it receives no $\Hom$s from
	$$\shA_{l-\gamma-1}(l-\gamma-1) \boxtimes \langle \shC_0(1), \ldots, \shC_{l-\gamma-1}(l-\gamma-1) \rangle.$$
Also from the definition of $\sE$, 
	$$R\Hom(\shA_{l-\gamma-1}(l-\gamma-1) \boxtimes \shD^{l-\gamma-1}, \sE) = 0.$$
But recall $\shD^{l-\gamma-1} =  \langle \gamma_0^* (\foc_0(1)), \ldots, \gamma_0^*(\foc_{l-\gamma-2}(l-\gamma-1))\rangle$, and Lem. \ref{lem:sod:A_0} for $D(S)$ says
	\begin{align*}  & \langle \shC_0(1), \ldots, \shC_{l-\gamma-2}(l-\gamma-1) \rangle  
	  =  \langle \shD^{l-\gamma-1}, \shC_1(1), \ldots, \shC_{l-\gamma-1}(l-\gamma-1)\rangle
	\end{align*}
for all $\gamma = 0,1, \ldots, l-2$. Hence we are done.
\end{proof}

\noindent \textit{Proof of Prop. \ref{prop:pi_S}.} To prove the proposition, we show 
	$$R\Hom(a, \cone(b \to \pi_S\,b)) = 0$$
for any $a,b$ in one of the two cases:
\begin{enumerate}
\item \label{case:1}  $a\in  \shB^m \boxtimes\, \shC^L_{m}(m+1-l)$ or $a\in \sE$, $b \in  \shB^k \boxtimes\, \shC^L_{k}(k+1-l)$, $1\le k \le m \le l-1$;
\item   \label{case:2} $a \in \sE$, $b \in \sE$.
\end{enumerate}
For case (\ref{case:1}), from Lem. \ref{lem:region:pi_S}, $\cone(b\to \pi_S\,b)$ belongs to the subcategory generated by $D(X)\boxtimes \shC^L_{1}(2-l), \ldots, D(X)\boxtimes \shC^L_{k-1}(k-l)$, hence receives no $\Hom$s from $D(X)\boxtimes \shC^L_{m}(m+1-l)$ for $k \le m \le l-1$, or $\sE \subset \shD_{X_T}$, by semiorthogonality of (\ref{sod:H for S:Serre}). 
For case (\ref{case:2}), Lem. \ref{lem:region:pi_S:sE} implies $\cone(b\to \pi_S\,b)$ belongs to $\shD_{X_T}^\perp$ 
hence receives no $\Hom$s from $b \in \sE \subset \shD_{X_T}$. \hfill$\square$

\subsection{Generation.} \label{sec:generation}

To finish the proof of Thm. \ref{thm:HPDgen}, we need to show the fully faithful images of the subcategories in 
Prop \ref{prop:pi_S} generate $\shD_{Y_S}$. Our strategy is to show the right orthogonal of these images inside $\shD_{Y_S}$ is zero. Using adjunction of the functor $\pi_S$, the orthogonality conditions translate into vanishing conditions on $D(\shH)$, on which we can play the familiar game on the 'chessboard' Figure \ref{Figure:l>=i}, and show the desired vanishing 'block-by-block'.

\begin{figure}
\begin{center}
\includegraphics[height=2.8in]{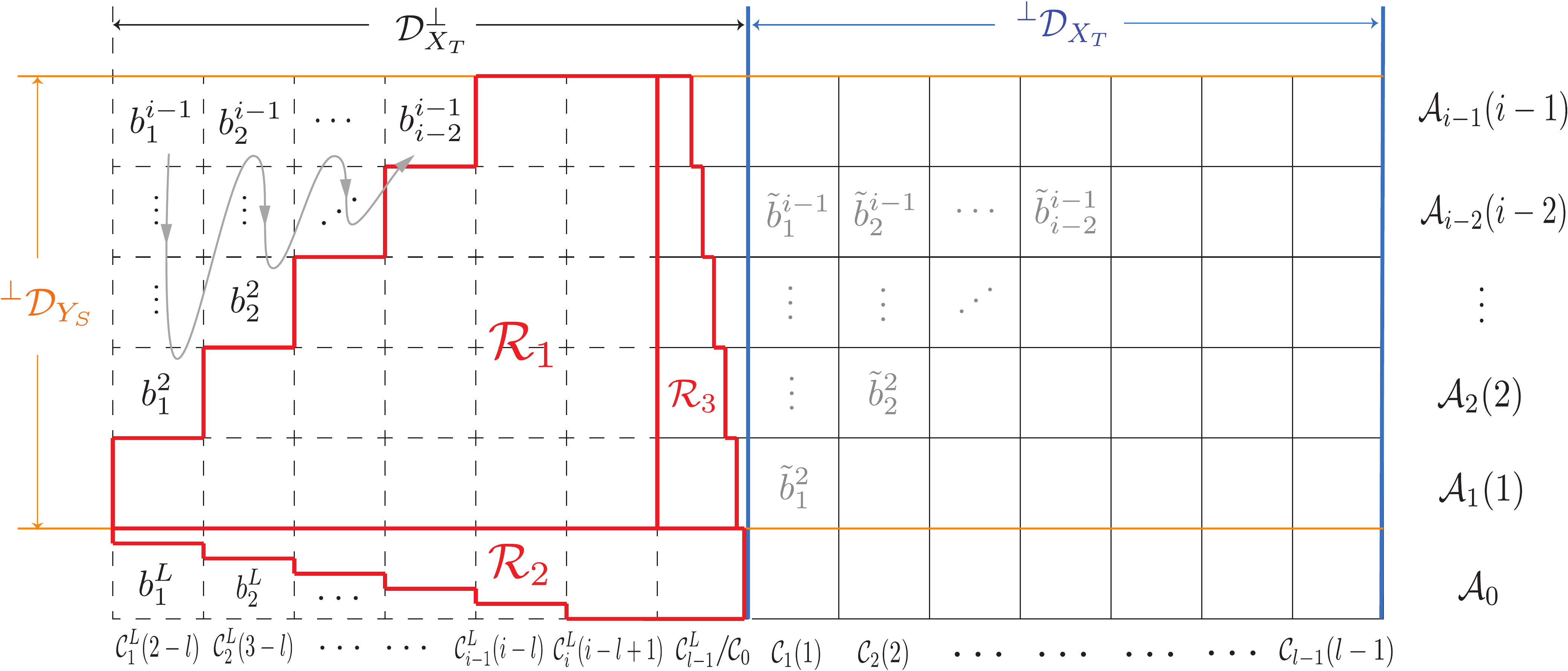}
\caption{The 'chessboard' in the case $l=8$, $i=6$. The grey line with arrows indicates the 'Zig-Zag' routine ($b_1^{i-1} \to \cdots \to b_1^2 \to b_2^{i-1} \to \cdots \to b_2^2 \to \cdots \to b_{i-1}^{i-1}$) of induction in Step $1$ of the proof of {\em generation} part, and the grey $\tilde{b}^{\alpha}_{\beta}$ on the right indicates the block we use to detect the vanishing of $b^{\alpha}_\beta$. }\label{Figure:l>=i}
\end{center}
\end{figure}

First notice from Lem. \ref{lem:region:pi_T}, the (left) orthogonal of $\sE$ inside $\shD_{X_T}$ in the decomposition (\ref{eqn:X_T}) are totally contained in ${}^\perp \shD_{Y_S}$, therefore $\sE$ and $\shD_{X_T}$ have the same image under $\pi_S = \LL_{^\perp \shD_{Y_S}}$ by Lem. \ref{lem:mut}. Hence we need only to show the images under $\pi_S$ of
	$$\shB^1 \boxtimes \shC^L_1(2-l) , \ldots, \shB^{l-1}\boxtimes \shC^L_{l-1}, \quad \text{and} \quad \shD_{X_T}$$
 generate $\shD_{Y_S}$.  Let $b\in \shD_{Y_S}$ be an object in the right orthogonal of these images. 
	\begin{align}
	 R\Hom_{\shH}(a,b) & = 0  \quad \forall ~a\in \shD_{X_T}. \label{cond:X_T} \\
	 R\Hom_{\shH}(a,b)  & = 0  \quad \forall ~a\in \shB^1\boxtimes \shC^L_1(2-l), \ldots, \shB^{l-1}\boxtimes \shC^L_{l-1} \label{cond:B}
	\end{align}	
by adjunction $R\Hom_{\shH}(a,b) = R\Hom_{\shD_{Y_S}}(\pi_S\,a, b)$ of $\pi_S$. We want to show $b = 0$. From decomposition (\ref{sod:H for S:Serre}), the first condition (\ref{cond:X_T}) is equivalent to
	$$b \in \shD_{X_T}^{\perp} = \langle D(X) \boxtimes \shC_1^L(2-l), \ldots, D(X) \boxtimes \shC^L_{l-1}\rangle.$$
We can further decompose $\shD_{X_T}^{\perp}$ using $D(X) = \langle \shA_0, \shA_1(1), \ldots, \shA_{i-1}(i-1)\rangle$,
	$$\shD_{X_T}^{\perp} = \big\langle  \shA_\alpha(\alpha) \boxtimes \shC^L_{\beta}(\beta+1-l) \big\rangle_{\alpha \in [0,i-1], \beta \in [1,l-1].}$$
 Therefore $b$ has the following components with respect to the decomposition:
	\begin{equation}b^{\alpha}_\beta \in \shA_\alpha(\alpha) \boxtimes \shC^L_{\beta}(\beta+1-l),\quad \text{with} \quad \alpha \in [0,i-1], ~\beta \in [1,l-1].\end{equation}
Notice may further put $b^0_\beta$ into a distinguished triangle
	$$b_\beta^R \to b_\beta^0 \to b_\beta^L \xrightarrow{[1]}, \quad \beta \in [1, l-1],$$
using decomposition $\shA_0 = \langle (\shB^\beta)^\perp, \shB^\beta \rangle$, where $(\shB^\beta)^{ \perp}$ is the right orthogonal of $\shB^\beta$ inside $\shA_0$,
	\begin{equation}\label{def:b^Lb^R}  b_\beta^L \in (\shB^\beta)^\perp \boxtimes \shC^L_{\beta}(\beta+1-l), \quad b_\beta^R \in  \shB^\beta \boxtimes \shC^L_{\beta}(\beta+1-l).
	\end{equation}
See Figure \ref{Figure:l>=i}. Now the game is played as follows.

\bigskip
\noindent\textbf{Step 1.} {\em The components of $b$ which are 'above the staircase region $\shR_1$' are zero, i.e.}
	\begin{equation}\label{van:balpha} b^\alpha_\beta = 0 \quad \text{\em for} \quad \beta \in [1,l-1],~ \alpha \in [\beta+1, i-1].
	\end{equation}
Let's show this by induction, with an induction routine:
	$$(b_1^{i-1} = 0) \implies (b_1^{i-2} = 0) \implies \cdots \implies (b_1^2=0) \implies (b_2^{i-1} = 0) \implies \cdots \implies (b_{i-1}^{i-2} = 0) $$ 
which is indicated by the grey curved line with arrows in Figure \ref{Figure:l>=i}. Assume we have done proving (\ref{van:balpha}) those $b^{\alpha}_{\beta}$'s in the region with smaller $\beta$ or with same $\beta$ but larger $\alpha$ (for the base case $\alpha=i-1,\beta=1$ this assumption is empty). To show $b^{\alpha}_{\beta} = 0$, we look at $\Hom$s from
	\begin{equation}\label{eqn:step1:X}
	D(X)\boxtimes \shC_{\beta}(\beta)= D(X)(-1)\boxtimes \shC_{\beta}(\beta).
	\end{equation}
If there is no $\Hom$s from above category to $b^{\alpha}_{\beta}$, then from (\ref{sod:H:k}), $b^{\alpha}_{\beta} \in D(X) \boxtimes \shC^L_{\beta}(\beta+1-l)$ will belong to subcategory generated by
	$$D(X) \boxtimes \shC^L_{\beta+1}(\beta-l), \ldots, D(X) \boxtimes \shC^L_{l-1},$$
which is contained in the left orthogonal of $D(X) \boxtimes \shC^L_{\beta}(\beta+1-l)$. This will force $b^{\alpha}_{\beta}=0$.  We now claim the $\Hom$s from (\ref{eqn:step1:X}) to $b^\alpha_\beta$ totally come from
	\begin{equation}\label{eqn:step1:A}
	\shA_{\alpha}(\alpha-1) \boxtimes \shC_\beta(\beta).
	\end{equation}	
The reason is, if we consider $D(X) = \langle \shA_{\alpha}(\alpha), {}^\perp(\shA_{\alpha}(\alpha))\rangle$
\footnote{More explicitly, ${}^\perp(\shA_{\alpha}(\alpha)) = \langle \shA_{\alpha+1}(\alpha+1), \ldots, \shA_{i-1}(i-1), S_X^{-1}(\shA_0), \ldots, S_X^{-1}(\shA_{\alpha-1} (\alpha-1)) \rangle.$ }, 
then $D(X)(-1) \boxtimes \shC_\beta(\beta) = \langle \shA_{\alpha}(\alpha-1), {}^\perp(\shA_{\alpha}(\alpha))(-1) \rangle \boxtimes \shC_\beta(\beta)$. The $\Hom$ from the latter component to $b^{\alpha}_\beta \in \shA_{\alpha}(\alpha) \boxtimes \shC^L_{\beta}(\beta+1-l)$ is a cone of the form (\ref{cone-for-Hom}), with the untwisted term zero by $\beta$-vanishing \footnote{More precisely, $\Hom_{S}(\shC_{\beta-1}(\beta), \shC^L_{\beta}(\beta+1-l) )= 0$ and $\shC_{\beta}(\beta) \subset \shC_{\beta-1}(\beta)$.}, and the twisted term is zero by
	$$R\Hom_{X}({}^\perp(\shA_{\alpha}(\alpha))(-1), \shA_{\alpha}(\alpha)(-1)) = R\Hom_X ({}^\perp(\shA_{\alpha}(\alpha)),\shA_{\alpha}(\alpha)) = 0.$$
To show vanishing from (\ref{eqn:step1:A}) is where the 'staircase' induction routine comes in. We show all the other surviving components of $b$, except $b^\alpha_\beta$, receive no $\Hom$s from (\ref{eqn:step1:A}). In fact, for components with larger $\beta$, the $\Hom$s are zero by $\beta$-vanishing; For components with the same $\beta$ but smaller $\alpha$, i.e. components in $\langle \shA_0, \shA_1(1), \ldots \shA_{\alpha-1}(\alpha-1)\rangle \boxtimes \shC^L_{\beta}(\beta+1-l)$, then $\Hom$s are zero since untwisted terms vanishes by $\beta$-vanishing, and twisted term contains factor
	$$R\Hom_X(\shA_\alpha(\alpha), \langle \shA_0, \shA_1(1), \ldots \shA_{\alpha-1}(\alpha-1)\rangle) = 0;$$
For components with smaller $\beta$, by induction assumption, these are components of the form $b^{\alpha'}_{\beta'}$ with $0 \le \alpha ' \le \beta'$, $\beta ' < \beta$. Note $\alpha' \le \beta -1 \le \alpha -2$, then the $\Hom$s from (\ref{eqn:step1:A}) are  zero by $\alpha$-vanishing.

Now note $b\in \shD_{Y_S}$ itself receives no $\Hom$s from (\ref{eqn:step1:A}), therefore by Lem. \ref{lem:component_van}, $b^\alpha_\beta$ receives no $\Hom$s from (\ref{eqn:step1:A}) also. Hence $b^\alpha_\beta = 0$, and by induction we have proved (\ref{van:balpha}).

\bigskip\noindent\textbf{Step 2.} {\em The components of $b$ which are 'below staircase region $\shR_2$' are zero:}
	\begin{equation}\label{van:b^L}	b^L_\beta = 0 \quad \text{\em for} \quad \beta=1,2,\ldots, l-1.
	\end{equation}
We show this by induction on $\beta \in [1, l-1]$. Assume it is true for components with smaller $\beta$ (for base case $\beta=1$ the assumption is empty), then to show $b^L_\beta = 0$ we again only need to look at $\Hom$s from $D(X) \boxtimes \shC_\beta(\beta)$. But this time we consider decomposition
	\begin{equation}D(X) = \langle \foa_0, \shA_1, \foa_1(1), \shA_2(1), \ldots, \foa_{i-2}(i-2), \shA_{i-1}(i-2), \foa_{i-1}(i-1) \rangle.\end{equation}
	First, there are no $\Hom$s from the subcategories $\shA_1 \boxtimes \shC_\beta(\beta), \shA_2 (1) \boxtimes \shC_\beta(\beta), \ldots, \shA_{i-1}(i-2) \boxtimes \shC_\beta(\beta)$ to $b^L_\beta$, since the cones (\ref{cone-for-Hom}) of which have zero untwisted terms by $\beta$-vanishing, and zero twisted terms by $\alpha$-vanishing
		$$R\Hom_{X}(\langle \shA_1, \ldots, \shA_{i-1}(i-2)\rangle(1), \shA_{0}) = 0.$$	
	Second, there are also no $\Hom$s from subcategories $\langle \foa_0, \foa_{1}(1), \ldots, \foa_{\beta-1}(\beta-1) \rangle \boxtimes \shC_\beta(\beta)$ to $b^L_\beta$ by definition of $b^L_\beta$. More precisely, the $\Hom$s of these totally come from the twisted terms, which are zero exactly by adjoint pairs $(\alpha_0^*, \alpha_0)$ and the defining equation (\ref{def:b^Lb^R}) of $b^L_\beta$:
	\begin{align*} R\Hom_X(\langle \foa_0, \foa_{1}(1), \ldots,  \foa_{\beta-1}(\beta-1)\rangle (1), (\shB^\beta)^\perp ) =  R\Hom_{\shA_0}(\shB^\beta, (\shB^\beta)^\perp ) = 0.
	\end{align*}
	Hence it remains to show there are no $\Hom$s from the following subcategories to $b^L_\beta$:
	\begin{equation}\label{eqn:step2}
	\foa_{\beta}(\beta) \boxtimes \shC_\beta(\beta), \ldots, \foa_{i-1}(i-1) \boxtimes \shC_\beta (\beta).
	\end{equation}
	Now the 'staircase' shape of the region and the induction hypothesis begin to play their roles. We show there are no $\Hom$s from (\ref{eqn:step2}) to all other non-zero components other than $b^L_\beta$. These components can be formed into three different categories:
	\begin{enumerate}[leftmargin = *]
	\item the components with larger $\beta$'s, i.e., ones in $D(X)\boxtimes \shC^L_{\beta+1}(\beta+2-l), \ldots, D(X)\boxtimes \shC^L_{l-1}$. These are no $\Hom$s from (\ref{eqn:step2}) to these simply by $\beta$-vanishing.
	\item The surviving components from the staircase region of \textit{Step 1} with equal or smaller $\beta$, i.e. $b^{\alpha'}_{\beta'}$ with $1 \le \beta' \le \beta$, and $1 \le \alpha' \le \beta'$. The $\Hom$'s from (\ref{eqn:step2}) to these are zero by $\alpha$-vanishing: for components with $\beta' = \beta$, the untwisted terms are zero, and for twisted terms
		$$R\Hom_X(\langle \foa_{\beta}(\beta), \ldots, \foa_{i-1}(i-1)\rangle, \shA_{\alpha'}(\alpha'-1)) = 0 $$
	since $\shA_{\alpha'}(\alpha'-1) \subset \shA_{\alpha'-1}(\alpha'-1)$ and $1 \le \alpha' \le \beta$. For components with $\beta' \le \beta -1$, apart from twisted terms zero as above, we also have for untwisted terms 
		$$R\Hom_X(\langle \foa_{\beta}(\beta), \ldots, \foa_{i-1}(i-1)\rangle, \shA_{\alpha'}(\alpha')) = 0 $$
	since $\alpha' \le \beta' \le \beta - 1$.
	\item The surviving $b^R_{\beta'}$ with $1 \le \beta' \le \beta$ from our induction. Notice the untwisted terms in the cone for $\Hom$s from (\ref{eqn:step2}) to these are always zero, since $\beta \le 1$, and all these $\shB^{\beta'} \subset \shA_0$. For twisted terms,  
		\begin{align*} & R\Hom_X(\langle \foa_{\beta}(\beta), \ldots, \foa_{i-1}(i-1)\rangle (1), \shB^{\beta'}) \\
		& = R\Hom_{\shA_0}(\langle \alpha_0^* (\foa_{\beta}(\beta+1)), \ldots, \alpha_0^*(\foa_{i-1}(i)), \shB^{\beta'})  = 0
		\end{align*}
by $\shB^{\beta'} \subset \shB^\beta$ for $\beta' \le \beta$, Lem. \ref{lem:sod:A_0}, and definition (\ref{def:B^k}) of $\shB^\beta$.
	\end{enumerate}
Finally $b$ itself receives no $\Hom$s from (\ref{eqn:step2}) where $\beta \ge 1$ 
by Lem. \ref{lem:mut}, $b^L_\beta$ receives no $\Hom$s from (\ref{eqn:step2}). All together, $R\Hom(D(X) \boxtimes \shC_\beta(\beta), b^L_\beta) = 0$, which forces $b^L_\beta = 0$ and concludes the induction for (\ref{van:b^L}).	

\medskip
\noindent\textbf{Final step.} By previous two steps, $b$ belongs to 'the region $\shR_1 \cup \shR_2$', i.e. the subcategory generated by 
	$$ \shA_{\alpha}(\alpha) \boxtimes \shC^L_{\beta} (\beta+1-l), \quad \text{for} \quad \alpha \in [1, \min\{\beta, i-1\}], ~\beta \in [1,l-1],$$
and $ \shB^1 \boxtimes \shC^L_1(2-l) , \ldots, \shB^{l-1}\boxtimes \shC^L_{l-1}$.
But the fact $b\in\shD_{Y_S}$ and the condition (\ref{cond:B}) imply $b$ belongs to the right orthogonal of these categories as well. Therefore $R\Hom_{\shH}(b,b) = 0$, which implies $b=0$. This concludes the proof of generation and hence the proof Thm. \ref{thm:HPDgen}.	

\hfill$\square$

\section{Generalizations and Applications} \label{sec:nceg}

\subsection{Noncommutative HP-duals}  \label{sec:ncHPD} As our proof is purely \textbf{categorical}, it is straightforward to generalize our main theorem to noncommutative cases. The point of this consideration is that the HP-dual category $\sC$ in (\ref{sod:H_X}) is not necessarily given by the derived category of an algebraic variety $Y$, but in various examples are given by interesting noncommutative varieties which in a way is still geometric.

\subsubsection{Azumaya Varieties} \label{sec:Azumaya}The simplest but interesting cases of noncommutative varieties are Azumaya varieties (or more generally, Azumaya algebras over schemes), which can be regarded as most geometric among non-commutative varieties, and was introduced by Grothendieck \cite{Gr} in the study of Brauer groups. Other standard references are \cite[Chap. IV]{Milne}, \cite[Chap. 1]{Cal00}. See also \cite[\S 2]{Kuz06Hyp}.

\begin{definition} An algebraic variety $X$ together with a sheaf of algebras $\sA_X$ (noncommutative) is called an \textbf{Azumaya variety}, if $\sA_X$ is a locally free $\sO_X$-module of finite rank, and for any $x \in X$, there exists an \'etale (or analytic) neighbourhood $U \to X$ of $x$ and a locally free sheaf $\sE$ on $U$ such that $\sA_U \simeq \underline{\End}_{\sO_U}(\sE)$, the local endomorphism sheaf of $\sE$.
\end{definition}

The definitions are the same no matter whether we are using \'etale topology or analytic one (but not true for Zariski topology), and the definitions given here are equivalent to ones given in above references, cf. \cite[Thm. 1.1.6]{Cal00}. A morphism of Azumaya varieties $f:(X,\sA_X) \to (Y,\sA_Y)$ is a morphism $\underline{f}$ of underlying varieties and a morphism of $\sO_X$-algebras $f^{\#}: \underline{f}^* \sA_Y \to \sA_X$. A morphism $f$ is called strict if $f^{\#}$ is an isomorphism. An ordinary variety is an Azumaya variety by taking $\sA_X = \sO_X$, and morphisms between varieties are always strict.

The abelian category $\Coh(X,\sA_X)$ of right coherent $\sA_X$-modules on $(X,\sA_X)$ can also be described by coherent $\alpha$-twisted sheaves on $X$, where $\alpha \in \check{H}_{\textrm{\'et}}^2(X,\sO_X^{*})$ is the class determined by $\sA_X$, cf. \cite[\S 1.2, 1.3]{Cal00}. Similarly for quasi-coherent $\sA_X$-modules. We will be interested in the bounded derived categories $D(X,\sA_X)$ of the abelian category $\Coh(X,\sA_X)$. Following Kuznetsov, we assume $X$ to be embeddable. Standard derived functors like tensor $\otimes$, $R\sHom$, $f^*$, $f_*$, etc are well defined in the categories of (embeddable) Azumaya varieties as in variety cases. See \cite{Kuz06Hyp} for more details.

It makes perfect sense to say an Azumaya variety $(Y,\sA_Y)$ together with a morphism\footnote{Actually it works even if the underlying varieties $Y$ admit a rational morphism towards $\PP V^*$ under certain conditions. For more details, see \cite{Kuz06Hyp}.} $Y \to \PP V^*$ to be the \textbf{HP-dual} of the variety\footnote{We expect $X$ can also be Azumaya variety, as long as we define properly what do we mean by universal hyperplanes for $(X,\sA_X)$. We expect this even for more general noncommutative schemes $(X,\sR_X)$.} $X$ with $X \to \PP V$ (and a Lefschetz decomposition). This means we have a $\PP V^*$-linear derived equivalence $D(Y,\sA_Y) \simeq \sC$, where $\sC$ is the HP-dual category defined by (\ref{sod:H_X}), and we further require it to be given by an Fourier-Mukai kernel supported on $\sH_X \times_{\PP V^*} Y$.

Fibre products and faithful base-changes can also be defined for Azumaya varieties, and everything in Sec. \ref{sec:base-change} holds for strict base changes of Azumaya varieties, cf. \cite{Kuz06Hyp}. We can similarly define two pair of morphisms to $\PP V$ and $\PP V^*$ to be admissible similarly as in Def. \ref{def:admissible} (the faithful conditions for the pair hold). Similar criterion for admissibility also holds:
\begin{lemma} Lem. \ref{lem:admissible pairs} is true for Azumaya varieties if $f$ and $q$ are strict morphisms \footnote{Notice when $X$ is noncommutative, then $\PP V$ has to be noncommutative also if $f$ is strict.}, and all the conditions of inclusions and expected dimensions hold for the underlying varieties.
\end{lemma}
Then exactly the same method as in last section gives:
\begin{proposition} Thm. \ref{thm:HPDgen} holds for admissible HP-dual pairs  $X\to \PP V$, $(Y,\sA_Y) \to \PP V^*$ and $S\to \PP V^*$, $(T,\sA_T) \to \PP V$, where $(Y,\sA_Y) $ and $(T,\sA_T)$ are Azumaya varieties.
\end{proposition}

\subsubsection{Categorical HP-duals.} \label{sec:catHPD}The other extreme case is we do not require $Y$ to be geometric at all, and regard HP-dual of $X \to \PP V$ purely as the $\PP V^*$-linear category $\sC$ in (\ref{sod:H_X}). Let's denote the HP-dual category of $S \to \PP V^*$ to be $\sD$, which is $\PP V$-linear. Then since HP-dual category always exists, Thm. \ref{thm:HPDgen} holds almost without conditions.

\begin{proposition}\label{prop:HPDcat} Let $X \to \PP V$ and $S \to \PP V^*$ be morphisms from smooth varieties, with Lefschetz decomposition (\ref{lef:X}) and resp. (\ref{lef:S}). Then there are decompositions
	\begin{align*} 
	 \sC & = \langle \shB^1 (2-N), \cdots, \shB^{N-2}(-1),\shB^{N-1} \rangle, \quad \shB^1 \subset \shB^2 \subset \cdots \subset\shB^{N-1} = \shA_0, \\
	\sD & = \langle \shD^1(2-N), \cdots, \shD^{N-2}(-1), \shD^{N-1} \rangle, \quad \shD^1 \subset \shD^2 \subset \cdots \subset\shD^{N-2} = \shC_0, 
	\end{align*} 
where $\shB^k$ and resp. $\shD^k$ are defined by (\ref{def:B^k}) and resp. (\ref{def:D^k}). Assume $\shH \ne X \times S$ \footnote{ Recall $\shH : = Q(X,S) \subset X \times S$ defined by incidence relation $\{(x, s)~|~ s(x)= 0\}$.}. Then we have decompositions
	\begin{align*}
		\sD_X & = \langle \sE, ~\shA_1(1)\boxtimes \shD^1, \shA_2(2) \boxtimes \shD^2, \ldots, \shA_{i-1}(i-1) \boxtimes 		\shD^{i-1} \rangle,  \\
		\sC_S  & =  \langle \shB^1 \boxtimes \shC_{1}^L(2-l) , \shB^2 \boxtimes \shC_{2}^L(3-l), \ldots, \shB^{l-1}\boxtimes 		\shC_{l-1}^L, ~\sE \rangle,
	\end{align*}
where  $\sE$ is the same triangulated category, $\sD_X$ (resp. $\sC_S$) is the base-change category of $\sD$ (resp. $\sC$) along $X \to \PP V$ (resp. $S \to \PP V^*$) given by Prop. \ref{prop:bcsod}.
\end{proposition}

If we assume one of the pair to be geometric, say $T$, then $\sD_X$ can be realized geometrically by $D(X_T)$ if the $X_T$ satisfies admissible condition. In particular, the Lefschetz type decomposition results always hold for the admissible general section $X_T$:

\begin{corollary}\label{cor:HPDdec} Let $X \to \PP V$ be a morphism from a smooth projective variety with a Lefschetz decomposition (\ref{lef:X}). Let $(S,T)$ be a HP-pair with respect to Lefschetz decomposition (\ref{lef:S}). Assume $\shH \ne X \times S$. Then if $X_T$ is of expected dimension, and whether $Q(S,T) = S \times T$ (e.g. $S=L$, $T = L^\perp$) or $Q(S,T) \ne S \times T$ but $Q(X_T,S)$ is a divisor in $X_T \times S$. Then Thm. \ref{thm:HPDgen} holds if $D(Y)$ (resp. $D(Y_S)$) is replaced by $\sC$ (resp. $\sC_S$). In particular, we have Lefschetz type decompositions for the fiber product $X_T$:
	$$D(X_T) = \langle \sE, ~\shA_1(1)\boxtimes \shD^1, \shA_2(2) \boxtimes \shD^2, \ldots, \shA_{i-1}(i-1) \boxtimes 		\shD^{i-1} \rangle. $$
\end{corollary}
Notice the ambient part of the above decompositions can be zero. For example, if $S$ is rectangular, then the ambient part is non-zero precisely if $l < i$. Since the corollary requires no condition on $X$ at all except from the fact that $X$ admits decomposition (\ref{lef:X}), hence it is a very powerful tool in exploring the properties of Lefschetz decompositions. For example, assume $X$ to be connected, and take $S=L$ to be a generic $(i-2)$-dimensional linear subspace, then we obtain the the length $i$ of a Lefschetz decomposition of a projective variety $X$ should be smaller than or equal to $\dim X + 1$, cf. \cite[Prop. 7.6]{Kuz07HPD}.

\subsubsection{General noncommutative HP-duals} \label{sec:gen} Most interesting cases are in between the two cases, where the HP-dual category is equivalent to a bounded derived category $D(Y,\sR_Y)$ of coherent $\sR_Y$-modules on an algebraic algebraic variety $Y$, with $Y \to \PP V^*$, where $\sR_Y$ is a finite (noncommutative) $\sO_Y$-algebra on $Y$, and the equivalence is given by Fourier-Mukai kernel is an object of derived category $D(Q(X,Y),\sR)$ of certain non-commutative incidence loci. Usually $\sR_Y$ is not an Azumaya algebra, but locally isomorphic to a matrix algebra in an open subset of $Y$, typically on smooth loci of $Y$. Typical situations are when $(Y,\sR_Y)$ is given by the categorical resolution of singularities of the variety $Y$. 

Although so far we do not have as a complete framework for these noncommutative varieties as for Azumaya varieties, especially for the theory about fibre products, base-change properties and criterion, etc, we can apply results of last section to the HP-dual categories, and work out what is the corresponding base-change categories in concrete examples. This type of arguments have already been used by Kuznetsov extensively in \cite{Kuz14SODinAG}. And as long as a proper framework has been set up, we expect our arguments in section \ref{sec:main} apply directly.

\subsection{Applications: HP-dual for universal linear sections} 
As a first application, we show a relative version of HP-duality theorem can follow from Thm. \ref{thm:HPDgen}, and using this result we show the HP-duality theorem for universal linear sections. The more general relative version will be treated in \S \ref{sec:relative}. For simplicity of notations we denote $P = \PP V$, $P^* = \PP V^*$.

Consider the following relative setting in Kuzentsov's \cite{Kuz07HPD}. Let $X$ be a smooth projective variety, with a non-degenerate morphism $f: X \to P$,  equipped with a Lefschetz decomposition (\ref{lef:X}) with respect to $\sO_X(1) = f^* \sO_{\PP V} (1)$, and $g:Y \to P^*$ is the HP-dual. Let $B$ be a smooth algebraic variety, and $L \subset \sO_B \otimes V^*$ be a vector subbundle of rank $l$, and $L^\perp \subset B \otimes V$ the orthogonal vector subbundle over $B$. Denote also$P_B : = P \times B$, $P_B^* = P^* \times B$.

\begin{theorem}[Kuznetsov {\cite[Thm. 6.27]{Kuz07HPD}}] \label{thm:HPD_rel} Assume the following fibre products
	$$\sX_{L^\perp} :=  (X \times B) \times_{P_B} \PP(L^\perp), \qquad \sY_{L} := (Y \times B) \times_{P^*_B} \PP(L)$$
are of expected dimensions. Then there are semiorthogonal decompositions
				\begin{align*}
				D(\sX_{L^\perp}) & = \langle \sC_L, \shA_l(l) \boxtimes D(B),  \ldots, \shA_{i-1}(i-1)\boxtimes D(B) \rangle, \\
				D(\sY_L) & = \langle \shB^1\boxtimes D(B)(2-l), \ldots, \shB^{l-1} \boxtimes D(B), \sC_L \rangle
				\end{align*}
			with a same triangulated category $\sC_L$ as the 'primitive' parts.				
\end{theorem}

\begin{proof} Let $S = \PP(L) \subset P^* \times B$ with the projection map to $p: S \to P^*$, and let $T = \PP(L^\perp) \subset P \times B$ with canonical map $q: T \to P$. Then a relative version of Orlov-type result shows $T$ is HP-dual to $S$ with the following semiorthogonal decomposition of the projective bundle $S= \PP(L)$ over $B$ (cf. Cor. \ref{cor:linear-duality}):
	$$D(S) = \langle D(B), D(B)(1), \ldots, D(B)(l-1) \rangle.$$
Now the result follows from applying Thm. \ref{thm:HPDgen} to the two HP-dual pairs $X \to P$, $Y \to P^*$ and $S \to P^*$, $T \to P$.
\end{proof}

\begin{theorem}[Duality of universal linear sections, Kuznetsov {\cite[\S 6]{Kuz07HPD}}] \label{thm:univ_lin_sec} Let $G_l : = \Gr(l, V^*)$ be the Grassmannian of linear $l$-dimensional subspaces of $V^*$, and $L_l \subset V^* \otimes \sO_{G_l}$ tautological rank $l$ sub-bundle, and $L_l^\perp \subset V \otimes \sO_{G_l}$ the orthogonal bundle, i.e. $L_l^\perp := \Ker ( V \otimes \sO_{G_l} \to L_l^*)$. Define the \textbf{universal linear sections} of $X$ and $Y$ to be
	\begin{align*} 
		\sX_l : &= (X \times G_l) \times _{P \times G_l} \PP(L_l^\perp) \subset X \times G_l ;\\
		\sY_l: & = (Y \times G_l) \times_{P^* \times G_l} \PP(L_l) \subset Y \times G_l.
	\end{align*}
Then we have semiorthogonal decompositions
	\begin{align*}
		D(\sX_l) & = \big\langle \sC_l ,~ \shA_l(l) \boxtimes D(G_l), \ldots, \shA_{i-1}(i-1) \boxtimes D(G_l) \big\rangle, \\
		 D(\sY_l) & =  \big\langle \shB^1 \boxtimes D(G_l)(2-l), \ldots, \shB^{l-1} \boxtimes D(G_l), ~\sC_l \big\rangle.
	\end{align*}	
\end{theorem}
\begin{proof} $f: X \to P$ and $Y \to P^*$ are HP-dual. Notice since over generic points $[L] \in G_l$, the fibers of $\sX_l$ and $\sY_l$, which are the linear sections $X_{L^\perp}$ and $Y_L$, are of expected dimensions. Hence the smooth projective varieties $\sX_l$ and $\sY_l$ are of expected dimensions. Now apply previous theorem to $B = G_l$, which amounts to apply Thm. \ref{thm:HPDgen} to the two HP-dual pairs $X \to P$, $Y \to P^*$ and $S = \PP(L_l) \to P^*$, $T = \PP (L_l^\perp) \to P$. 
\end{proof}

\subsection{Application: duality} \label{sec:duality} As a second application, we reprove that 'HP-duality is indeed a duality relation' based on our theorem. Recall to say $g: Y \to P^*$ is HP-dual to $f: X \to P$ with respect to a decomposition (\ref{lef:X}), is to specify a Fourier-Mukai kernel $\shP  \in D(Y \times_{P^*} \shH_X)$ which gives a fully faithful embedding $\Phi_{\shP} : D(Y) \to D(\shH_X)$ with image the HP-dual category $\sC$ in (\ref{sod:H_X}). The following is a refinement of {\cite[Thm. 7.3]{Kuz07HPD}}:


\begin{theorem}[Duality] \label{thm:duality} Assume $g: Y \to P^*$ is HP-dual to $f: X \to P$ with respect to a Lefschetz decomposition (\ref{lef:X}), where the embedding $D(Y) \to D(\shH_X)$ is given by a $P^*$-linear Fourier-Mukai transform with kernel
	$$\shP   \in D(Y \times_{P^*} \shH_X) = D(Q(X,Y)).$$
Then $f: X\to P$ is HP-dual to $g:Y \to P^*$ with respect to the Lefschetz decomposition 
	\begin{equation}\label{lef:Y:dual}
	D(Y) = \langle (\shB^{N-1})^\vee, (\shB^{N-2})^\vee (1), \ldots, (\shB^1)^\vee (N-2) \rangle,
	\end{equation}
which is the dual of the decomposition (\ref{lef:Y}) obtained in Thm. \ref{thm:HPDgen}, and $(-)^\vee$ means the anti-autoequivalence $R\sHom_Y(-, \sO_Y): D(Y)^{op} \to D(Y)$, and the $P$-linear embedding $D(X) \to D(\shH_Y)$ is given by the Fourier-Mukai transform with the same kernel
	$$\shP \in D(X \times_{P} \shH_Y ) = D(Q(X,Y)).$$
\end{theorem}

The first part of the theorem, $f: X \to P$ is HP-dual to $g: Y \to P^*$, is Kuznetsov's {\cite[Thm. 7.3]{Kuz07HPD}}, and the proof is given by setting $l = N-1$ in Thm. \ref{thm:univ_lin_sec}. The new part is the statement that exactly same Fourier-Mukai kernel $\shP$ gives rise to the embedding $D(X) \to D(\shH_X)$, which is based on computations using Thm. \ref{thm:HPDgen}.

\begin{proof} As the proof of Thm. \ref{thm:HPD_rel} and \ref{thm:univ_lin_sec}, apply Thm. \ref{thm:HPDgen} to $f: X \to P$, $g: Y \to P^*$, and $q: T= \PP (L_{N-1}^\perp) = P \xrightarrow{\id} P$, $p: S= \PP(L_{N-1}) = Q \to P^*$, where $Q \subset P \times P^*$ is the universal quadric. Note $T=P \to P$ is HP-dual to $S=Q \to P^*$ (linear duality) with respect to
	$$D(Q) = \langle D(P), D(P)(1), \ldots, D(P)(N-2)\rangle,$$
i.e. the decomposition (\ref{lef:S}) of $D(S)$ is rectangular with $\shC = D(P)$, $l = N-1$.
\begin{align*}
\begin{ytableau}
\none[D(S) = D(Q)]
&\none 
&\none
&\scriptstyle 
&\scriptstyle 
&\scriptstyle 
&\none[\cdots]
&\scriptstyle 
&\scriptstyle
&\scriptstyle 
&*(lightgray) \scriptstyle 
&\none &\none
&\none[D(T)=D(P)]
\end{ytableau}
\end{align*}
The Fourier-Mukai transform $D(P) \to D(\shH_Q)$ is actually $P\times P$-linear, rather than just $P$-linear, with the kernel given by $\sO_Q \in D(\shH_Q \times_{P \times P} \Delta_P) = D(Q)$. 
Then Thm. \ref{thm:HPDgen} implies we have decompositions
	\begin{align}
	D(X_T) & = D(X) = \sE_{X_T}, \\
	D(Y_S) & = D(\shH_Y) = \langle \shB^1(3-N) \boxtimes D(P)  , \ldots, \shB^{N-2}  \boxtimes D(P), ~\sE_{Y_S}\rangle, \label{duallef:H_Y}
	\end{align}
and $\sE_{X_T} \simeq \sE_{Y_S}$. By Rmk. \ref{rmk:FM} the $P$-linear fully faithful embedding $D(X) = \sE_{X_T} \to D(\shH_Y)$ is given by
	$$ \Phi^L_{\shP|Q} \circ \Phi_{\sO_Q|X}: D(X) \to D(\shH_Y)$$
with image $\sE_{Y_S}$. To obtain the dual decomposition of (\ref{duallef:H_Y}), we define
	$$F: = \otimes \sO_{P^*}(1) \circ (-)^\vee \circ \Phi^L_{\shP|Q} \circ \Phi_{\sO_Q|X} \circ (-)^\vee: D(X) \to D(\shH_Y),$$
where $(-)^\vee= R\sHom( -, \sO)$ is dual functor on the respective spaces\footnote{Note we also take dual $(-)^\vee$ on $D(X)$ to make $F$ a covariant rather than a contravariant functor.}. Then $F$ is also a $P$-linear fully faithful functor which embeds $D(X)$ into $D(\shH_Y)$. Taking dual of the  decomposition (\ref{duallef:H_Y}) of $D(\shH_Y)$, then tensoring $\sO(1)$, we obtain $F$ induces decomposition
	\begin{equation}\label{eqn:sod:F}
	D(\shH_Y) = \langle F(D(X)), ~(\shB^{N-2})^\vee(1) \boxtimes D(P), \ldots, (\shB^1)^\vee (N-2) \boxtimes D(P)\rangle.
	\end{equation}
It remains to compute the Fourier-Mukai kernel $\shP^\dagger$ of $F$. Notice $F$ can be expressed as, 
 	$$F = \otimes \sO_{P^*}(1) \circ (\Phi^L_{\shP|Q})^{op} \circ  \Phi_{\sO_Q|X}^{op}: D(X) \to D(\shH_Y),$$
where for a functor $\Phi$, $\Phi^{op}$ is the \textbf{opposite} functor defined in Def. \ref{def:op}.
	 
First, we compute $\Phi_{\sO_Q|X}^{op}$. Note the kernel $\sO_Q|_X = f''^* \sO_Q \in \shD(\shH_X)$ is obtained from the kernel $\sO_Q \in D(Q)$ by base-change along $f:X \to P$. The map $f''$ is base-change map obtained from $f$ as in the following diagram:
\begin{equation}
\begin{tikzcd}[back line/.style={dashed}, row sep=1.5 em, column sep=3.0 em]
Q\ar[crossing over,hook]{r}{j} \ar{dd}{p}	& \shH_Q	\ar[back line]{dd}\\
	&& \shH_X \ar[crossing over]{llu}[near start]{f''} \ar[hook]{r}[swap]{\Gamma'_f}	&\shH_{X,Q} \ar{llu}[swap]{f'}	\ar{dd} \\
	P \ar[back line, hook]{r}{\Delta_P} 	& P \times P  \ar[leftarrow, back line]{rrd}[very near start]{f \times 1}	\\
				&&	X \ar[crossing over]{llu}{f} \ar[hook]{r}[swap]{\Gamma_f} \ar[leftarrow, crossing over]{uu}[swap]{p_X}	& X \times P  
\end{tikzcd}
\end{equation}
Therefore $\sO_Q|_X = \sO_{\shH_X}$, and $\Phi_{\sO_Q|X} = \Phi_{\sO_{\shH_X}}: D(X) \to D(\shH_{X,Q})$ is $X\times P$-linear. Hence the oposite is $\Phi_{\sO_Q|X}^{op} = \Phi_{\sO_{\shH_X}}^{op} = \Phi_{\sO_{\shH_X}^{op}}$, where by Lem. \ref{lem:op},
	$$\sO_{\shH_X}^{op} =\omega_{\Gamma_f'} [\dim \Gamma_f'] \in D(\shH_X).$$
Here notice $\shH_X$ and $\shH_{X,Q}$ are both smooth. (Notice the map $\shH_{X,Q} \to X$ is base-changed from $\shH_Q \to P$, hence smooth, therefore $\shH_{X,Q}$ is smooth.)

Second, $(\Phi^L_{\shP|Q})^{op}$ is obtained from $(\Phi^L_{\shP})^{op} = \Phi_{(\shP^L)^{op}}$ by base change along $q: Q\to P^*$:
\begin{equation}\label{diagram:base_change_Y_Q}
\begin{tikzcd}[back line/.style={dashed}, row sep=2.6 em, column sep=2.6 em]
Q(X,Y)\ar{r}{p_1} \ar{dd}	& \shH_X	\ar[back line]{dd}[swap]{p_2}\\
	&& Q(X,\shH_Y) \ar[crossing over]{llu}[near start]{q''} \ar{r}[swap]{p_1'}	&\shH_{X,Q} \ar{llu}[swap]{q'}	\ar{dd}{p_2'} \\
	Y \ar[back line]{r}{g}	& P^*\ar[leftarrow, back line]{rrd}[near start]{q}	\\
				&&	\shH_Y = Y_Q \ar{llu}{q_Y} \ar{r}[swap]{g_Q} \ar[leftarrow, crossing over]{uu}	& Q ,
\end{tikzcd}
\end{equation}
From Lem. \ref{lem:op_of_L} we have
	$$(\shP^L)^{op} = \shP \in D(Q(X,Y)),$$
Then $(\Phi^L_{\shP|Q})^{op}: D(\shH_{X,Q}) \to D(Y)$ is $Q$-linear, given by the kernel $q''^* \shP \in D(Q(X,\shH_Y))$. 

\medskip
Third, to compute the kernel of $(\Phi^L_{\shP|Q})^{op} \circ  \Phi_{\sO_Q|X}^{op}: D(X) \to D(\shH_{X,Q}) \to D(\shH_Y)$, denoted by $\shP_1$, we use Lem. \ref{lem:rel_convolution} for base $S = X \times P$, $T = Q$, $U = P$. We have diagram

\begin{equation}\label{diagram:composition}
\begin{tikzcd}[back line/.style={}]		
		&	& Q(X,Y)   \ar{ld}[swap]{p_1} \ar[equal]{d} \ar[hook]{dr}{\Gamma_{f}''}	  \\
		& \shH_X   \ar{ld} \ar[back line, near end, hook]{rd}[swap]{\Gamma_f'} & Q(X,Y) \ar[crossing over]{lld} \ar[back line]{rrd} & Q(X,\shH_Y)  \ar[crossing over]{ld}{p_1'} \ar{rd} \\
	X \ar[hook]{rd}{\Gamma_f}	&	& \shH_{X,Q} \ar{ld}  \ar{rd}	&	& \shH_Y \ar{ld}{g_Q}\\
		&X\times P \ar{r}	& P	&Q  \ar{l} 	
\end{tikzcd}
\end{equation}

Note that it is essential to utilize the $X\times P$-(resp. $Q$-) linearity of  $\Phi_{\sO_Q|X}^{op}$ (resp. $(\Phi^L_{\shP|Q})^{op}$): the condition of Lem. \ref{lem:rel_convolution} is not satisfied for the bases $S= T= P$.  

To check the condition of Lem. \ref{lem:rel_convolution} for the above diagram is satisfied, consider the following splitting diagram:
\begin{equation}\label{diagram:splitting}
\begin{tikzcd}[back line/.style={dashed}, row sep=3 em, column sep=2.6 em]
	Q(X,Y) \ar[hook]{r}[swap]{\Gamma_f''} 	\ar{d}{p_1} \ar[bend left = 15]{rr}{\id}	&Q(X,\shH_Y) \ar{r}[swap]{q''} 	\ar{d}{p_1'}	&Q(X,Y) 	\ar{d}{p_1}\\
	\shH_X  \ar[hook]{r}{\Gamma_f'} \ar[bend right = 15]{rr}[swap]{\id}	& \shH_{X,Q}   \ar{r}{q'} 	 &\shH_X 
\end{tikzcd}
\end{equation}
Since the ambient square is identity, and the right square is Tor-independent since $q$ and therefore $q', q''$ are flat, so the left square is Tor-independent, which is the square \ref{eqn:convolution_sq} of Lem. \ref{lem:rel_convolution} in the current case. Hence the condition of  Lem. \ref{lem:rel_convolution} is satisfied, and the composition is given by the kernel
	$$\shP_1 = p_1^* \, \omega_{\Gamma_f'} [\dim \Gamma_f']  \otimes \Gamma_f''^*\, q''^* \shP \in D(Q(X,Y)).$$
Note $\Gamma_f''^*\, q''^*= \id$ and $\omega_{\Gamma_f'} = \Gamma_f'^* \, \omega_{q'}^\vee$ by the splitting diagram (\ref{diagram:splitting}). However $\omega_{q'} =  p_2'^*\, \omega_{Q/P^*}= p_2'^* \sO_{Q}(1-N,1)$, where $p_2': \shH_{X,Q} \to Q$ is the projection, cf. diagram (\ref{diagram:base_change_Y_Q}). Therefore
	$$\shP_1  = \shP \otimes \sO_{P \times P^*}(N-1,-1)[2-N] \in D(Q(X,Y)),$$
and $F = \otimes \sO_{P^*}(1) \circ \Phi_{\shP_1}$ is given by the kernel
	$$\shP^\dagger = \shP \otimes \sO_{X}(N-1) [2-N]  \in D(Q(X,Y)).$$ 
Therefore $\Phi^{X \to \shH_Y}_{\shP}: D(X) \to D(\shH_Y)$ for $\shP \in D(X \times_{P} \shH_Y)$ is related to $F = \Phi_{\shP^\dagger}$ through
	$$\Phi^{X \to \shH_Y}_\shP = F \circ \otimes \sO_X(1-N) \circ [N-2]: D(X) \to D(\shH_X).$$
Hence $\Phi^{X \to \shH_Y}_{\shP}$ is $P$-linear and fully faithful, satisfies $\Phi^{X \to \shH_Y}_\shP(D(X)) = F(D(X))$. From (\ref{eqn:sod:F}) we have
	$$D(\shH_Y) = \langle \Phi_{\shP}^{X \to \shH_Y}(D(X)), ~(\shB^{N-2})^\vee(1) \boxtimes D(P), \ldots, (\shB^1)^\vee (N-2) \boxtimes D(P)\rangle.$$
Hence $f: X \to P$ is HP-dual to $g: Y \to P^*$ with respect to the Lefschetz decomposition (\ref{lef:Y:dual}) of $Y$, with a $P$-linear embedding $D(X) \to D(\shH_X)$ given by the same $\shP \in D(Q(X,Y))$. 
\end{proof}

\subsection{Examples} \label{sec:examples}
\begin{example}[Intersections of Quadrics] Let $X= Q_1^{2m-1} \subset \PP^{2m}$ be a $(2m-1)$ dimensional quadric, $m \in \ZZ_{\ge 1}$, with a Lefschetz decomposition $\shA_0 = \langle \slashed{S}_X, \sO_X \rangle$, and $\shA_{1} = \shA_{2} = \ldots = \shA_{2m-2} = \langle \sO_X \rangle$. Then its HP-dual is a $2m$-dimensional quadric $Y = Q_1^{2m}$ with a ramified double covering map onto the dual projective space $\check{\PP}^{2m}$, ramified over the dual quadric $\check{Q}_1^{2m-1} \subset \check{\PP}^{2m}$ of $Q_1^{2m-1}$. The decompositions are indicated by
\begin{align*}
\begin{ytableau}
\none[\ \ \ \  \sO_X]&\none
&*(lightgray)
&*(lightgray)
&*(lightgray)
&\none&\none[\cdots]&\none
&*(lightgray)
&*(lightgray)
&&
&\none[\ \ \ \  \slashed{S}_Y] \\
\none[\ \ \ \  \slashed{S}_X]&\none
&*(lightgray)
&&
&\none&\none[\cdots]&\none
&&&&&
\none[\ \ \ \  \sO_Y]&\none
\end{ytableau}
\end{align*}
Let $T = Q_2^{2m-1}\subset \PP^{2m}$ and $S=Q_2^{2m} \to \check{\PP}^{2m}$ be another such pair, with decomposition
\begin{align*}
\begin{ytableau}
\none[\ \ \ \  \sO_S]&\none
&
&
&
&\none&\none[\cdots]&\none
&
&
&*(lightgray)&*(lightgray)
&\none[\ \ \ \  \slashed{S}_T] \\
\none[\ \ \ \  \slashed{S}_S]&\none
&
&*(lightgray)
&*(lightgray)
&\none&\none[\cdots]&\none
&*(lightgray)&*(lightgray)&*(lightgray)&*(lightgray)&
\none[\ \ \ \  \sO_T]&\none
\end{ytableau}
\end{align*}

Assume $X$ and $T$ intersects transversely.\footnote{We actually don't need this, since both $X$ and $S$ are non-degenerate, then admissibility conditions are satisfied as long as $X \cap T$ and $Y \times_{\check{\PP}^{2m}} S$ are of expected dimensions ($2m-2$ and resp. $2m$).} Then so is the two maps $Y\to \check{\PP}^{2m}$ and $S \to \check{\PP}^{2m}$. Our theorem implies there is decomposition
	\begin{align*}
	D(Q_1^{2m-1} \cap Q_2^{2m-1} )  = \big\langle \sE, \sO(1), \ldots, \sO(2m-3) \big\rangle.
	\end{align*}
This agrees with the decomposition given by \cite[Cor. 5.7]{Kuz08quadric}, and from which we know the essential part is given by $\sE \simeq D(C)$, where $C$ is an orbifold $\PP^1$ with $\ZZ/2 \ZZ$-stack structure over $2m+1$ points. Then second decomposition of our theorem implies there is a decomposition
	\begin{align*}	
		D(Q_1^{2m} \times_{\check{\PP}^{2m}} Q_2^{2m})  =\big\langle \langle \slashed{S},\sO \rangle (1-2m), \sO(2-2m), \ldots, \sO(-2), \langle \slashed{S},\sO \rangle (-1),  D(C) \big\rangle,
	\end{align*}
where $Q_1^{2m} \times_{\check{\PP}^{2m}} Q_2^{2m}$ is a smooth $2m$-dimensional manifold which admits degree $4$ finite surjection onto $\PP^{2m}$.

Similarly, for intersection of two even dimensional quadrics $Q_k^{2m} \subset \PP^{2m+1}$, $k=1,2$, our theorem implies there is a decomposition
	$$D(Q_1^{2m} \cap Q_2^{2m} )  = \big\langle \sE', \sO(1), \ldots, \sO(2m-2) \big\rangle.$$
This agrees with \cite[Cor. 5.7]{Kuz08quadric}, and from which we know $\sE' \simeq D(C')$, where $C'$ is a hyperelliptic curve ramified over $2m+2$ points on  $\PP^1$. Then the second decomposition of the theorem simply states similar thing happens to the dual quadrics:
	$$D(\check{Q}_1^{2m} \cap \check{Q}_2^{2m}) =  \big\langle  \sO(2-2m), \ldots, \sO(-1) , \sE' \big\rangle.$$
Where $\sE'$ is given by the derived category of the same hyperelliptic curve $C'$. The game can be also played for two odd dimensional quadrics with ramified double covering over the projective spaces, and we omit the details.

\end{example}

\begin{example}[Quadric sections of Pfaffian-Grassmannian duality]\leavevmode

\medskip\noindent $(1)$. Let $X = \Gr(2,5) \subset \PP(\wedge^2 \CC^5) = \PP^9$, with Lefschetz decomposition 
		$$D(X) = \big\langle \shA, \shA(1), \shA(2),\shA(3), \shA(4) \big\rangle,$$ 
	where $\shA = \langle \sU, \sO \rangle$. Then it is HP-dual to $Y = \Pf(2,5) = {\Gr}(2,5) \subset \check{\PP}^9$.
	\begin{align*}
	\begin{ytableau}
	\none[\ \ \ \  \sO_X]
	&\none
	&*(lightgray)&*(lightgray)&*(lightgray)&*(lightgray)&*(lightgray)
	&&&&&&
	\none[\ \ \ \  \sU_Y]
	\\
	\none[\ \ \ \  \sU_X]
	&\none
	&*(lightgray)&*(lightgray)&*(lightgray)&*(lightgray)&*(lightgray)
	&&&&&&
	\none[\ \ \ \  \sO_Y]
	\end{ytableau}
	\end{align*}
Let $T = Q^{8}\subset \PP^9$, $S = \check{Q}^8 \subset \check{\PP}^9$, with
	\begin{align*}
	\begin{ytableau}
	\none[\ \ \ \  \sO_S]
	&\none
	&&&&&&&&&*(lightgray)&*(lightgray)
	&\none[\ \ \ \ \slashed{S}_T]
	\\
	\none[\ \ \ \  \slashed{S}_S]
	&\none
	&&&*(lightgray)&*(lightgray)&*(lightgray)&*(lightgray)&*(lightgray)&*(lightgray)&*(lightgray)&*(lightgray)
	&\none[\ \ \ \  \sO_T]
	\end{ytableau}
	\end{align*}
then as long as $X \cap T$ and $Y \cap S$ are of expected dimension, we have
		$$D(\Gr(2,5) \cap Q^8) = \big\langle \sE, \shA, \shA(1), \shA(2) \big\rangle,$$
which agrees with \cite[Cor. 4.4]{Kuz15CY} and from which we know $\sE$ satisfies $S_{\sE}^2 = [4]$ \footnote{ $\sE$ is \textit{not} a $2$-Calabi-Yau category. The Serre functor is $S_{\sE} = \sigma \circ [2]$, where $[2]$ is the shift functor and $\sigma$ is a non-trivial involution. The reason is, if $\sigma$ is trivial, then by $2$-Calabi-Yau property $\HH_{-2}(M) \simeq \HH^0(M) \ne 0$, where $M = \Gr(2,5) \cap Q^8$, but computation shows $\HH_{-2}(M) = 0$, cf. Prop. 4.5, \cite{KP16}.}. We also similar decomposition for the dual intersections:
		$$D(\Pf(2,5) \cap \check{Q}^8 )= \big\langle \shB(-2), \shB(-1), \shB, \sE \big\rangle, $$
where $\sE$ is given by the same category.

	If we let $T = Q^{9} \to \PP^9$, $S = \check{Q}^{9} \to \check{\PP}^9$ to be $9$-dimensional quadric with ramified double coverings over $\PP^9$ resp. $\check{\PP}^9$ ramified over $Q^8$ resp. $\check{Q}^8$, then 
	\begin{align*}
	\begin{ytableau}
	\none[\ \ \ \  \sO_S]
	&\none
	&&&&&&&&&&*(lightgray)
	&\none[\ \ \ \ \slashed{S}_T]
	\\
	\none[\ \ \ \  \slashed{S}_S]
	&\none
	&&*(lightgray)&*(lightgray)&*(lightgray)&*(lightgray)&*(lightgray)&*(lightgray)&*(lightgray)&*(lightgray)&*(lightgray)
	&\none[\ \ \ \  \sO_T]
	\end{ytableau}
	\end{align*}
Then we have
		$$D(\Gr(2,5) \times_{\PP^9} Q^9) = \big\langle \sE', \shA, \shA(1), \shA(2), \shA(3) \big\rangle.$$
This agrees with \cite[Cor 4.7]{Kuz15CY}, and from which we know $\sE'$ is a Calabi-Yau category of dimension $2$. We have similar decompositions for the dual intersections $\Gr(2,5) \times_{\check{\PP}^9} \check{Q}^9$. The fibre products $\Gr(2,5) \times_{\PP^9} Q^9$ and $\Gr(2,5) \times_{\check{\PP}^9} \check{Q}^9$ are called Gushel-Mukai varieties, and were studied by \cite{KP14, KP16}.

\medskip\noindent $(2)$.  Let $X = \Gr(2,7)\subset \PP(\wedge^2 \CC^7) = \PP^{20}$, and $(Y,\sR_Y) \to \check{\PP}^{20}$ be the noncommutative Pfaffian. $T = Q^{19} \subset \PP^{20}$ and $S = Q^{20} \to \check{\PP}^{20}$ ramified double covering as in last example. Since these maps are again non-degenerate, admissibility conditions are satisfied as long as the intersections are of expected dimensions. We have
		$$D(\Gr(2,7) \cap Q^{19})  = \big\langle \sE, \shA, \shA(1), \ldots, \shA(4) \big\rangle,$$
where $\shA = \langle S^2\sU, \sU, \sO \rangle$. This agrees with Cor 4.4 of \cite{Kuz15CY}, and from which we know $\sE$ satisfies $S_{\sE}^2 = [8]$ \footnote{ Similarly, $\sE$ is not a $4$-Calabi-Yau category since the Serre functor is $S_{\sE} = \sigma \circ [4]$, where $\sigma$ is a non-trivial involution. The argument is similar: if $\sigma$ is trivial, then $\HH_{-4}(M) \simeq \HH^0(M) \ne 0$, where $M = \Gr(2,7) \cap Q^{19}$, but from Lefschetz theorem and Schubert calculus we know $\HH_{-4}(M) = 0$, a contradiction.}. We also have
		$$D(Y_S, \sR_S) = \langle \shB(-12), \ldots, \shB(-1),\shB, ~ \sE \rangle,$$
where $(Y_S,\sR_S)$ is the pullback of the noncommutative variety $(Y,\sR_Y)$ along the flat base-change (ramified double covering) $S \to \check{\PP}^{20}$, $\shB \simeq \shA$, and $\sE$ is the same category as above.
	If we take $T = Q^{20} \to \PP^{20}$ and $S = Q^{19} \subset \check{\PP}^{20}$, then we have
		$$D(\Gr(2,7) \times_{\PP^{20}} Q^{20})  = \big\langle \sE', \shA, \shA(1), \ldots, \shA(5) \big\rangle,$$
where $\shA = \langle S^2\sU, \sU, \sO \rangle$, which agrees with \cite[Cor. 4.7]{Kuz15CY} and from which we know $\sE'$ is a Calabi-Yau category of dimension $5$. We also have similar decomposition for the quadric section of noncommutative Pfaffian.

\medskip\noindent $(3)$.  Let $X = \Gr(2,6)\subset \PP(\wedge^2 \CC^6) = \PP^{14}$, and $(Y,\sR_Y) \to \check{\PP}^{14}$ be the noncommutative Pfaffian. $T = Q^{13} \subset \PP^{14}$ and $S = Q^{14} \to \check{\PP}^{14}$. Then
		$$D(\Gr(2,6) \cap Q^{13}) = \big\langle \sE_1, \shA_2, \shA_3(1), \shA_4(2), \shA_5(3) \big\rangle,$$ 
where $\shA_2 = \langle S^2\sU, \sU, \sO \rangle$, $\shA_2 = \shA_3 = \shA_4 = \langle \sU, \sO \rangle$. The problem is that if we only care about semiorthogonal decompositions of the quadric sections of $\Gr(2,6)$, not the derived equivalence, then this is definitely not the best we can do: consider $X = \Gr(2,6) \subset \PP(\CC^{15}) \subset \PP(S^2\CC^{15})$ with the second map the double Veronese embedding, then it is equipped with a rectangular decomposition 
		$$D(\Gr(2,6)) = \big\langle \shB, \shB\otimes \sO_{\PP(S^2\CC^{15})}(1), \shB \otimes \sO_{\PP(S^2\CC^{15})}(2) \big\rangle,$$
where $\shB = \langle \sO, \sU^\vee, S^2 \sU^\vee, \sO_{ \PP(\CC^{15})}(1), \sU^\vee \otimes  \sO_{ \PP(\CC^{15})}(1)\rangle$. Then $X \to  \PP(S^2\CC^{15})$ always have a HP-dual category $\sC$, and $T = H \subset  \PP(S^2\CC^{15})$ is a hyperplane, HP-dual to $S = \pt$.  Then $H \cap  \PP(\CC^{15}) = Q^{13} $ is a quadric, and Cor. \ref{cor:HPDdec} implies
		$$D(\Gr(2,6) \cap Q^{13}) = \big\langle \sE_2, \shB, \shB\otimes  \sO_{\PP(S^2\CC^{15})}(1) \big\rangle.$$
Now $\sE_2$ is a smaller category than $\sE_1$. This agrees with \cite{Kuz15CY}, and from which we know $\sE_2$ is a Calabi-Yau category of dimension $3$.

\end{example}

The readers are invited to apply Thm. \ref{thm:HPDgen} to more examples in \cite{Kuz07HPD}, \cite{Kuz15CY} and \cite{Kuz14SODinAG}.


\section{Relative version}\label{sec:relative}
In this section we prove the relative version of our main result, Thm. \ref{thm:HPDgen:rel}. In the relative theory, we fix a base scheme $B$, and the projective morphism $f: X \to \PP V$ is replaced by a $B$-morphisms $f: X \to \PP_B(E)$, where $E$ is a vector bundle over $B$. Similarly for $Y$, $S$, $T$. Then the main theorem Thm. \ref{thm:HPDgen} holds in this relative settings for their respective fiber products  over $\PP_B(E)$ resp. $\PP_B(E^*)$.

In particular, apply to case of dual linear sections, our result implies the relative HP-duality theorem for linear sections (Thm. \ref{thm:HPD:rel}). This relative HP-duality theorem was first studied in the case when $E$ is trivial bundle by Kuznetsov \cite[Thm. 6.27]{Kuz07HPD} (cf. Thm. \ref{thm:HPD_rel}), later in the case of relative degree $2$ Veronese embeddings by Auel, Bernardara, and Bolognesi \cite[Thm 1.13]{ABB14}, and in a more general situation (including degree $d$ Veronese embeddings, formulated in the framework of gauged Landau-Ginzburg models) using the methods of Variations of Geometric Invariant Theory (VGIT) by Ballard et al. \cite[Thm. 3.1.3]{BDF+}. 

For a fixed base $B$, the upshot is: everything of our theory in absolute case still holds, with all the categories being $B$-linear, and all orthogonal relations $R\Hom_X(-,-) = 0$ replaced by $a_* R\sHom_X(-,-) =0$, where $a: X \to B$ is the structure morphism of a $B$-scheme $X$. 

\medskip\noindent \textbf{Notational Conventions.} Let's fix a base scheme $B$ and assume it is a smooth quasi-projective variety over $\kk$. We will consider the category of $B$-schemes, i.e. schemes with a morphism to $B$, and the morphisms between schemes are compatible with the maps to $B$. For a $B$-scheme $X$, we will denote by $a_X: X \to B$ the structural morphism. For a coherent sheaf $\sF \in \coh(B)$, we denote the associated projective bundle $\PP(\sF): =  \Proj_B \Sym_{\sO_B}^\bullet \sF$. We fix a locally free sheaf $\sE$ is of rank $N$ on $B$, and its corresponding vector bundle $E = V(\sE):= \Spec_{\sO_B} \Sym^\bullet \sE^\vee$. Let $\PP (E) := \PP(\sE^\vee):= \Proj_B \Sym_{\sO_B}^\bullet \sE^\vee$ be the projective bundle of $E$. \footnote{Notice $\PP(E) = \PP(\sE^\vee)$. We define in this way to let $\PP(E)$ parametrise one dimensional subbundles of $E$ rather than one dimensional quotients, hence agree with absolute notation $\PP(V)$.} Denote $E^*$ the dual vector bundle of $E$ on $B$, i.e. $E^* = V(\sE^\vee)$. Sometimes we will write $P=P_B = \PP(E)$ and $P^*=P_B^* = \PP(E^*)$ if there is no confusion. Given a $B$-morphism $f: X \to Y$ and a point $b \in B$, we denote $X_b : = \Spec k(b) \times_B X$ (resp. $Y_b: = \Spec k(b) \times_B Y$) the fiber of $X$ (resp. $Y$) over $b$, where $k(b)$ is the residue field of $b$, and $f_b : X_b \to Y_b$ the induced map of $f$ on the fibers. Note the fiber $\PP(E)_b$ is just the projective space $\PP(E_b)$, where $E_b := E \otimes k(b)$ is the fiber of the vector bundle $E$ over $b$.

\subsection{Relative semiorthogonal decompositions} Assume $X$ is a smooth variety with a map  $a_X: X \to B$ to the base scheme $B$.
\begin{definition}  A semiorthogonal decomposition $D(X) = \langle \shA_1, \ldots, \shA_n \rangle$ is said to be \textbf{$B$-linear} if every $\shA_k$ is a $B$-linear subcategory of $D(X)$, $k=1,\ldots, n$.
\end{definition}

\begin{lemma} Let $\shA, \shB \subset D(X)$ be $B$-linear subcategories. Then the following two conditions are equivalent
\begin{enumerate} 
	\item $R\Hom_X(F, G) = 0$, for all $F\in \shB$, $G \in \shA$.
	\item $a_{X*}\, R\sHom_X(F,G) = 0$,  for all $F \in \shB$, $G \in \shA$. 
\end{enumerate}
\end{lemma}
\begin{proof} If $a_{X*}\, R\sHom_X(F,G) = 0$, taking global section we have $R\Hom_X(F, G) = 0$. On the other hand, if $(1)$ holds, then for every $F\in \shB$, $G\in \shA$, we have
	\begin{align*} & R\Hom(k(b), a_{X*} R\sHom(F, G)) = R\Hom(a_X^* k(b), R\sHom(F,G)) \\
	&= R\Hom(F \otimes a_X^* k(b), G) = 0, \quad \text{for every closed point $b \in B$.}
	\end{align*}
The last equality follows from $B$-linearity of $\shB$. Since $\{k(b)\in D(B)~|~ b \in B\}$ spans $D(B)$ (cf. {\cite[Prop. 3.17]{Huy}}), we have $a_{X*}\, \sHom_X(F, G) = 0$.	
\end{proof}

Therefore for a $B$-linear semiorthogonal decomposition $D(X) = \langle \shA_1, \ldots, \shA_n \rangle$, the semiorthogonal condition $(1)$ of Def. \ref{def:sod} is equivalent to the sheaf version:
	$$a_{X*} R\sHom_X(F ,G) = 0  \quad \text{for all} \quad F \in \shA_k, G \in \shA_l, 1\le l< k \le n.$$

 Let $a_X: X \to B$ and $a_Y: Y \to B$ be two projective morphisms. Consider the diagram
	\begin{equation} \label{eqn:X-times_B-Y}
	\begin{tikzcd}
	X \times_{B} Y\ar{d}{p_X} \ar{r}{p_Y} & Y \ar{d}{a_Y}\\
	X \ar{r}{a_X}          &B
	\end{tikzcd}
	\end{equation}
For $F \in D(X)$, $G \in D(Y)$,  we denote
	$$F \boxtimes_B G : = p_X^* \,F \,\otimes p_Y^* \, G \in D(X\times_B Y).$$
 Here note all the functors are derived functors. Denote by $a: X\times_B Y \to B$ the structure map. If the diagram (\ref{eqn:X-times_B-Y}) is Tor-independent, then for $F \in D(X)$, $G \in D(Y)$, we have $a_*(p_X^*\, F \otimes p_Y^*\, G)  = a_{X*} (F \otimes  p_{X*} p_Y^*\,G) = a_{X*}  (F \otimes  a_X^*\, a_{Y*} G)  =  a_{X*} F \,\otimes a_{Y*}G$, i.e. 
 	\begin{align}\label{eqn:rel:Kunneth}
	 a_*( F \boxtimes_B G) =  a_{X*} F \,\otimes a_{Y*}G \in D(B).
	\end{align}
We will call (\ref{eqn:rel:Kunneth}) the \textbf{relative K\"unneth formula}. Taking global sections, we directly have
	$$R \Gamma(X \times_B Y, F \boxtimes_B G) = R \Gamma(B, a_{X*} F \otimes_{\sO_B} a_{Y*}G).$$
If further $X \times_B Y$ is smooth, the for every $F_1, F_2 \in D(X)$, $G_1, G_2 \in D(Y)$, we have
 	\begin{align}\label{eqn:rel:Kunneth:Hom}a_* R\sHom_{} (F_1\boxtimes_B G_1, F_2 \boxtimes_B G_2) = a_{X*}  R\sHom_X(F_1, F_2) \otimes a_{Y*} R\sHom_Y(G_1, G_2).
	\end{align}

\subsection{Relative HP-dual theory}  Let $f: X \to \PP(E)$ be a $B$-morphism, where $X$ is a smooth $\kk$-variety, proper over $B$.
\begin{definition} A \textbf{$B$-linear Lefschetz decomposition} of $D(X)$ with respect to $\sO_X(1) := f^* \sO_{\PP (E)}(1)$ is a semiorthogonal decomposition of the form
\begin{equation} \label{lef:X:rel} 
	D(X) = \langle \shA_0, \shA_1 (1), \ldots, \shA_{i-1}(i-1) \rangle, \quad \shA_0 \supset \shA_1 \supset \cdots \supset \shA_{i-1}
	\end{equation}
where $\shA_{k}$ are all $B$-linear, $k=0,\ldots, i-1$, and $\shA_k(k)$ means $\shA_k \otimes \sO_X(k)$.
\end{definition}

Then all the constructions and results of \S \ref{sec:lef} and \S \ref{sec:set up} still work, with all subcategories involved being \emph{$B$-linear}, and all orthogonal relations $R\Hom(F,G) = 0$ being replaced by the relative version $a_{X*} R\sHom_X(F,G) = 0$. In particular, the $\shB^k$ defined by (\ref{def:B^k}) are well defined $B$-linear subcategories, with $\shB^{N-1} = \shA_0$. And similarly for $\shD^k$'s.

\begin{definition} Let $Q_B \subset \PP(E) \times_B \PP(E^*)$ be the \textbf{relative incidence quadric} over $B$. Denote by $\shH_X : = X \times_{\PP(E)} Q_B$ the \textbf{relative universal hyperplane} of $f:X \to \PP(E)$. \footnote{Some authors use the symbol $\sX := \shH_X$ and $\sX_0 := Q_B$}
\end{definition}

\begin{remark}$Q_B$ and $\shH_X$ admit several descriptions. $Q_B$ is first of all, the fiberwise universal quadric, whose fiber over $b\in B$ is universal quadric $Q_b = \{ (p, s)~|~s(p) = 0\} \subset  \PP(E_b) \times \PP(E^*_b)$. Alternatively, $Q_B$ is the zero locus of a canonical section $\theta_E$ of the line bundle $\sO_{\PP(E) \times_B \PP(E^*)}(1,1) : = \sO_{\PP(E)}(1) \boxtimes_{B}  \sO_{\PP(E^*)}(1)$, where the section $\theta_E$ under the identification
	\begin{align*}
	 \Gamma(\PP(E) \times_B \PP(E^*), \sO_{\PP(E)}(1) \boxtimes_{B}  \sO_{\PP(E^*)}(1)) = \Gamma(B, E^* \otimes E).
	\end{align*}
corresponds to the identity map $E \xrightarrow{\id} E$. 

Therefore $\shH_X \subset X \times_B \PP(E^*)$ is the fiberwise universal hypersurface $\shH_{X_b} \subset X_b \times \PP (E_b^*)$ for $f_b: X_b \to \PP(E_b)$, $b \in B$, and also the zero loci of the canonical section $\theta_X := f^* \theta_E \in \Gamma(X \times_B \PP(E^*), \sO_{X \times_B \PP(E^*)}(1,1))$, where $\sO_{X \times_B \PP(E^*)}(1,1) = f^*\,\sO_{\PP(E) \times_B \PP(E^*)}(1,1) = \sO_{X}(1) \boxtimes_{B}  \sO_{\PP(E^*)}(1)$. Since $Q_B$ is smooth over $\PP(E)$, $\shH_X = X \times_{\PP(E)} Q_B$ is smooth over X, hence smooth. 
\end{remark}

\begin{lemma} \label{sod:H_X:rel} Assume $f: X \to \PP(E)$ has a $B$-linear Lefschetz decomposition (\ref{lef:X:rel}). Denote $i_{\shH_X}: \shH_X \hookrightarrow X \times_B \PP (E^*)$ the inclusion morphism and $i_{\shH_X}^*$ the derived pullback functor. Then $i_{\shH_X}^*$ is fully faithful on the subcategories $\shA_1 (1)\boxtimes_B D(\PP (E^*)), \ldots, \shA_{i-1}(i-1)\boxtimes_B D(\PP (E^*)) $, and the images induce a $\PP(E^*)$-linear decomposition
	\begin{equation} \label{sod:H_X:rel}
	D(\shH_X) = \langle \sC_B, \shA_1(1) \boxtimes_B D(\PP (E^*)), \cdots, \shA_{i-1}(i-1)\boxtimes_B D(\PP (E^*))\rangle.
	\end{equation}
Here $\shA_k(k)\boxtimes_B D(\PP (E^*))$ denotes $i_{\shH_X}^*(\shA_k(k)\boxtimes_B D(\PP (E^*)))$.
\end{lemma}

\begin{proof} This follows directly from the general vanishing Lem.\ref{lem:van:rel} applied to $S = \PP(E^*)$.
\end{proof}

\begin{definition}[Relative HP-dual] \leavevmode
\begin{enumerate}[leftmargin = *]
	\item The $\sC_B$ in (\ref{sod:H_X:rel}) is called the \textbf{relative HP-dual category} of $f: X \to \PP(E)$ with respect to decomposition (\ref{lef:X:rel}).
	\item A $B$-morphism $g: Y \to \PP(E^*)$ is called the \textbf{relative HP-dual} of $f: X \to \PP(E)$ with respect to decomposition (\ref{lef:X:rel}), if there exists a kernel $\shE \in D(Y \times_{\PP(E^*)} \shH_X) = D(Q_B(X,Y))$ such that the $\PP(E^*)$-linear Fourier-Mukai transform $\Phi_{\shE}: D(Y) \to D(\shH_X)$ is fully faithful, and induces $D(Y) \simeq \sC_B$.
\end{enumerate}
\end{definition}

\begin{remark} In \cite{BDF+} Ballard et al.  introduce the notion \textbf{Weak HP-dual}, where we simply require there is a $\PP(E^*)$-linear embedding $\Phi: D(Y) \to D(\shH_X)$, not necessarily given by a geometric FM transform $\Phi_{\shE}$ with kernel $\shE$ \emph{scheme-theoretically} supported on $Y \times_{\PP(E^*)} \shH_X$. And it is not hard to see our result Thm. \ref{thm:HPDgen:rel} directly applies to weak HP-duals pairs, except for the part of descriptions of FM-kernels.
\end{remark}

\medskip\noindent\textbf{Vanishing results.} Now let $q: S \to \PP(E^*)$ be another $B$-morphism, with $S/\kk$ smooth and $a_S: S \to B$ proper. Let $\shH: = \shH_{X,S} = \shH_X \times_{\PP(E^*)} S$. Then
	$$\shH = X \times_{\PP(E)} \shH_S = X  \times_{\PP(E)} Q_B  \times_{\PP(E^*)} S $$
is the fiberwise universal quadric $Q(X_b, S_b) \subset X_b \times S_b$ defined by $\{(x,s)~|~ s(x) = 0\}$. We also use the notation $Q_B(X,S)$. Again the inclusion $i_{\shH}: \shH \hookrightarrow X \times_B S$ is defined by a canonical section $\theta_{X,S}$ of the line bundle $\sO_{X\times_B S}(1,1)$. 

\begin{lemma}[Relative vanishing lemma] \label{lem:van:rel} Suppose $X\times_B S$ is a smooth $\kk$-variety of expected dimension \footnote{As before, by this we mean $\dim X\times_B S = \dim X + \dim S - \dim B$.}, and $\theta_{X,S} \ne 0$ (i.e. $\shH \ne X \times_B S$). Then for any $F_1, F_2 \in D(X)$, $G_1, G_2 \in D(S)$, we have $a_{X\times S *} R\sHom_{\shH}(i_\shH^* (F_1\boxtimes\, G_1), i_\shH^*(F_2 \boxtimes \,G_2))$ equals the cone of 
    \begin{align}\label{cone:rel}
    \begin{split}
     \big[a_{X*}\, R\sHom_X(F_1(1),& F_2)\,\otimes\,  a_{S*} \, R\sHom_S(G_1(1),G_2)\\
    & \rightarrow a_{X*}\, R\sHom_X(F_1,F_2)\,\otimes \,a_{S*} \,R\sHom_S(G_1,G_2)\big].
    \end{split}
    \end{align}
\end{lemma}

\begin{proof} For any $E_1, E_2 \in D(X \times_B S)$, we have a distinguished triangle
	$$R\sHom(E_1, E_2(-1,-1)) \to R \sHom(E_1, E_2) \to R \sHom(E_1, E_2 \otimes i_{\shH *} \sO_{\shH}) \xrightarrow{[1]}.$$
For last term of the triangle, we have 
	$$R\sHom_{X\times_B S}(E_1, E_2 \otimes  i_{\shH *} \sO_{\shH} ) = R\sHom_{X\times_B S}(E_1, i_{\shH *} i_{\shH}^* E_2) = R\sHom_{\shH}(i_\shH^* E_1, i_\shH^* E_2).$$
Now apply to the case $E_1 = F_1 \boxtimes_B G_1$, $E_2 = F_2 \boxtimes_B  G_2$, and use relative K\"unneth formula (\ref{eqn:rel:Kunneth:Hom}) for the first two terms of the above triangle, we are done.
\end{proof}

Suppose now $q:S \to \PP (E^*)$ admits a $B$-linear Lefschetz decomposition
	  \begin{equation} \label{lef:S:rel}
            D(S) = \langle \shC_{0}, \shC_1(1), \ldots, \shC_{l-1}(l-1)\rangle, 
   	 \end{equation}
with respect to the line bundle $\sO_S(1) = q^* \sO_{\PP (E^*)}(1)$, where $\shC_0 \supset \shC_1 \supset \cdots \supset \shC_{l-1}$  a descending sequence of $B$-linear admissible subcategories, $1 \le l \le N$. 
Then all the '$\alpha$-' and '$\beta$-vanishing' results in the Rmk. \ref{rmk:alpha&beta_van} hold with box-tensor terms '$\shA \boxtimes \shC$' replaced by '$\shA \boxtimes_B \shC$'.

\subsection{Relative version of main theorem}
Suppose the $B$-schemes $X$, $S$ are smooth over $\kk$, proper over $B$, and $g: Y \to \PP(E^*)$ is relative HP-dual to $f: X \to \PP(E)$ with respect to relative Lefschetz decomposition (\ref{lef:X:rel}), and $p: T \to \PP(E)$ is relative HP-dual to $g: S \to \PP(E^*)$ with respect to (\ref{lef:S:rel}). As before, we denote $X_T : = X \times_{\PP(E)} T$ and $Y_S : = Y \times_{\PP(E^*)} S$. From now on we make the following assumption:

\medskip \noindent \emph{Assumption $(\dagger ')$.} We assume $X \times_B S$ is a smooth variety of expected dimension, $\shH \subset X\times_B S$ is of pure codimension one, and $\overline{f(X)}$ (resp. $\overline{q(S)}$) is of relative dimension at least $2$ over $B$, and length of Lefschetz decompositions (\ref{lef:X:rel}) (resp. (\ref{lef:S:rel})) is smaller than $\rank E$.

\begin{definition} \label{def:adm:rel} The two pairs of morphisms $(X \to \PP(E), \,Y \to \PP(E^*))$, and $(S \to \PP(E^*), \,T \to \PP(E))$ are called \textbf{admissible} if the base-change $X \to \PP(E)$ (resp. $S \to \PP(E^*)$) is faithful with respect to $(\shH_S, T )$ (resp. $(\shH_X, Y)$).
\end{definition}

\begin{definition}\label{def:non-deg_rel} A $B$-morphism $f: X \to \PP(E)$ is called \textbf{non-degenerate} if for any closed point $b \in B$ the image $f_b(X_b)$ is not contained in any hyperplane of $\PP(E_b)$. 
\end{definition}

If $X/B$ is proper and flat with integral geometric fibres, then the non-degeneracy of $f$ is equivalent to the following condition: denote the natural map $\lambda: E^* \to a_{X*} \sO_X(1)$ corresponding to $f$, where $\sO_X(1): = f^* \sO_{\PP(E)}(1)$, then for any closed point $b \in B$, the composition $E^*_b \to a_{X*} \sO_X(1) \otimes k(b) \to \Gamma(X_b, \sO_{X_b}(1))$ is an inclusion of vector spaces \footnote{The condition is satisfied, for example, if $a_{X*} \sO_X(1)$ is locally free, and the natural map $\lambda: E^* \subseteq a_{X*} \sO_X(1)$ is an inclusion of vector subbundle.}. This in particular implies $\lambda: E^* \to a_{X*} \sO_X(1)$ is an inclusion of $\sO_B$-submodule. 

\begin{lemma} Suppose $X/B$ is proper and flat with integral geometric fibers, $\pi : \PP(E) \to B$ a projective bundle. If the $B$-morphism $f: X \to \PP(E)$ is \emph{non-degenerate}, then the universal hyperplane $\shH_X $ is \emph{flat} over the dual projective bundle $\PP(E^*)$.
\end{lemma}

\begin{proof}
Consider $F: \shH_X \to \PP(E^*)$, $G: \shH_X \to B$, $\pi': \PP(E^*) \to B$. By \cite[$\mathrm{IV}_3,$ 11.3.11]{EGA}, in order to show $F$ is flat, it suffices to test the flatness on fibers, i.e. to show that $G$ is flat (which follows from $\shH_X \to X$ is smooth, and $X \to B$ is flat), and that $f_b: X_b \to \PP(E^*_b)$ is flat for any $b\in B$, which follows from Lem. \ref{lem:H_flat}.
\end{proof}

The analogue of Lem.\ref{lem:admissible pairs} holds. We mention a few cases and leave the rest to the readers.

\begin{lemma}[Criterion for admissibility] \label{lem:admissible pairs:rel} 
Suppose $X/B$ is proper and flat with integral geometric fibers, and $f: X \to \PP(E)$ is non-degenerate.
	\begin{enumerate}
	\item If $Q_B(S,T) = S \times_B T$  (for example $S = L_B$, $T = L_B^\perp$ dual linear subbundles). Then the two pairs are admissible provided $X_T$ and $Y_S$ are of expected dimensions.
	\item If $Q_B(S,T) \subset S \times_B T$, and $Q_B(X_T, S) \subset X_T \times_B S$ are divisors. Then the two pairs are admissible provided $X_T$ and $Y_S$ are of expected dimensions.
	\end{enumerate}
\end{lemma}

\begin{proof} The non-degeneracy condition of $f$ implies $\shH_X \to \PP(E^*)$ is flat. Hence we immediately get $S \to \PP(E^*)$ is faithful for the pair $(\shH_X, Y)$ provided $Y_{S}$ is of expected dimension. Also since $\shH_{X, S} \to S$ is flat, $\shH_{X, S}$ is a divisor in $X \times_B S$. 
On the other hand, if $Q_B(S,T) = S \times_B T$, then $X_{T} \times_B S \to X_T$ is flat, and $\shH_{X, S}$ is of expected dimension, therefore the base-change $X \to \PP(E)$ is faithful for the pair $(\shH_{S},T)$ provided $X_{T}$ is of expected dimension. If $Q_B(S,T) \subset S \times_B T$, and $Q_B(X_T, S) \subset X_T \times S$ are divisors, then $f$ is faithful for $Q_B(S,T)$. Then also $f$ is faithful for the pair $(\shH_{S},T)$ provided $X_{T}$ is of expected dimension. 
\end{proof}
	
\begin{theorem}[Relative HPD theorem for general sections] \label{thm:HPDgen:rel} If the above two relative HP-dual pairs $(X \to \PP(E), \,Y \to \PP(E^*))$, and $(S \to \PP(E^*), \,T \to \PP(E))$ satisfying $(\dagger')$ are \textbf{admissible} (Def. \ref{def:adm:rel}). Then
\begin{enumerate}[leftmargin = *]
\item We have $B$-linear semiorthogonal decompositions
	\begin{align}
	D(Y) & = \langle \shB^1 (2-N), \cdots, \shB^{N-2}(-1),\shB^{N-1} \rangle, \quad \shB^1 \subset \shB^2 \subset \cdots \subset\shB^{N-1} = \shA_0 \label{sod:Y:rel},\\
	D(T) & = \langle \shD^1(2-N), \cdots, \shD^{N-2}(-1), \shD^{N-1} \rangle, \quad \shD^1 \subset \shD^2 \subset \cdots \subset\shD^{N-2} = \shC_0  \label{sod:T:rel},
	\end{align} 
where $\shB^k$ and resp. $\shD^k$ are defined by the exactly same formula (\ref{def:B^k}) and resp. (\ref{def:D^k}).
\item There exit semiorthogonal decompositions for $X_T = X \times_{\PP(E)} T$ and $Y_S = Y \times_{\PP(E^*)} S$:
	\begin{align}
		D(X_T) &= \langle \sE_{X_T}, ~\shA_1(1)\boxtimes_B \shD^1, \shA_2(2) \boxtimes_B \shD^2, \ldots, \shA_{i-1}(i-1) \boxtimes_B \shD^{i-1} \rangle, \label{sod:X_T:rel} \\
		D(Y_S )&= \langle \shB^1 \boxtimes_B \shC_{1}^L(2-l) , \shB^2 \boxtimes_B \shC_{2}^L(3-l), \ldots, \shB^{l-1}\boxtimes_B \shC_{l-1}^L, ~\sE_{Y_S} \rangle, \label{sod:Y_S:rel}
	\end{align}
and an equivalence $\sE_{X_T} \simeq \sE_{Y_S}$. Here $\shC^L_k = S_{D(S)}(\shC_k) \otimes \sO_S(l) \simeq \shC_k$. If the decomposition for $D(S)$ is rectangular, then $\shC^L_k = \shC_k$.
\end{enumerate}
\end{theorem}

The Fourier-Mukai functors involved (for example $D(X_T) \hookrightarrow D(\shH)$, $D(Y_S) \hookrightarrow D(\shH)$, $\sE_{X_T} \simeq \sE_{Y_S}$, and etc) are explicitly given as in Rmk. \ref{rmk:FM} before. And as Rmk. \ref{rmk:FM}, the components $\shA_k(k)\boxtimes_B \shD^k = (\shA_k \boxtimes_B \shD^k)|_{X_T} \otimes \sO_{X_T}(k) \subset D(X_T)$ can be regarded as restricted from the ambient product $X\times_B T$ to the fiber product $X \times_{\PP(E)} T \subset X \times_B T $, i.e. are actually 'ambient parts'. Similarly the terms $\shB^k \boxtimes_B \shC^L_k(1+k-l) = (\shB^k \boxtimes_B \shC^L_k) |_{Y_S} \otimes \sO_{Y_S}(1+k - l)$ are also 'ambient contributions', i.e. can regarded as restrictions from ambient products to $Y\times_{\PP(E^*)} S \subset Y\times_B S$.

\begin{proof} As the absolute case, this can be shown in two major steps.

\medskip\noindent\textbf{Step 1. Base-change.}  From definition of HP-dual, we have decompositions
	\begin{align} 
	D(\shH_X) & = \langle \Phi_{\shE_X}(D(Y)),  \shA_1(1)\boxtimes_B D(\PP (E^*)), \cdots, \shA_{i-1}(i-1)\boxtimes_B D(\PP 	(E^*))\rangle.\label{sod:H_X:rel} \\
	D(\shH_S) &=  \langle \Phi_{\shE_S} (D(T)),~ D(\PP (E))\boxtimes_B  \shC_1(1), \ldots, D(\PP (E)) \boxtimes _B 			\shC_{l-1}(l-1) \rangle, \label{sod:H_S:rel} \\
	&=  \big\langle  \langle S_{\shH_S}( D(\PP(E))\boxtimes_B  \shC_k(k)\rangle_{k=1,\ldots, l-1}~,~\Phi_{\shE_S} (D(T)) \big		\rangle, \label{sod:H_S:Serre}
	\end{align}
As before we omit the restriction symbols $|_{\shH_{X}}$ and $|_{\shH_{S}}$ for simplicity of notations. Here $S_{\shH_{S}}$ is the Serre functor on the smooth variety $\shH_S$. Note $\shH_S \subset \PP(E) \times_B S$ is a divisor of  $\sO_{\PP(E) \times_B S}(1,1)$ on the smooth variety $\PP(E) \times_B S$. Denote $P_B := \PP(E)$, then $\omega_{P_B/B} = \sO_{P_B}(-N) \otimes a_{P_B}^* \det E^\vee$, $\omega_{P_B  \times_B S / S} = pr_1^* \omega_{P_B / B}$, and
	$$\omega_{\shH_{S}} =\omega_{\PP(E) \times_B S} |_{\shH_S} (1,1) = (\sO_{P_B} (1-N) \boxtimes_B \omega_S(1)) \otimes a_{\shH_S}^* \, \det E^\vee.$$
Therefore we have
	$$S_{\shH_S}( D(\PP E) \boxtimes_B  \shC_k(k) )= D(\PP(E)) \boxtimes_B S_{D(S)}(\shC_k)(k+1) = D(\PP(E))\boxtimes_B \shC^L_{k}(1+k -l), $$
where $\shC_{k}^L: =S_{D(S)}(\shC_k) \otimes \sO_S(l)$, and the term $a^*\,\det E^\vee$ is absorbed by $B$-linearity.

By \textbf{admissibility condition}, we can faithfully base-change (\ref{sod:H_S:rel}) along $q:S \to \PP(E)$,
	\begin{equation}\label{diagram:base_change_Y_S:rel}
	\begin{tikzcd}[back line/.style={dashed}, row sep=1.5 em, column sep=2.6 em]
	Q_B(X,Y)\ar{r}{} \ar{dd}	& \shH_X	\ar[back line]{dd}[swap]{}\\
	&& Q_B(X,Y_S) \ar[crossing over]{llu}[near start]{} \ar{r}[swap]{}	&\shH_{X,S} = \shH \ar{llu}{}	\ar{dd} \\
	Y \ar[back line]{r}{g}	& \PP(E^*)\ar[leftarrow, back line]{rrd}[near start]{q}	\\
				&&	Y_S \ar{llu}{} \ar{r}[swap]{} \ar[leftarrow, crossing over]{uu}	& S ,
	\end{tikzcd}
	\end{equation}
Then $\Phi_{\shE_X|S}:D(Y_S) \to D(\shH)$ is $S$-linear and fully faithful, where $\shE_X|S = q^* \shE_X \in D(Q_B(X,Y_S))$, and we have $S$-linear semiorthogonal decomposition:
	\begin{equation}\label{sod:H for X:rel}
	D(\shH) = \langle \Phi_{\shE_X|S}(D(Y_S)), \shA_1(1)\boxtimes_B D(S), \ldots, \shA_{i-1}(i-1)\boxtimes_B D(S) \rangle .
	\end{equation}
Similarly we can faithfully base-change (\ref{sod:H_X:rel}) and (\ref{sod:H_S:Serre}) along $f: X \to \PP(E)$,
	\begin{equation}\label{diagram:base_change_X_T:rel}
	\begin{tikzcd}[back line/.style={dashed}, row sep=1.5 em, column sep=2.6 em]
	Q_B(S,T)\ar{r}{} \ar{dd}	& \shH_S	\ar[back line]{dd}[swap]{}\\
	&& Q_B(S,X_T) \ar[crossing over]{llu}[near start]{} \ar{r}[swap]{}	&\shH_{S,X} = \shH \ar{llu}{}	\ar{dd} \\
	T \ar[back line]{r}{p}	& \PP(E)\ar[leftarrow, back line]{rrd}[near start]{f}	\\
				&&   X_T  \ar{llu}{} \ar{r}[swap]{} \ar[leftarrow, crossing over]{uu}	& X ,
	\end{tikzcd}
	\end{equation}
Then $\Phi_{\shE_S|X}:D(X_T) \to D(\shH)$ is $X$-linear and fully faithful, where $\shE_S|X = f^* \shE_S \in D(Q_B(S,X_T))$, and we have $X$-linear semiorthogonal decomposition:
	\begin{align} 
	D(\shH) &= \langle \Phi_{\shE_S|X}(D(X_T)), ~D(X) \boxtimes_B \shC_1(1) , \ldots, D(X)\boxtimes_B \shC_{l-1}(l-1) \rangle,\label{sod:H for S:rel}\\
	 &=  \langle D(X) \boxtimes_B \shC^L_1(2-l) , \ldots, D(X)\boxtimes_B \shC^L_{l-1}, ~  \Phi_{\shE_S|X}(D(X_T)) \rangle. \label{sod:H for S:Serre:rel} 
	\end{align}
These are exactly the relative versions of (\ref{sod:H for X}) and (\ref{sod:H for S}), (\ref{sod:H for S:Serre}). Relative version of (\ref{sod:H:k}) also holds. 

\medskip \noindent\textbf{Step 2. Game on the 'chessboard'.} Now we put everything into the common variety $D(\shH)$  again, and the situation is encoded in 'chessboard' of Figure \ref{Figure:l<=i} and \ref{Figure:l>=i} as before, with all blocks now corresponding to $\shA_{\alpha}(\alpha)\boxtimes_B \shC_{\beta}(\beta)$. Since the vanishing has exactly the same pattern, the game on the 'chessboard' can be played just same as before with all $R\Hom_{\shH}$'s (resp. $R\Hom_{X}$'s, $R\Hom_{S}$'s) being replaced by $a_{\shH*} \sHom$'s (resp. $a_{X*} R\sHom_{X}$'s, $a_{S*} R\sHom_{S}$'s).

For fully-faithful part, notice the computations of Lemma \ref{lem:region:pi_T} , \ref{lem:region:pi_S}, \ref{lem:region:pi_S:sE} and hence the statements are the same, with the resulting cones belonging to the same region (with $-\boxtimes -$ replaced by $- \boxtimes_B -$). Therefore all statements hold.  For generation part, we still show the right orthogonal of the images under $\pi_S$ is zero, and let $b$ belongs to this. Then again the components of $b$ can be shown to be zero in three steps. Notice all the semiorthogonal decompositions we use in \textbf{Step $1$}, \textbf{Step $2$} in absolute case hold in relative version, therefore same argument shows $b$ belongs to $\shR_1 \cup \shR_2$, hence $b =0$.
\end{proof}

\subsection{Relative HP-duality theorem for linear sections}
Apply the relative version of main theorem to linear section case, we obtain relative version of HP-duality theorem. Let $\sL \subset \sE^*$ to be (the locally free sheaf associated to) a vector sub-bundle of $E$ of rank $l$, $\sL^\perp \subset \sE$ the orthogonal sub-bundle. Denote $ L_B:= \PP(V(\sL)) \subset \PP(E^*)$ and $L_B^\perp:= \PP(V(\sL^\perp)) \subset \PP(E)$ the linear projective subbundles, where recall $V(\sL)$ is the vector bundle associated to the locally free sheaf $\sL$. Suppose $f: X \to \PP(E)$ is a $B$-morphism between smooth varieties, with $\sO_X(1):= f^* \sO_{\PP(E)}(1)$. For this section we always assume $X/B$ is proper and flat with integral geometric fibers, and $f$ is \textbf{non-degenerate} (Def. \ref{def:non-deg_rel}), which is equivalent to the natural map of $\sO_X$-modules
	\begin{equation}\label{eqn:relHPD:cond}
		\lambda: E^* \subseteq a_{X*} \sO_X(1) 
	\end{equation}
satisfies for every closed point $b \in B$, the induced map $E^*_b \to \Gamma(X_b, \sO_{X_b}(1))$ is an inclusion of vector spaces.

\begin{proposition}[Relative Orlov type theorem] \label{prop:orlov_rel}Suppose $f: X \to \PP(E)$ is a $B$-morphism between smooth varieties, and $L_B \subset \PP(E^*)$ a sub-linear projective bundle defined as above, such that $X_{L_B^\perp} : = X \times_{\PP(E)} L_B$ is of expected dimension over $B$, then there is a $X$-linear semiorthogonal decomposition
	$$D(\shH_{X,L_B}) = \langle j_* p^* D(X_{L_B^\perp}), \pi^* D(X) \otimes D(B)(1), \ldots , \pi^* D(X) \otimes D(B)(l-1) \rangle.$$
Here $\pi^* D(X) \otimes D(B)(k)$is the restriction of $D(X) \boxtimes_B D(B)(k)$ to $\shH_{X,L_B} \subset X \times_B L_B$ ,where $k =1, \ldots, l-1$ . 
 \end{proposition}
 
\begin{proof} 
The proof is just the relative version of R.Thomas' proof of \cite[Prop. 3.6]{RT15HPD}. The situation and the notations of maps are illustrated in the following commutative diagram:

\begin{equation*}
	\begin{tikzcd}
	X_{L^\perp_B} \times_B L_B \ar{d}[swap]{p} \ar[hook]{r}{j} & \shH_{X,L_B} \ar{d}{\pi} \ar[hook]{r}{\iota} & X \times_B L_B\ar{ld}[swap]{q} \ar{rd}\\
	X_{L^\perp_B} \ar[hook]{r}{i}         & X \ar{r}{a_X} & B  & L _B \subset \PP(E^*) \ar{l}[swap]{a_{L_B}}
	\end{tikzcd}	
\end{equation*}


\medskip\noindent\textbf{Fully faithfulness.} The fully faithfulness of the functors and semiorthogonality of the images follows from a direct computation: 

\begin{enumerate}[leftmargin=*, wide, align=left]
\item For any $F_1, F_2 \in D(X)$ and $G_1, G_2 \in D(B)$, $l-1 \ge k_1\ge k_2 \ge 1$, we have $a_* R\sHom_{\shH}(\pi^*\,F_1 \otimes G_1(k_1), \pi^*\, F_2 \otimes G_2(k_2))$ is a cone of 
	$$a_{X*} R\sHom(F_1,F_2(-1)) \otimes a_{L_B *} R\sHom(G_1(k_1), G_2(k_2 -1)) $$ 
and
	$$a_{X*} R\sHom(F_1,F_2) \otimes a_{L_B *} R\sHom(G_1(k_1), G_2(k_2)) $$ 
From the $B$-linear semiorthogonal decomposition for $L_B$, we see the later terms of both tensor product vanish if $l-1 \ge k_1> k_2 \ge 1$, thus $R\Hom(\pi^*\,F_1 \otimes G_1(k_1), \pi^*\, F_2 \otimes G_2(k_2))) = 0$. If $l-1 \ge k_1 = k_2 \ge 1$, then and the twisted term vanish, so we have $R\Hom_\shH (\pi^*\,F_1 \otimes G_1(k_1), \pi^*\, F_2 \otimes G_2(k_1)) =  R\Hom_{X \times_B L_B}(F_1 \boxtimes_B G_1(k_1), F_2 \otimes_B G_2(k_1))$. 

\item For $F \in D(X)$, $P \in D(B)$, $k \in [1,l-1]$, and $G \in D(X_{L^\perp} )$, then 
	\begin{align*}
	&\Hom(\pi^*\,F(k) \otimes P(k), j_*\,p^*\, G) = \Hom(\pi^*\, F, j_*p^* (G \otimes P^\vee (-k))) \\
	 &= \Hom(F, \pi_* j_* (p^* \, G \otimes P^\vee (-k))) = \Hom(F, i_* p_* (p^* \, G \otimes P^\vee (-k)) ) \\
	 &= \Hom(F, i_*(G \otimes p_* (P^\vee (-k))) ) = 0.
	\end{align*}
where $p_* (P^\vee (-k)) = a_{X_{L_B^\perp}}^* a_{L_B*} (P^\vee (-k)) = 0$ for $1-l \le -k \le 1$ follows from Serre vanishing for projective bundle $a_{L_B}: L_B \to B$. 

\item To show $R\Hom(j_*\,p^*\,F, j_*\,p^*\, G) = R\Hom(\pi^*\, F, \pi^*\, G) = R\Hom(F,G)$ for $F,G \in D(X_{L^\perp})$, the key is to compute the cone of the counit $j^*\,j_* \to \id $. Consider $X_{L_B^\perp} \times_B L_B \overset{j}{\hookrightarrow} \shH_{X,L_B} \overset{\iota}{\hookrightarrow} X \times_B L_B$, then the composition is cut out by a canonical section
	$$s \in \Gamma(X\times_B L_B, \sO_X(1) \boxtimes_B \sL^*) = \Gamma(B, \sHom (\sL, a_{X*} \sO_X(1))$$
corresponding to the inclusion $\sL \subseteq E^* \subseteq a_{X*} \, \sO_X(1)$ of $\sO_B$-modules over $B$ (recall we use the strange notation $L_B = \PP(V(\sL))$), where the first map is an inclusion of vector subbundle and the second map is (\ref{eqn:relHPD:cond}). Consider the relative Euler sequence on $L_B = \PP(V(\sL))$:
	$$ 0 \to \Omega_{L_B/B}(1) \to \sL^* \otimes \sO_{L_B} \to \sO_{L_B}(1) \to 0.$$
From the fact that $\iota$ is cut out by $\sO_X(1) \boxtimes_B \sO_{L_B}(1)$, we know $j$ is cut out by the section in $\sO_{X_{L^\perp_B}} (1) \boxtimes_B \Omega_{L_B/B}(1)$, hence the normal bundle of $j$ is given by
	$$N_j = \sO_{X_{L_B ^\perp}} (1) \boxtimes_B \Omega_{L_B/B}(1).$$ 
Therefore by standard argument on Fourier-Mukai kernels
, we have the counit $j^*\,j_* \to \id $ is an iterated extension of (cf. e.g \cite[Prop. 11.8]{Huy}) 
	$$\bigwedge^r N_j^* [r] \otimes (-) \quad \text{for} \quad r=1, \ldots, l-1.$$
Therefore from $R\, p_* \bigwedge^r N_j = 0$, for $r=1,2,\ldots, l-1$, which comes from the fact that $\{\Omega_{L_B/B}^{l-1}(l-1), \ldots, \Omega_{L_B/B}^1(1), \sO_{L_B}\}$ is a relative full exceptional collection on $L_B$, we have
	$$R\Hom(\cone(j^*\,j_* \to \id) p^*F, p^* G) = R\Hom(F, p_*(\cone(j^*\,j_* \to \id)^\vee \otimes G) = 0.$$
Hence $R\Hom(j_*\,p^*\,F, j_*\,p^*\, G) = R\Hom(\pi^*\, F, \pi^*\, G) = R\Hom(F,G)$ for $F,G \in D(X_{L^\perp})$.
\end{enumerate}

\medskip\noindent\textbf{Generation.} The generation part also can follow from an argument analogous to \cite[Prop. 3.6]{RT15HPD}. However, we would like to point out it can be checked fibrewisely (in the case when $X_{L_B^\perp}$ has fiberwisely expected dimension over $B$) and hence reduced to the absolute case because all the families in considerations are flat over $B$. 

In the following argument we assume $X$ is \emph{smooth} over $B$, and $X_{L_B^\perp}$ has fiberwisely expected dimension (by which we mean $\dim (X_b)_{L_b^\perp} = \dim X  - \dim B - l$ for every $b \in B$). This condition is at least satisfied by $X =\PP(E)$. Then the generation in general can follow from base-change the decomposition for the case $X=\PP(E)$ along $X \to \PP(E)$. 

For any $F \in D(\shH_{X, L_B})$ that is left right orthogonal to all the components in desired decomposition. Then consider $F|_b: = i^*_{\shH_{X_b,L_b}} \, F \in D(\shH_{X_b,L_b})$. Under the above assumption the fibers of family $X_{L_B^\perp}$ are all cut out by same family of linear system, then $X_{L_B^\perp} \to B$ is automatically flat of relative dimension $\dim X  - \dim B - l$. Therefore all squares in the following diagrams are exact cartesian:
\begin{equation*}
\begin{tikzcd}
	\shH_{X_b, L_b} \ar[hook]{r}{i_b} 	&\shH_{X,L_B} \\
	(X_{b})_{ L^\perp_b} \times L_b \ar[hook]{u}{j_b}\ar[hook]{r}{i_b}  \ar{d}{p_b}	&X_{L_B^\perp} \times_B L_B \ar[hook]{u}{j} \ar{d}{p} \\
	(X_{b})_{ L^\perp_b} \ar[hook]{r}{i_b} \ar{d}		& X_{L_B^\perp} \ar{d} \\
	b  \ar[hook]{r}{i_b}	& B	
\end{tikzcd}
\qquad  \qquad
\begin{tikzcd} 
	\shH_{X_b, L_b}	 \ar[hook]{r}{i_b}\ar{d}{\pi_b}	& \shH_{X,L_B} \ar{d}{\pi} \\
	X_b \ar[hook]{r}{i_b} \ar{d}		& X\ar{d} \\
	b	 \ar[hook]{r}{i_b}	& B,
\end{tikzcd}
\end{equation*}
we have $ i_{b*} \, j_{b*}\, p^*_b = j_* \, i_{b*} \,p^*_b = j_*\, p^*\, i_{b*}:  D((X_b)_{L_b^\perp})  \to D(\shH_{X,L_B})$, and $ i_{b*}\, \pi_b^* = \pi^* \, i_{b*}:  D(X_b) \to D(\shH_{L_B}) $. Tensoring $\sO_{L_B}(1)$ obviously commutes with base-change also. Hence 
	\begin{align*}
	&R\Hom(F|_b, j_{b*}\, p^*_b D((X_b)_{L_b^\perp}) ) =  R\Hom(F, i_{b*} \, j_{b*}\, p^*_b D((X_b)_{L_b^\perp}))  \\
	& = R\Hom(F,  j_*\, p^*\, (i_{b*} D((X_b)_{L_b^\perp}) ) = 0,
	\end{align*}
and
	\begin{align*} 
	& R\Hom(F|_b,\pi_b^* D(X_b )\otimes \sO_{L_b}(k) ) = R\Hom(F, i_{b*}\, (\pi_b^* D(X_b)\otimes \sO_{L_b}(k) )) \\
	& = R\Hom(F,  \pi^* \,( i_{b*} D(X_b) \otimes \sO_{L}(k) ) ) = 0,
	\end{align*}
where $k=1,2,\ldots, l-1$. Now by the generation part of absolute case \cite[Prop. 3.6]{RT15HPD} for $X_b$, we have $F|_b = 0$ for every $b \in B$. Therefore $F=0$.
\end{proof}

In particular, in the case of $X = \PP(E)$, we have
\begin{corollary}[Linear duality]\label{cor:linear-duality} $L^\perp_B \to \PP(E)$ is relative HP-dual to $L_B \to \PP(E^*)$ with respect to the Lefschetz decomposition $D(L_B) = \langle D(B), D(B)(1), \ldots, D(B)(l-1)\rangle$, with the $\PP(E)$-linear fully faithful functor $D(L^\perp_B) \to D(\shH_{L_B})$ given by the Fourier-Mukai kernel $\sO_{L^\perp_B \times_B L_B} \in D(L^\perp \times_{\PP(E)} \shH_{L_B})$.
\end{corollary}

\begin{remark} If $E = B \times V$ is a trivial bundle, where $V$ is a $\kk$-vector space of dimension $N$, we have $\PP(E) = B \times \PP(V)$. Then the corollary implies that the composition with projection $L^\perp_B \to B \times \PP(V) \to \PP(V)$ is HP-dual to $L_B \to B \times \PP(V^*) \to \PP(V^*)$ with respect to the above decomposition for $D(L_B)$.
\end{remark}

\begin{theorem}[Relative HP-duality theorem]\label{thm:HPD:rel} If the $B$-morphism $g: Y \to \PP(E^*)$ is the relative HP-dual to the $f: X \to \PP(E)$ with respect to the relative Lefschetz decomposition (\ref{lef:X:rel}). Then we have $B$-linear semiorthogonal decompositions of the form (\ref{sod:Y:rel}) of $D(Y)$. If $f$ is non-degenerate (Def. \ref{def:non-deg_rel}), and the fiber products
	\begin{align*} X_{L^\perp_B} : = X \times_{\PP(E)} L_B^\perp, \quad \text{and} \quad Y_{L_B}: = Y \times_{\PP(E^*)} L_B,
	\end{align*}
are of expected dimensions. Then there are semiorthogonal decompositions 
	\begin{align*}
		D(X_{L^\perp_B}) & = \langle \sC_{L_B}, \shA_l \otimes \sO_{\PP(E)}(l),  \ldots, \shA_{i-1} \otimes \sO_{\PP(E)}(i-1) \rangle, \\
		D(Y_{L_B}) & = \langle \shB^1 \otimes \sO_{\PP(E^*)}(2-l), \ldots, \shB^{l-2} \otimes \sO_{\PP(E^*)}(-1), \shB^{l-1}, \sC_{L_B} \rangle,
	\end{align*}
where the 'primitive' parts are both equivalent to a same triangulated category $\sC_{L_B}$.
\end{theorem}
\begin{proof}
%
%
From Lem. \ref{lem:admissible pairs:rel}, the admissible conditions of Thm. \ref{thm:HPDgen:rel} are satisfied precisely if $X_{L^\perp_B}$ and $Y_{L_B}$ are of expected dimensions. Them the theorem follows from applying Thm. \ref{thm:HPDgen:rel} to the pairs of relative HP-duals $f: X \to \PP(E)$, $g: Y \to \PP(E^*)$ and $S = L_B \to \PP(E^*)$, $T = L_B^\perp \to \PP(E)$.  \end{proof}

The following is the relative version of duality theorem.
\begin{theorem}[Relative HP-duality is a duality relation] Suppose $g: Y \to \PP(E^*)$ is relative HP-dual to $f: X \to \PP(E)$ over $B$ with respect to a Lefschetz decomposition (\ref{lef:X:rel}), where the embedding $D(Y) \to D(\shH_X)$ is given by a $\PP(E^*)$-linear Fourier-Mukai transform with kernel
	$$\shP   \in D(Y \times_{\PP(E^*)} \shH_X) = D(Q_B(X,Y)).$$
Then $f: X\to \PP(E)$ is HP-dual to $g:Y \to \PP(E^*)$ with respect to the Lefschetz decomposition 
	\begin{equation*}\label{lef:Y:dual:rel}
	D(Y) = \langle (\shB^{N-1})^\vee, (\shB^{N-2})^\vee (1), \ldots, (\shB^1)^\vee (N-2) \rangle,
	\end{equation*}
which is the dual of the decomposition (\ref{lef:Y}) obtained in Thm. \ref{thm:HPDgen:rel}, and the $\PP(E)$-linear embedding $D(X) \to D(\shH_Y)$ is given by the same kernel $\shP \in D(X \times_{\PP(E)} \shH_Y ) = D(Q_B(X,Y)).$
\end{theorem}

\begin{proof} The strategy is similar to Thm. \ref{thm:duality} before. First we claim $T= \PP(E)$ is relative HP-dual to $S =Q_B \to \PP(E^*)$ over $B$ with respect to the following Lefschetz decomposition of $Q_B \to \PP(E^*)$ (regard $Q_B$ as a $\PP^{N-2}$-bundle over $\PP(E)$):
	$$D(Q_B) = \langle D(\PP(E)), D(\PP(E))(1), \ldots, D(\PP(E))(N-2)\rangle.$$

The point is to notice this is again just linear duality: let $a_P: \PP(E) \to B$ be the natural projection, then $\PP(E) \times_B \PP(E^*) = \PP_{\PP(E)}(a_P^* \, E^*)$ is a projective bundle over $\PP(E)$, and then the relative universal quadric $Q_B = \PP_{\PP(E)} (T^*_{\PP(E)/B} (1) ) \subset \PP_{\PP(E)}(a_P^* \, E^*)$ is a linear projective subbundle, and $T  = \PP(E)  =  \PP_{\PP(E)} (\Tot(\sO_{\PP(E)} (-1)) ) \subset \PP(E) \times_B \PP(E) =  \PP_{\PP(E)}(a_P^* \, E) $ is the orthogonal projective linear bundle. Apply Cor \ref{cor:linear-duality} to the base $B_1 = \PP(E)$, the vector bundle $E_1 = a_P^* E$ and the subbundle $\sL_1 =T^*_{\PP(E)/B} (1) \subset E_1^*$, then $L_{B_1}^\perp = T = \PP(E)$ is relative HP-dual to $L_{B_1} = S = Q_B$ over $B_1 = \PP(E)$ (hence of course also over $B$).

Then we apply the relative theorem Thm. \ref{thm:HPDgen:rel} to the relative HP-dual pairs $X \to \PP(E)$, $Y \to \PP(E)$ and $T = \PP(E)$, $S  = Q_B \to \PP(E^*)$, the rest of the computation is exactly the same as Thm. \ref{thm:duality} with all $P$ (resp. $P^*$, $Q$, etc) replaced by $\PP(E)$ (resp. $\PP(E^*)$, $Q_B$, etc). The computation of kernels is guaranteed by admissibility condition for the pairs, which is always satisfied by above $(X,Y)$ and $(S,T)$. The computation of convolution can be performed the same way as before, since the condition of Lem. \ref{lem:rel_convolution}, i.e. the Tor-independence of the square
	$$
	\begin{tikzcd} Q_B(X,Y) \ar[hook]{r} \ar{d} & Q_B(X,\shH_Y) \ar[d] \\
		\shH_X \ar{r} & \shH_{X,Q}
	\end{tikzcd}
	$$
is satisfied again by considering the splitting (\ref{diagram:splitting}) in relative setting, where we notice $q',q''$ in this case is the base-change of $q: Q_B \to \PP(E^*)$ which is flat.
\end{proof}

\end{document}